\documentclass{amsart}
\usepackage{a4wide,amsmath,amsthm,amssymb,amsfonts,stmaryrd}
\usepackage{mathdots}
\usepackage{latexsym}
\usepackage[all]{xy}
\usepackage{CJK}
\usepackage{bm}
\usepackage{graphicx}
\usepackage{subfigure}
\usepackage{mathrsfs}
\usepackage{float}
\usepackage{makecell}
\usepackage{multirow}
\usepackage{xcolor}
\usepackage{cite}
\usepackage[pagebackref]{hyperref}
\usepackage{hyperref}
\usepackage{verbatim}
\usepackage{amsrefs}

\numberwithin{equation}{section}
\newtheorem{theorem}{\bf{Theorem}}[section]
\newtheorem{proposition}[theorem]{\bf{Proposition}}
\newtheorem{definition}[theorem]{\bf{Definition}}
\newtheorem{corollary}[theorem]{\bf{Corollary}}
\newtheorem{lemma}[theorem]{\bf{Lemma}}
\newtheorem{remark}[theorem]{\bf{Remark}}

\newtheorem{conjecture}[theorem]{\bf{Conjecture}}

\newcommand{\abs}[1]{\left|#1\right|}

\newcommand{\ceil}[1]{\lceil#1\rceil}
\newcommand{\floor}[1]{\lfloor#1\rfloor}
\newcommand{\car}[1]{\left|#1\right|}
\newcommand{\pairangone}[1]{\langle#1\rangle}
\newcommand{\pairang}[2]{\langle#1,#2\rangle}
\newcommand{\rest}{\lvert}
\newcommand{\Waff}[1]{W_{#1}^{\mathrm{aff}}}
\newcommand{\Wafftwist}[2]{\widetilde{W}_{#1,#2}^{\mathrm{aff}}}

\newcommand{\mrgl}{\operatorname{GL}}
\newcommand{\mrsl}{\operatorname{SL}}
\newcommand{\mrsp}{\operatorname{Sp}}

\newcommand{\mrm}{\operatorname{M}}
\newcommand{\mrn}{\operatorname{N}}

\newcommand{\mrtr}{\operatorname{tr}}
\newcommand{\mrdet}{\operatorname{det}}
\newcommand{\mrdim}{\operatorname{dim}}
\newcommand{\mrhom}{\operatorname{Hom}}

\newcommand{\mrdiag}{\operatorname{diag}}
\newcommand{\mrgcd}{\operatorname{gcd}}

\newcommand{\mrirr}{\operatorname{Irr}}
\newcommand{\mrrep}{\operatorname{Rep}}

\newcommand{\mrend}{\operatorname{End}}
\newcommand{\mraut}{\operatorname{Aut}}
\newcommand{\mrmod}{\operatorname{Mod}}

\newcommand{\mrInd}{\operatorname{Ind}}
\newcommand{\mrind}{\operatorname{ind}}
\newcommand{\mrres}{\operatorname{Res}}
\newcommand{\mrinf}{\operatorname{Inf}}

\newcommand{\mrker}{\operatorname{Ker}}

\newcommand{\id}{\operatorname{id}}

\newcommand{\mrsupp}{\operatorname{Supp}}

\newcommand{\mrst}{\mathrm{St}}

\newcommand{\mbz}{\mathbb{Z}}
\newcommand{\mbc}{\mathbb{C}}
\newcommand{\mbr}{\mathbb{R}}
\newcommand{\mbq}{\mathbb{Q}}

\newcommand{\mbg}{\mathbb{G}}

\newcommand{\mfp}{\mathfrak{p}}
\newcommand{\mfm}{\mathfrak{m}}
\newcommand{\mfn}{\mathfrak{n}}
\newcommand{\mfo}{\mathfrak{o}}
\newcommand{\mfa}{\mathfrak{a}}
\newcommand{\mfb}{\mathfrak{b}}
\newcommand{\mfc}{\mathfrak{c}}
\newcommand{\mfh}{\mathfrak{h}}

\newcommand{\mfs}{\mathfrak{s}}
\newcommand{\mfA}{\mathfrak{A}}
\newcommand{\mfJ}{\mathfrak{J}}
\newcommand{\mfH}{\mathfrak{H}}
\newcommand{\mfS}{\mathfrak{S}}

\newcommand{\mco}{\mathcal{O}}
\newcommand{\mcc}{\mathcal{C}}

\newcommand{\mca}{\mathcal{A}}

\newcommand{\mcf}{\mathcal{F}}
\newcommand{\mcg}{\mathcal{G}}
\newcommand{\mch}{\mathcal{H}}
\newcommand{\mct}{\mathcal{T}}

\newcommand{\mcm}{\mathcal{M}}
\newcommand{\mcn}{\mathcal{N}}
\newcommand{\mcp}{\mathcal{P}}
\newcommand{\mcu}{\mathcal{U}}

\newcommand{\msl}{\mathscr{L}}

\newcommand{\Lambdamin}{\Lambda_{\text{min}}}
\newcommand{\Lambdamax}{\Lambda_{\text{max}}}
\newcommand{\amin}{\mfa_{\text{min}}}
\newcommand{\amax}{\mfa_{\text{max}}}
\newcommand{\bmin}{\mfb_{\text{min}}}
\newcommand{\bmax}{\mfb_{\text{max}}}
\newcommand{\Jmin}{J_{\text{min}}}
\newcommand{\Jmax}{J_{\text{max}}}
\newcommand{\umin}{u_{\mathrm{min}}}
\newcommand{\umax}{u_{\mathrm{max}}}
\newcommand{\Jonemax}{J_{\text{max}}^{1}}
\newcommand{\Jonemin}{J_{\text{min}}^{1}}
\newcommand{\Honemax}{H_{\text{max}}^{1}}
\newcommand{\Honemin}{H_{\text{min}}^{1}}
\newcommand{\kappamax}{\kappa_{\mathrm{max}}}
\newcommand{\tildekappamax}{\wt{\kappa}_{\mathrm{max}}}

\newcommand{\bs}[1]{\boldsymbol{#1}}

\newcommand{\wt}[1]{\widetilde{#1}}

\begin{document}

	\title[Simple type theory for metaplectic covers of $\mrgl(r)$]{Simple type theory for metaplectic covers of $\mrgl(r)$ over a non-archimedean local field}
	
	\author{Jiandi Zou}
	\address{Institute for Advanced Study in Mathematics of Harbin Institute of Technology, Harbin, China}
	\email{idealzjd@gmail.com}
	
	\begin{abstract}
		
		Let $F$ be a non-archimedean locally compact field of residual characteristic $p$, let $G=\mrgl_{r}(F)$ and let $\wt{G}$ be an $n$-fold metaplectic cover of $G$ with $\mrgcd(n,p)=1$. We study the category $\mrrep_{\mfs}(\wt{G})$ of complex smooth representations of $\wt{G}$  having inertial equivalence class $\mfs=(\wt{M},\mco)$, which is a block of the category $\mrrep(\wt{G})$, following the ``type theoretical" strategy of Bushnell-Kutzko.
		
		Precisely, first we construct a ``maximal simple type" $(\wt{J_{M}},\wt{\lambda}_{M})$ of $\wt{M}$ as an $\mfs_{M}$-type, where $\mfs_{M}=(\wt{M},\mco)$ is the related cuspidal inertial equivalence class of $\wt{M}$. Along the way, we prove the folklore conjecture that every cuspidal representation of $\wt{M}$ could be constructed explicitly by a compact induction. Secondly, we construct ``simple types" $(\wt{J},\wt{\lambda})$ of $\wt{G}$, and prove that each of them is an $\mfs$-type of a certain block $\mrrep_{\mfs}(\wt{G})$. When $\wt{G}$ is either a Kazhdan-Patterson cover or Savin's cover, the corresponding blocks turn out to be those containing discrete series representations of $\wt{G}$. Finally, for a simple type $(\wt{J},\wt{\lambda})$ of $\wt{G}$ we describe the related Hecke algebra $\mch(\wt{G},\wt{\lambda})$, which turns out to be not far from an affine Hecke algebra of type A, and is exactly so if $\wt{G}$ is one of the two special covers mentioned above. 
		
		We leave the construction of a ``semi-simple type" related to a general block $\mrrep_{\mfs}(\wt{G})$ to a future phase of the work.
		
	\end{abstract}
	
	\maketitle
	\tableofcontents

	\section{Introduction}
	
	\subsection{Background}
	
	A basic but important question in the representation theory of $p$-adic groups is to study the category of complex smooth representations $\mrrep(G)$, where $G$ is a reductive group over a non-archimedean locally compact field $F$ of residual characteristic $p$. In \cite{bernstein1984centre}, Bernstein developed a block decomposition 
	$$\mrrep(G)=\prod_{\mfs}\mrrep_{\mfs}(G),$$ 
	where each $\mfs$ in the product consists of (the $G$-conjugacy class of) a Levi subgroup $M$ of $G$ and an orbit $\mco$ of a cuspidal representation $\pi_{0}$ of $M$ twisted by unramified characters of $M$, and $\mrrep_{\mfs}(G)$ is the subcategory of $\mrrep(G)$ consisting of smooth representations having the inertial cuspidal support $\mfs$. 
	
	The so-called ``type theory", justifying the title of this paper, dates back to the seminal work of Bushnell-Kutzko \cite{bushnell1998smooth}, which depicts a general strategy to describe $\mrrep_{\mfs}(G)$ for an inertial equivalence class $\mfs=(M,\mco)$. A type of $\mrrep_{\mfs}(G)$, in their sense, consists of an open compact subgroup $J$ of $G$ and an irreducible representation $\lambda$ of $J$, such that the restriction of any representation $\pi\in\mrrep_{\mfs}(G)$ to $J$ contains $\lambda$ (\emph{cf.} \S \ref{subsectionblockstypes} for a precise definition). Once a type $(J,\lambda)$ of $\mrrep_{\mfs}(G)$ exists, an equivalence of categories could be defined
	$$\bs{\mathrm{M}}_{\lambda}:\mrrep_{\mfs}(G)\rightarrow\mrmod(\mch(G,\lambda)),\quad\pi\mapsto\mrhom_{G}(\mrind_{J}^{G}\lambda,\pi),$$
	where $\mrind_{J}^{G}\lambda$ denotes the compact induction of $\lambda$ as a representation of $G$, and $\mrmod(\mch(G,\lambda))$ denotes the category of non-degenerate modules of the Hecke algebra $\mch(G,\lambda)$. More precisely, they proposed a general strategy, which is divided into two following questions:
	\begin{enumerate}
		\item Construct a type $(J_{M},\lambda_{M})$ of $\mrrep_{\mfs_{M}}(M)$, where $\mfs_{M}$ denotes the corresponding inertial equivalence class of $M$ consisting of the same data of $\mfs$.
		\item Construct a type $(J,\lambda)$ of $\mrrep_{\mfs}(G)$ based on $(J_{M},\lambda_{M})$. Usually,   $(J,\lambda)$ could be chosen as a so-called ``covering pair" of $(J_{M},\lambda_{M})$.
	\end{enumerate}
	We remark that the question (1), albeit innocent-looking, relates to the extremely difficult folklore conjecture that every cuspidal representation of $M$ could be constructed from the compact induction of an irreducible representation of an open compact modulo center subgroup. Once being constructed in a certain way, the uniqueness of $(J_{M},\lambda_{M})$ up to $M$-conjugacy is also quite intriguing and worth pursuing (for instance see  \cite{bushnell129admissible}*{Theorem 6.2.4}). Finally, the subsequent question is obviously
	\begin{enumerate}\setcounter{enumi}{2}
		\item Describe $\mch(G,\lambda)$ and $\mrmod(\mch(G,\lambda))$. 
	\end{enumerate}
	The rough expectation is that each $\mch(G,\lambda)$ should not be far from a certain affine Hecke algebra. We leave \S \ref{sectionHeckecoveringpair} for the missing details in this paragraph.
	
	We are trying to give a historical summary, which, due to the ignorance of the author, could be quite partial and by no means complete. Beside some sporadic results for low rank groups $G$, perhaps the first foundational result is due to Howe \cite{howe1977tamely}, where he gave the full answer to the question (1) in the case $G=\mrgl_{r}(F)$ and $p$ greater that $r$. Under the same condition, questions (2) and (3) were also subsequently answered, see \cite{howe1990hecke}, \cite{waldspurger1986algebres}. After that, two main streamlines occurred: the first one dealt with special groups $G$ and all the blocks $\mrrep_{\mfs}(G)$, while the second one dealt with rather general groups and blocks $\mrrep_{\mfs}(G)$ with certain restrictions.
	
	We describe the first streamline. A milestone is the work of Bushnell-Kutzko \cite{bushnell129admissible}, \cite{bushnell1998smooth}, \cite{bushnell1999semisimple}, where they completely answered the three questions above for $G=\mrgl_{r}(F)$. They gave an explicit construction for a cuspidal representation of any Levi subgroup of $G$, influenced by and generalizing the work of Howe \cite{howe1977tamely} and Carayol \cite{carayol1984representations}. 
	In particular, the corresponding Hecke algebra $\mch(G,\lambda)$ is indeed an affine Hecke algebra of type A, which provides another aspect of the classification result of Bernstein and Zelevinsky  \cite{bernstein1977induced}, \cite{zelevinsky1980induced} for irreducible representations of $\mrgl_{r}(F)$. After that, various results were built up for other groups, including
	\begin{itemize}
		\item When $G$ is an inner form of $\mrgl_{r}(F)$, the three questions above were resolved in a sequence of articles \cite{secherre2004representations}, \cite{secherre2005representations}, \cite{secherre2005types}, \cite{secherre2008representations}, \cite{secherre2012smooth} of S\'echerre and Stevens, influenced by the earlier work of Broussous, Grabitz, Silberger and Zink \cite{broussous1999minimal}, \cite{broussous2000pure}, \cite{grabitz1999continuation}, \cite{grabitz2001level}, etc.
		
		\item For $G=\mrsl_{r}(F)$, the three questions were resolved by Bushnell-Kutzko \cite{bushnell1993admissible}, \cite{bushnell1994admissible} and Goldberg-Roche \cite{goldberg2002types}, \cite{goldberg2005hecke}. For an inner form $G$ of $\mrsl_{r}(F)$, although questions (1) and (2) remain to be resolved, Aubert-Baum-Plyman-Solleveld \cite{aubert2017hecke} somehow answered question (3) for the Hecke algebra related to a product of several blocks $\mrrep_{\mfs}(G)$.
		
		\item Assume $p\neq 2$. For a classical group $G$, or more precisely, the  fixed-point subgroup $G$ of $\mrgl(V)$ defined by an involution related to a sesquilinear form of the $r$-dimensional vector space $V$ over $F$, the first question was resolved by Stevens  \cite{stevens2001intertwining}, \cite{stevens2002semisimple}, \cite{stevens2005semisimple}, \cite{stevens2008supercuspidal}. Notably, the uniqueness of $(J_{M},\lambda_{M})$ up to $M$-conjugacy was shown in \cite{kurinczuk2021endo} recently. The second question was resolved by Miyauchi-Stevens \cite{miyauchi2014semisimple}. They also partially explored the third question when $M$ is a maximal Levi subgroup of $G$, but the general case is still pending.
		
	\end{itemize}
	All these works essentially rely on the original work of Bushnell-Kutzko for $G=\mrgl_{r}(F)$.
	
	We describe the second streamline. Let $G$ be a general reductive group over $F$. First of all, we sum up the following results for certain rather easily described blocks:
	\begin{itemize}
		\item Borel \cite{borel1976admissible} and Casselman \cite{casselman1980unramified} studied the smooth representations having Iwahori fixed vectors, and related them to the modules of the Iwahori Hecke algebra. Indeed, the study of this case dates back to the earlier work of Iwahori-Matsumoto \cite{iwahori1965some}. Then at least when $G$ is split, in which case the above representations are exactly in the unramified principal block, the above three questions were fully answered.
		
		\item If $G$ is split, Roche \cite{roche1998types} resolved the three questions for general principal blocks under some mild condition on the residual characteristic $p$. 
		
		\item If the corresponding block $\mrrep_{\mfs}(G)$ is of depth 0, saying that the corresponding cuspidal representation $\pi_{0}$ of $M$ is of depth $0$, Moy-Prasad \cite{moy1994unrefined}, \cite{moy1996jacquet} as well as Morris \cite{morris1999level} resolved the first two questions. The third question was resolved by Morris \cite{morris1993tamely} based on the previous work of Howlett-Lehrer \cite{howlett1980induced}.
		
	\end{itemize}
	Now we assume that $G$ is split over a tamely ramified extension of $F$. Motivated and influenced by the work of Howe \cite{howe1977tamely}, Moy-Prasad \cite{morris1999level}, Adler \cite{adler1998refined} and others, Yu \cite{yu2001construction} gave an explicit construction of the 
	so-called ``tame supercuspidal representations"\footnote{We use interchangably the terminologies ``cuspidal" and ``supercuspidal", since for complex representations they are the same.} via compact induction, which in particular resolved the first question for blocks with respect to a tame supercuspidal representation $\pi_{0}$.
	Then Kim-Yu \cite{kim2017construction} resolved the second question for such blocks based on Yu's construction of tame supercuspidal representations, and Hakim-Murnaghan \cite{hakim2008distinguished} fully explored the ``uniqueness" of Yu's construction.
	
	Unlike the special cases mentioned above, Yu's construction is general enough to recover ``almost all" the supercuspidal representations. Indeed, when $F$ is of characteristic 0 and $p$ is large enough, Kim \cite{kim2017construction} proved that Yu's construction exhausts all the supercuspidal representations. Recently, Fintzen \cite{fintzen2021types} showed that for any $F$ with $p$ not dividing the order of the Weyl group of $G$, which in particular recovers Kim's assumption, Yu's construction is exhaustive. It is also worth mentioning that Kaletha \cite{kaletha2019regular} constructed the so-called ``regular supercuspidal representations" as a subclass of tame supercuspidal representations, having simpler constructing datum and easier described Langlands parameter.
	
	Finally, to study question (3) for those blocks with $\pi_{0}$ tame supercuspidal, Yu proposed a conjecture to transfer the original Hecke algebra $\mch(G,\lambda)$ to another Hecke algebra related to a depth-0 block of a twisted Levi subgroup of $G$ (\emph{cf.} \cite{yu2001construction}*{Conjecture 0.2} and \cite{adler2021regular}*{Conjecture 1.1}). This conjecture was recently proved by Adler-Mishra \cite{adler2021regular} for regular supercuspidal blocks and Ohara \cite{ohara2021hecke} in general.
	
	Let us also briefly mention that the type theory is not the only possible way to study the category $\mrrep_{\mfs}(G)$. Indeed, Bernstein himself constructed a progenerator $\Pi_{M}$ of the block $\mrrep_{\mfs_{M}}(M)$, as well as its parabolic induction $\Pi_{G}=i_{P}^{G}(\Pi_{M})$ as a progenerator of the block $\mrrep_{\mfs}(G)$, such that the map
	$$\mrrep_{\mfs}(G)\rightarrow\mrmod(\mch(\Pi_{G})),\quad\pi\mapsto\mrhom_{G}(\Pi_{G},\pi)$$
	is an equivalence of categories, where $\mch(\Pi_{G})=\mrend_{G}(\Pi_{G})$ denotes the corresponding endomorphism algebra (\emph{cf.} \cite{renard2010representations}*{Chapitre VI}). Using harmonic analysis over $p$-adic groups, Heiermann  \cite{heiermann2011operateurs} constructed the generators of $\mch(\Pi_{G})$ and calculated the corresponding relations when $G$ is a classical group, which was generalized by Solleveld \cite{solleveld2020endomorphism} to general reductive groups based on some conjectural assumptions. It is shown that $\mch(\Pi_{G})$ is the semi-direct product of an affine Hecke algebra and a twisted finite group algebra, which parallels the question (3) we mentioned above in the type theory.
	
	Now we introduce our main player: the finite central cover of a $p$-adic reductive group. Fix a positive integer $n$. Assume that $F^{\times}$ contains the subgroup of $n$-th roots of unity $\mu_{n}$ of cardinality $n$. Let $\wt{G}$ be an $n$-fold cover of $G$, which is a central extension of $G$ by $\mu_{n}$ as $\ell$-groups, that is,
	$$\xymatrix{0 \ar[r] & \mu_{n} \ar[r]^-{} & \wt{G} \ar[r]^-{\bs{p}} & G \ar[r] & 0}.$$ 
	To proceed a general discussion, usually we need to restrict to some special $n$-fold covers. For instance, those covers constructed by Brylinski-Deligne \cite{brylinski2001central} are general enough to include most interesting covers, and concrete enough to make various specific calculations. So from now on we also assume $\wt{G}$ to be a so-called $n$-fold metaplectic cover. In this article, by ``$n$-fold metaplectic cover" we mean an $n$-fold cover of $G$ arising from Brylinski-Deligne's construction (see Section \ref{sectionnfoldcover}).  Also, in the discussion below we assume $\wt{G}$ to be tame, saying that $n$ is relatively prime to $p$.
	
	Like the previous case, the Bernstein decomposition is still valid (\emph{cf.} \cite{kaplan2022note})
	$$\mrrep(\wt{G})=\prod_{\mfs}\mrrep_{\mfs}(\wt{G}),$$
	where the inertial equivalence class $\mfs=(\wt{M},\mco)$ consists of (the $G$-conjugacy class of) a Levi subgroup $\wt{M}$ of $\wt{G}$ and an orbit $\mco$ of a cuspidal representation $\wt{\pi}_{0}$ of $\wt{M}$ twisted by unramified characters of $M$ \footnote{By convention, for any closed subgroup $H$ of $G$, we let $\wt{H}$ be its preimage by the projection $\bs{p}$ in $\wt{G}$.}. Still, one wonders if the three questions above could be dealt with for the category $\mrrep_{\mfs}(\wt{G})$, and in particular, a type $(\wt{J},\wt{\lambda})$ of $\mrrep_{\mfs}(\wt{G})$ is expected to be constructed. Here, essentially we only need to consider genuine representations, saying that the action of $\mu_{n}$ on $\wt{\pi}_{0}$ is given by a fixed character $\epsilon$ of order $n$. 
	
	Unlike the linear case, only few cases for covers were explored, including:
	\begin{itemize}
		\item When $G$ is split, Savin \cite{savin2004unramified} resolved the three questions above for the genuine unramified principal block of $\wt{G}$, or in other words, the block of genuine representations having Iwahori fixed vectors. This in particular generalizes the results of Borel and Casselman. Also, the corresponding genuine Iwahori Hecke algebra is isomorphic to the Iwahori Hecke algebra of a linear reductive group, which leads to the so-called ``Shimura correspondence" for representations.
		
		\item When the corresponding block $\mrrep_{\mfs}(\wt{G})$ is of depth 0, Howard-Weissman \cite{howard2009depth} generalized the argument of Moy-Prasad to resolve questions (1) and (2).
		
		\item When $\wt{G}$ is a cover of a torus, the three questions become somehow transparent. They were studied by Weissman \cite{weissman2016covers}.
		
		\item When $\wt{G}$ is a tame Kazhdan-Patterson cover of $G=\mrgl_r(F)$, Suzuki \cite{suzuki2004metaplectic} constructed maximal simple types of $\wt{G}$ and sketched a solution to the question (1) based on the work of Bushnell-Kutzko \cite{bushnell129admissible}.
		
		\item For sporadic cases, partial results for question (1) were obtained, including Blondel \cite{blondel1985representations}, \cite{blondel1992uniqueness} for $G=\mrgl_{r}(F)$ and Ngo \cite{van2017beta} for $G=\mathrm{SO}_{r}(F)$ and $\wt{G}=\mathrm{Spin}_{r}(F)$.
		
	\end{itemize}
	
	So one is curious to explore the three questions above for general covers $\wt{G}$ and general blocks $\mrrep_{\mfs}(\wt{G})$. 
	
	\subsection{Main results and proofs}
	
	The main goal of the paper is to partially answer the three questions for a tame $n$-fold metaplectic cover $\wt{G}$ of $G=G_{r}:=\mrgl_{r}(F)$. So, let $\mfs=(\wt{M},\mco)$ be an inertial equivalence class of $\wt{G}$ with $\wt{\pi}_{0}\in\mco$ being a genuine cuspidal representation of $\wt{M}$. Let $\mfs_{M}=(\wt{M},\mco)$ be the corresponding inertial equivalence class of $\wt{M}$.
	
	We first give the answer to the first question, which could be summed up as the following theorem.
	
	\begin{theorem}[\emph{cf.} Theorem \ref{thmlambdaMtype}, Theorem \ref{thmsimpletypeconjugacy}, Theorem \ref{thmcuspidalconstructionLevi}]\label{thmmainQone}
		
		We may construct a maximal simple type (\emph{cf.} Definition \ref{defmaxsimtypeM}) $(\wt{J_{M}},\wt{\lambda}_{M})$ of $\wt{M}$, where $J_{M}$ is an open compact subgroup of $M$ and $\wt{\lambda}_{M}$ is a genuine irreducible representation of $\wt{J}_{M}$, such that $\wt{\pi}_{0}$ contains $\wt{\lambda}_{M}$. Moreover, $(\wt{J_{M}},\wt{\lambda}_{M})$ is an $\mfs_{M}$-type. Such $(\wt{J_{M}},\wt{\lambda}_{M})$ is unique up to $M$-conjugacy. 
		
	\end{theorem}
	
	Write  $M=G_{r_{1}}\times\dots\times G_{r_{k}}$ with $r_{1}+\dots+r_{k}=r$. For each $i$, we regard $G_{r_{i}}$ as a subgroup of $M$ and $\wt{G_{r_{i}}}$ as a subgroup of $\wt{M}$ via the pull-back.
	Then there exists a maximal simple type $(\wt{J}_{i},\wt{\lambda}_{i})$ of $\wt{G_{r_{i}}}$ such that $\wt{\lambda}$ equals the tensor product $\wt{\lambda}_{1}\boxtimes\dots\boxtimes\wt{\lambda}_{k}$ (See \S \ref{subsectionmetacoverGL} for the meaning of the tensor product). So essentially we only need to explain the construction of a maximal simple type of $\wt{G}$. Then Theorem \ref{thmmainQone} could be easily reduced to the case where $M=G$. 
	
	More generally, we explain the construction of a simple type $(\wt{J},\wt{\lambda})$ of $\wt{G}$.  Our construction is largely based on the theory of Bushnell-Kutzko for simple strata, simple characters, etc. To proceed, we assume that the readers are familiar with the corresponding concepts and leave Section \ref{sectionstrata} for more details. 
	
	Let $V$ be an $r$-dimensional vector space over $F$. Up to choosing a basis of $V$, we identify $G$ with $\mraut_{F}(V)$ and $A=\mrm_{r}(F)$ with $\mrend_{F}(V)$. Let $[\mfa,u,0,\beta]$ be a strict simple stratum in $A$. Here, $E=F[\beta]$ is a subfield of $A$ of degree $d$ over $F$, and $\mfa$, as a hereditary order in $A$, is normalized by $E^{\times}$. Let $B$ be the centralizer of $E$ in $A$, then $\mfb=\mfa\cap B$ is a hereditary order in $B$. 
	
	We consider the related subgroups $H^{1}=H^{1}(\beta,\mfa)$, $J^{1}=J^{1}(\beta,\mfa)$ and $J=J(\beta,\mfa)$ of $G$, where the first two groups are pro-$p$ open compact and the third one is open compact, such that $$J/J^{1}\cong U(\mfb)/U^{1}(\mfb)\cong\mcm=\mrgl_{m_{1}}(\bs{l})\times\dots\times\mrgl_{m_{t}}(\bs{l}),$$ where $m=r/d=m_{1}+\dots+m_{t}$, and $\bs{l}$ denotes the residue field of $E$, and $\mcm$ is a Levi subgroup of $\mrgl_{m}(\bs{l})$. We may find a maximal open compact subgroup $K$ of $G$ that contains $J$, and since $\wt{G}$ is tame, we may find a splitting (i.e. a section and a group homomorphism)  $\bs{s}:K\rightarrow \wt{G}$.
	
	We consider a simple character $\theta$ of $H^{1}$, the Heisenberg representation $\eta$ of $\theta$ of $J^{1}$, and a $\beta$-extension $\kappa$ of $J$ that extends $\eta$. We let $\wt{\kappa}$ be the pull-back of $\kappa$ as a non-genuine irreducible representation of $\wt{J}$. On the other hand, we let $\varrho$ be a cuspidal representation of $\mcm$. After inflating to $\bs{s}(J)$ and then extending to $\wt{J}$ by $\epsilon$, we get a genuine irreducible representation $\wt{\rho}$ of $\wt{J}$. We let $\wt{\lambda}=\wt{\kappa}\otimes\wt{\rho}$, then $(\wt{J},\wt{\lambda})$ is a \emph{homogeneous type} of $\wt{G}$. 
	
	\begin{remark}
		
		As already been mentioned by Suzuki \cite{suzuki2004metaplectic} for maximal simple types,  $\wt{\lambda}$ is indeed the ``genuine pull-back" of a homogeneous type $\lambda$ of $G$. More precisely, let $\rho$ be the inflation of $\varrho$ to $J$ and let $\lambda=\kappa\otimes\rho$ be a related homogeneous type of $G$. Then $\wt{\lambda}$ is obtained from extending the representation $\lambda\circ\bs{p}$ of $\bs{s}(J)$ to $\wt{J}$ by $\epsilon$.
		
	\end{remark}
	
	Moreover, $(\wt{J},\wt{\lambda})$ is a \emph{twisted simple type} of $\wt{G}$ if (\emph{cf.} Section \ref{sectionsimpletypes} for more details)
	\begin{itemize}
		\item $m_{1}=\dots=m_{t}=m_{0}$ for a positive integer $m_{0}$;
		\item There exists a cuspidal representation $\varrho_{0}$ of $\mrgl_{m_{0}}(\bs{l})$, such that $\varrho$ is isomorphic to $(\varrho_{0}\boxtimes\dots\boxtimes\varrho_{0})\chi_{g_{0}}$ as a representation of $\mcm=\mrgl_{m_{0}}(\bs{l})\times\dots\times\mrgl_{m_{0}}(\bs{l})$, where $g_{0}$ is a diagonal element in  $U(\mfb)$ and $\chi_{g_{0}}:=\epsilon([g_{0},\cdot]_{\sim})$ is a character of $\mcm$.
	\end{itemize}
	In the above definition, if moreover $g_{0}=1$ and $\chi_{g_{0}}$ is the identity character, then we call the related pair $(\wt{J},\wt{\lambda})$ a \emph{simple type} of $\wt{G}$. Finally, a simple type $(\wt{J},\wt{\lambda})$ is \emph{maximal} if $t=1$ and $m_{0}=m$, or in other words, $\mfb$ is a maximal hereditary order in $B$. We remark that the intertwining set $I_{G}(\wt{\lambda})$ of a simple type $(\wt{J},\wt{\lambda})$ could be calculated explicitly, this further enables us to study the related Hecke algebra $\mch(\wt{G},\wt{\lambda})$.
	
	Now we may explain the proof of Theorem \ref{thmmainQone} for $M=G$. Let $(\wt{J},\wt{\lambda})$ be a maximal simple type of $\wt{G}$. First we show that it is a type of a cuspidal inertial equivalence class of $\wt{G}$. Indeed, we consider $\bs{J}=E^{\times}J$, which is an open compact modulo center subgroup of $G$. Moreover, the normalizer $N_{G}(\wt{\lambda})$ and intertwining set $I_{G}(\wt{\lambda})$ turn out to be equal, which is an open subgroup of $\bs{J}$ of finite index and denoted by $J_{\lambda}$. Thus, we may take an extension of $\wt{\lambda}$ to an irreducible representation $\wt{\lambda'}$ of $\wt{J_{\lambda}}$, and then take the compact induction $\wt{\bs{\lambda}}=\mrind_{\wt{J_{\lambda}}}^{\wt{\bs{J}}}\wt{\lambda'}$ as a genuine irreducible representation of $\wt{\bs{J}}$, called an \emph{extended maximal simple type} of $\wt{G}$. Since the normalizer of $\wt{\bs{\lambda}}$ in $G$ is $\wt{\bs{J}}$, the compact induction $\wt{\pi}_{0}=\mrind_{\wt{\bs{J}}}^{\wt{G}}\wt{\bs{\lambda}}$ is a genuine irreducible cuspidal representation of $\wt{G}$. This shows that $(\wt{J},\wt{\lambda})$ is a type of the inertial equivalence class related to $\wt{\pi}_{0}$. 
	
	On the other hand, given a cuspidal representation $\wt{\pi}_{0}$ of $\wt{G}$, we want to find a maximal simple type $(\wt{J},\wt{\lambda})$ contained in $\wt{\pi}_{0}$. This could be done in two steps, first we find a simple character $\theta$ contained in $\wt{\pi}_{0}$, and then we find a corresponding maximal simple type $\wt{\lambda}$ contained in $\wt{\pi}_{0}$. The first step is essentially a repetition of a similar argument of \cite{bushnell129admissible}*{Chapter 8} or \cite{secherre2008representations}*{Section 3 and 4}, due to the following obvious fact: if $H$ is an open compact pro-$p$-subgroup of $G$, then there exists a unique splitting $\bs{s}:H\rightarrow \wt{G}$; thus for a representation $\xi$ of $H$, the group $\,_{s}H:=\bs{s}(H)$ as well as the representation $\,_{s}\xi:=\xi\circ\bs{p}$ of $\,_{s}H$ are well-defined. Since the argument in the first step concerns only pro-$p$-subgroups and their representations, the corresponding argument could be simply transplanted. The second step reduces to studying depth 0 representations, which essentially follows from a similar argument in \cite{secherre2008representations}*{Section 5}. So Theorem \ref{thmmainQone} is explained.
	
	Now we answer the second and third questions for an inertial equivalence class related to a simple type $(\wt{J},\wt{\lambda})$ of $\wt{G}$. Keep the notation as before and let $r_{0}=dm_{0}$. From the construction, we may find a decomposition $V=\bigoplus_{i=1}^{t}V^{i}$ of vector spaces over both $F$ and $E$, where $V^{i}$ is a vector space of dimension $r_{0}$ over $F$ and $m_{0}$ over $E$. Moreover, we assume that this decomposition is compatible with the hereditary orders $\mfa$ and $\mfb$, saying that the lattice chains in defining $\mfa$ and $\mfb$ are decomposable with respect to this decomposition. We also notice that in this case $\mfb$ is a hereditary order in $\mrm_{m}(\mfo_{E})$ related to the composition $(m_{0},\dots,m_{0})$ of $m$. Consider the Levi subgroup $M=\mraut_{F}(V^{1})\times\dots\times\mraut_{F}(V^{t})$ of $G$ and a related parabolic subgroup $P=MU$. Let $U^{-}$ be the opposite of $U$ as a unipotent subgroup of $G$. Then we are able to construct
	\begin{itemize}
		\item an open compact subgroup $J_{P}=(H^{1}\cap U^{-})(J\cap P)$ of $G$ and an irreducible representation $\wt{\lambda}_{P}$ of $\wt{J_{P}}$, such that $\wt{\lambda}=\mrind_{\wt{J_{P}}}^{\wt{J}}\wt{\lambda}_{P}$.
		\item an open compact subgroup $J_{M}=J_{P}\cap M=J\cap M$ of $M$ and an irreducible representation $\wt{\lambda}_{M}$ of $\wt{J_{M}}$ as the restriction of $\wt{\lambda}_{P}$.
	\end{itemize}
	Then the answer to the second question is as follows:
	\begin{theorem}[\emph{cf.} Theorem \ref{thmlambdaMtype}, Theorem \ref{thmJPcoveringpair}]\label{thmmainQtwo}
		
		$(\wt{J_{M}},\wt{\lambda}_{M})$
		is a maximal simple type of $\wt{M}$, which is a type of a cuspidal inertial equivalence class $\mfs_{M}=(\wt{M},\mco)$. $(\wt{J_{P}},\wt{\lambda}_{P})$ is a covering pair of $(\wt{J_{M}},\wt{\lambda}_{M})$, so both $(\wt{J_{P}},\wt{\lambda}_{P})$ and $(\wt{J},\wt{\lambda})$ are $\mfs$-types with $\mfs=(\wt{M},\mco)$ an inertial equivalence class of $\wt{G}$.
		
	\end{theorem}
	
	The crux of the proof requires the calculation of the Hecke algebra $\mch(\wt{G},\wt{\lambda})\cong\mch(\wt{G},\wt{\lambda}_{P})$, which is closely related to the third question. Let $q_{E}$ be the cardinality of the residue field $\bs{l}$ of $E$ and let $\bs{q}_{0}=q_{E}^{m_{0}}$. We define a certain positive integer $s_{0}$ dividing $n$ (\emph{cf.} \S \ref{subsectionsimpletypes}) and the affine Hecke algebra $\wt{\mch}(t,s_{0},\bs{q}_{0})$ of type A (\emph{cf.} \S \ref{subsectionaffineHeckeA}).
	
	\begin{theorem}[\emph{cf.} Theorem \ref{thmcalhecke}]\label{thmmainQthree}
		
		Up to a scalar, there exists a canonical embedding of algebras
		$$\Psi:\wt{\mch}(t,s_{0},\bs{q}_{0})\hookrightarrow\mch(\wt{G},\wt{\lambda})$$
		which preserves the support in the sense made precise in \emph{loc. cit.}.
		
	\end{theorem}
	
	The proof is based on finding explicit generators of $\mch(\wt{G},\wt{\lambda})$, similar to a parallel argument of S\'echerre \cite{secherre2005types}*{Th\'eor\`eme 4.6}. 
	It is expected that $\mch(\wt{G},\wt{\lambda})$ is an affine Hecke algebra in general (\emph{cf.} Conjecture \ref{conjHGlambdaaffine}). Anyway, we give a rather satisfactory answer to questions (2) and (3) for simple types, which in particular includes all the cuspidal inertial equivalence classes.
	
	We also expect to construct types related to any inertial equivalence classes of $\wt{G}$, and calculate the corresponding Hecke algebra if possible. They are related to the so-called \emph{semi-simple types} in \cite{bushnell1999semisimple} and \cite{secherre2012smooth}. Hopefully this could be done in a sequel of this article.
	
	Finally, we focus on the case where $\wt{G}$ is either a Kazhdan-Patterson cover or the Savin's cover (\emph{cf.} \S \ref{subsectionmetacoverGL}). These two types of covers are special, in the sense that there exists a certain ``metaplectic tensor product" functor from representations of covers of smaller $\mrgl(F)$'s to representations of a Levi subgroup $\wt{M}$ of $\wt{G}$. Also, both of them admit a classification result of irreducible representations in the sense of Zelevinsky \cite{zelevinsky1980induced}. We list main results that we achieve for such covers.
	\begin{itemize}
		\item (\emph{cf.} Corollary \ref{corcalHeckeKPS} and Corollary \ref{corpsiisometry}) The embedding $\Psi$ in Theorem \ref{thmmainQthree} is an isomorphism, which can also be chosen as an isometry.
		\item (\emph{cf.} Proposition \ref{proplocusindirred}) Assume $G=G_{2r_{0}}$, let $P$ be a parabolic subgroup of $G$ of a Levi factor $M=G_{r_{0}}\times G_{r_{0}}$. Let $\wt{\pi}$ be a unitary cuspidal representation of $\wt{M}$ invariant under the action of the Weyl group  $W(G,M)$. Let $\nu(\cdot)=\abs{\mrdet(\cdot)}_{F}$, which is an unramified character of $G_{r_{0}}$. Then the positive real number $s\in\mbr$ such that the parabolic induction $i_{\wt{P}}^{\wt{G}}(\wt{\pi}\cdot(\nu^{-s}\boxtimes\nu^{s}))$ is reducible can be explicitly calculated.
		\item (\emph{cf.} Proposition \ref{propZLimage}) The corresponding equivalence of categories 
		$$\wt{\mct}_{G}:\mrrep_{\mfs}(\wt{G})\rightarrow\mrmod(\wt{\mch}(t,s_{0},\bs{q}_{0})),$$
		which is the composition of the functor $\bs{\mathrm{M}}_{\wt{\lambda}}:\mrrep_{\mfs}(\wt{G})\rightarrow\mrmod(\mch(\wt{G},\wt{\lambda}))$ and the functor $\Psi^{*}:\mrmod(\mch(\wt{G},\wt{\lambda}))\rightarrow\mrmod(\wt{\mch}(t,s_{0},\bs{q}_{0}))$ induced from the pull-back of $\Psi$, can be fully determined. 
		\item (\emph{cf.} Proposition \ref{propdiscinerclass}) The set of $\wt{G}$-conjugacy classes of weak equivalence classes of simple types $(\wt{J},\wt{\lambda})$ are in bijection with the set of \emph{discrete inertial equivalence classes} of $\wt{G}$. 
	\end{itemize}
	
	\subsection{Structure of the article}
	
	We outline the structure of this article. 
	
	Section 2-5 are preliminaries. We fix general notation in Section 2, sketch general results of type theory in Section 3, introduce metaplectic covers in Section 4 with an emphasis on the case $G=\mrgl_{r}(F)$, and recall the theory of strata and simple characters of Bushnell-Kutzko in Section 5. 
	
	In Section 6, we introduce the key definition of homogeneous types and (twisted) simple types. We also resolve the crucial problem of calculating the intertwining set and normalizer of a simple type $(\wt{J},\wt{\lambda})$ in $G$. Theorem \ref{thmmainQone} and Theorem \ref{thmmainQtwo}, except the exhaustion of $(\wt{J_{M}},\wt{\lambda}_{M})$, are stated, while some of them are proved in \S \ref{subsectionmainhomsimtype}. 
	
	In Section 7, the main focus is to calculate the Hecke algebra $\mch(\wt{G},\wt{\lambda})$ of a simple type $(\wt{J},\wt{\lambda})$. As we explained before, we are able to find an embedding $\Psi$ from an affine Hecke algebra $\wt{\mch}(t,s_{0},\bs{q}_{0})$ of type A to $\mch(\wt{G},\wt{\lambda})$, which in particular is an isomorphism if $\wt{G}$ is either a Kazhdan-Patterson cover or Savin's cover. As a result, Theorem \ref{thmmainQthree} is stated and proved, so is the rest of the claims in \S \ref{subsectionmainhomsimtype} for at least $M=G$.
	
	In Section 8, we state the exhaustion of $(\wt{J_{M}},\wt{\lambda}_{M})$ in Theorem \ref{thmmainQone}, whose proof is accomplished in Section 10.
	
	Finally in Section 9, our main focus is the two classes of special covers of $G=\mrgl_{r}(F)$ we mentioned above. We prove all the statements we claimed. 
	
	\subsection{Acknowledgement}
	
	I would like to thank Corinne Blondel, Fan Gao, Max Gurevich, Vincent S\'echerre and Chuijia Wang for useful discussions or correspondences. This research was supported by the Israel Science Foundation (grant No. 737/20). Moreover, I would like to thank an anonymous referee for his/her detailed report and pertinent advice.
	
	I would like to express my gratitude to Colin J. Bushnell for his generous encouragement, without which this work would never be able to appear. I hope that it is proper to dedicate it to his memory.
	
	\section{Preliminaries}
	
	In this article, we fix a non-archimedean locally compact field $F$, whose residual field $\bs{k}$ is of cardinality $q$. We write $\mfo_{F}$ for the ring of integers of $F$ and $\mfp_{F}$ for the maximal ideal of $\mfo_{F}$. We write $v_{F}:F^{\times}\rightarrow \mbz$ for the canonical discrete valuation with respect to $F$. 
	
	We fix a positive integer $n$ that divides $q-1$. It implies that the subgroup of $n$-th roots of unity of $F^{\times}$, denoted by $\mu_{n}(F)$ and usually abbreviated by $\mu_{n}$, consists of $n$ elements. Let $\abs{\cdot}_{F}$ denote the discrete valuation of $F$.
	
	We denote by $(\cdot,\cdot)_{n}:F^{\times}\times F^{\times}\rightarrow\mu_{n}$ the $n$-th Hilbert symbol, which is a bimultiplicative, anti-symmetric pairing, that descends to a non-degenerate bimultiplicative pairing $F^{\times}/F^{\times n}\times F^{\times}/F^{\times n}\rightarrow\mu_{n}$, where $F^{\times n}=\{x^{n}\mid x\in F^{\times}\}$. Since $\mrgcd(q,n)=1$, such a pairing is trivial on $\mfo_{F}^{\times}\times\mfo_{F}^{\times}$. We refer to \cite{weil1974basic}*{XIII.\S 5} for all the required properties in this article. In general, for a finite extension $E/F$, we denote by $(\cdot,\cdot)_{n,E}:E^{\times}\times E^{\times}\rightarrow\mu_{n}$ the corresponding $n$-th Hilbert symbol. It is known that $(x,y)_{n,E}=(x,\mrn_{E/F}(y))_{n,F}$, where $x\in F^{\times}$, $y\in E^{\times}$ and $\mrn_{E/F}:E^{\times}\rightarrow F^{\times}$ denotes the norm map.
	
	By $\ell$-groups in this article, we mean locally compact totally disconnected topological groups as in \cite{bernstein1976representations}. By representations of an $\ell$-group in this article, we mean complex smooth representations. In particular, a character is a one-dimensional smooth representation. 
	
	We fix an additive character $\psi_{F}:F\rightarrow\mbc^{\times}$ of conductor $\mfp_{F}$, i.e. it is trivial on $\mfp_{F}$ and is not trivial on $\mfo_{F}$. 
	
	Let $G$ be an $\ell$-group. We write $Z(G)$ for the center of $G$ and $[\cdot,\cdot]:G\times G\rightarrow G$ for the commutator map given by $[g_{1},g_{2}]=g_{1}g_{2}g_{1}^{-1}g_{2}^{-1},\ g_{1},g_{2}\in G$.
	
	We denote by $\mrdet:\mrm_{r}(F)\rightarrow F$ and $\mrtr:\mrm_{r}(F)\rightarrow F$ the determinant map and the trace map respectively. In general, for a finite extension $E/F$, we write $\mrdet_{E}:\mrm_{r}(E)\rightarrow E$ and $\mrtr_{E}:\mrm_{r}(E)\rightarrow E$ correspondingly. 
	
	Let $\nu=\abs{\cdot}_{F}$ which is an unramified character of $F^{\times}$. By composing with the determinant map $\mrdet_{F}$, we identify $\nu$ with a character of $\mrgl_{r}(F)$ for any $r$.
	
	Let $H\subset H'$ be two closed subgroups of $G$. We write $\mrInd_{H}^{H'}:\mrrep(H)\rightarrow\mrrep(H')$, $\mrind_{H}^{H'}:\mrrep(H)\rightarrow\mrrep(H')$ and $\rest_{H}:\mrrep(H')\rightarrow\mrrep(H)$ for the induction, compact induction and restriction functors respectively, where $\mrrep(\cdot)$ denotes the category of smooth representations. For a representation $\pi$ of $H'$, we often write $\pi$ instead of $\pi|_{H}$ for short, if the domain of definition of $\pi$ (i.e. $H$) is clear from the context. By convention, we say that a representation $\pi$ of $H'$ contains an irreducible representation $\pi'$ of $H$ if $\mrhom_{H}(\pi',\pi\rest_{H})\neq 0$. Assume $H$ to be an open normal subgroup of $H'$, let $\overline{H}=H'/H$ and let $\overline{\rho}$ be a representation of $\overline{H}$. Then we denote by $\mrinf_{\overline{H}}^{H'}\overline{\rho}$ the \emph{inflation} of $\overline{\rho}$ as a representation of $H'$ that extends trivially on $H$.
	
	Let $(\rho,W)$ be a representation of $H$. We write $(\rho^{\vee},W^{\vee})$ for the contragredient of $(\rho,W)$. For $g,x\in G$, we denote by $x^{g}=g^{-1}xg$, $H^{g}=g^{-1}Hg$, and $\rho^{g}(\cdot)=\rho(g\cdot g^{-1})$ as a representation of $H^{g}$. 
	We say that $g\in G$ \emph{intertwines} $\rho$, if the intertwining space
	$$\mrhom_{H^{g}\cap H}(\rho^{g},\rho)$$
	is non-zero, and we denote by $I_{G}(\rho)$ the intertwining set of $\rho$ consisting of $g\in G$ satisfying the above condition. In general, for two closed subgroups $H_{1}$, $H_{2}$ of $G$ and $\rho_{1}$, $\rho_{2}$ their representations respectively, we say that $g\in G$ intertwines $\rho_{1}$ and $\rho_{2}$, if the intertwining space
	$$\mrhom_{H_{1}^{g}\cap H_{2}}(\rho_{1}^{g},\rho_{2})$$
	is non-zero. We say that $g$ \emph{normalizes} $H$ (resp. $\rho$) if $H^{g}=H$ (resp. $\rho^{g}\cong\rho$), and we write $N_{G}(H)$ (resp. $N_{G}(\rho)$) for the corresponding normalizer. 
	
	For $x\in\mbr$, we denote by $\floor{x}$ (resp. $\ceil{x}$) the largest integer smaller than or equal to (resp. the smallest integer greater than or equal to) $x$.
	
	Let $\mfS_{k}$ be the group of permutations of $k$ elements. 
	
	Let $\mrgcd(k_{1},\dots,k_{m})$ denote the greatest common divisor of $k_{1},\dots,k_{m}$. 
	
	\section{Hecke algebra, cuspidal representations, types and covering pairs}\label{sectionHeckecoveringpair}
	
	In this section, we follow \cite{bushnell129admissible} and \cite{bushnell1998smooth} to recall some known results for Hecke algebra, cuspidal representations, types and covering pairs.
	
	\subsection{Hecke algebra}
	
	Our reference here is \cite{bushnell129admissible}*{\S 4} and \cite{bushnell1998smooth}*{\S 2}. Let $G$ be an $\ell$-group, let $J$ be an open compact subgroup of $G$
	, and let $(\rho,W)$ be an irreducible representation of $J$, which is necessarily finite dimensional. We fix a Haar measure $dx$ of $G$. 
	
	Let $\mch(G)$ denote the Hecke algebra of $G$ consisting of complex smooth compactly supported functions on $G$, equipped with the convolution for the product structure. 
	It is well-known that we have an equivalence of categories $$\mrrep(G)\rightarrow\mrmod(\mch(G)),$$
	where $\mrmod(\mch(G))$ denotes the category of non-degenerate $\mch(G)$-modules, and for $(\pi,V)\in\mrrep(G)$ the corresponding $\mch(G)$-module structure is given by $$v\mapsto \pi(f)v:=\int_{G}f(x)\pi(x)vdx$$ for $f\in\mch(G)$ and $v\in V$.
	
	We denote by $\mch(G,\rho)$ the space of compactly supported functions $\phi:G\rightarrow\mrend_{\mbc}(W^{\vee})$, such that $$\phi(g_{1}gg_{2})=\rho^{\vee}(g_{1})\phi(g)\rho^{\vee}(g_{2}),\quad g_{1},g_{2}\in J,\ g\in G.$$ 
	The convolution operation
	$$\phi_{1}\ast \phi_{2}(g)=\int_{G}\phi_{1}(x)\phi_{2}(x^{-1}g)dx,\quad \phi_{1},\phi_{2}\in\mch(G,\rho)$$
	gives $\mch(G,\rho)$ the structure of an associative $\mbc$-algebra with unit. We will use the abbreviation
	$$\phi^{k}=\underbrace{\phi\ast\dots\ast\phi}_{k\text{-copies}}.$$
	By \cite{bushnell129admissible}*{Proposition 4.1.1}, for $g\in G$, there exists a function $\phi\in \mch(G,\rho)$ whose support is $JgJ$ (which is unique up to scalar) if and only if $g$ intertwines $\rho$.
	
	We fix a complex structure on the finite dimensional complex vector space $W^{\vee}$. So $(\rho^{\vee},W^{\vee})$ is realized as a unitary representation. We define 
	$$h_{G}(\phi_{1},\phi_{2})=\int_{x\in G}\mrtr_{W^{\vee}}(\phi_{1}(x)\overline{\phi_{2}(x)})dx=\mrtr_{W^{\vee}}(\phi_{1}\ast\overline{\phi_{2}}(1)),\quad \phi_{1},\phi_{2}\in\mch(G,\rho),$$
	as a positive definite hermitian form on $\mch(G,\rho)$. Here, the ``bar" for $\overline{\phi_{2}(x)}$ denotes the complex conjugate on $\mrend_{\mbc}(W^{\vee})$ induced by that on $W^{\vee}$ and the ``bar" for $\overline{\phi_{2}}$ is given by $\overline{\phi_{2}}(x)=\overline{\phi_{2}(x^{-1})}$ for $x\in G$, and $\mrtr_{W^{\vee}}:\mrend(W^{\vee})\rightarrow\mbc$ denotes the trace map. Of course, $h_{G}$ depends on the choice of the Haar measure $dx$ on $G$. 
	
	For $a\in \mrend_{\mbc}(W)$, we write $a^{\vee}\in\mrend_{\mbc}(W^{\vee})$ for the transpose of $a$ with respect to the canonical pairing between $W$ and $W^{\vee}$. Then we get an isomorphism of algebras
	$$\mch(G,\rho)\rightarrow\mch(G,\rho^{\vee})^{\emph{op}},\quad \phi\mapsto \hat{\phi}:=[g\mapsto \phi(g^{-1})^{\vee}].$$
	The map
	$$(\phi,f)\mapsto\phi\ast f:=[g\mapsto\int_{G}\phi(x)f(x^{-1}g)],\quad \phi\in\mch(G,\rho^{\vee}),\ f\in\mrind_{J}^{G}\rho$$
	induces an isomorphism of algebras
	$\mch(G,\rho^{\vee})\cong\mrend_{G}(\mrind_{J}^{G}\rho).$
	Also, it provides $\mrind_{J}^{G}\rho$ with a left $\mch(G,\rho^{\vee})$-module structure and thus a right $\mch(G,\rho)$-module structure. 
	
	Let $(\pi,V)\in\mrrep(G)$. Let $$V_{\rho}:=\mrhom_{J}(W,V)=\mrhom_{G}(\mrind_{J}^{G}W,V)$$ be the space of $\rho$-invariants of $(\pi,V)$. It is endowed with a left $\mch(G,\rho)$-module structure via the right $\mch(G,\rho)$-action on $\mrind_{J}^{G}\rho$ explained as above. More precisely, the action is given by
	\begin{equation}\label{eqheckepirho}
		\phi\cdot f=[w\mapsto\int_{G}\pi(g)f(\hat{\phi}(g^{-1})w)dg],\quad \phi\in\mch(G,\rho),\ f\in V_{\rho},\ w\in W.
	\end{equation}
	This provides us with a functor
	\begin{equation}\label{eqinvariantfunctor}
		\bs{\mathrm{M}}_{\rho}:\mrrep(G)\rightarrow\mrmod(\mch(G,\rho)),\quad (\pi,V)\mapsto V_{\rho}.
	\end{equation}
	Here, $\mrmod(\mch(G,\rho))$ denotes the equivalence classes of non-degenerate $\mch(G,\rho)$-modules.
	
	\subsection{Cuspidal representations}
	
	In the rest of this section, let $G$ be a $p$-adic reductive group, or more generally, an $n$-fold cover of a $p$-adic reductive group (\emph{cf.} \cite{renard2010representations}, \cite{kaplan2022note}). 
	
	Let $P=MN$ be a parabolic subgroup of $G$, with $M$ being a Levi factor and $N$ being the unipotent radical of $P$. 
	
	We define, as in \cite{bernstein1977induced}*{\S 1.8}, the normalized parabolic induction functor and normalized Jacquet functor
	$$i_{P}^{G}:\mrrep(M)\rightarrow\mrrep(G),\quad r_{N}:\mrrep(G)\rightarrow\mrrep(M).$$
	
	We recall the following equivalent definitions for a cuspidal representation of $G$.
	
	\begin{definition}[\cite{bernstein1977induced}, \cite{kaplan2022note}]\label{defcuspidal}
		
		An irreducible representation $\pi$ of $G$ is called cuspidal if the following equivalent statements are satisfied:
		\begin{enumerate}
			\item $\pi$ does not occur as a subrepresentation of a parabolic induction $i_{P}^{G}(\rho)$ for any proper parabolic subgroup $P=MN$ of $G$ and any irreducible representation $\rho$ of $M$.
			
			\item The Jacquet module $r_{N}(\pi)$ is zero for any proper parabolic subgroup $P$ of $G$ with the unipotent radical $N$.
			
			\item Every matrix coefficient of $\pi$ is compact modulo the center.
			
		\end{enumerate}
		
	\end{definition}
	
	One typical method of obtaining cuspidal representations is the usage of compact induction, which is summarized as the following lemma:
	
	\begin{lemma}[\cite{bushnell2006local}*{Theorem 11.4}]\label{lemmacomindcusp}
		
		Let $\bs{J}$ be an open subgroup of $G$, containing and compact modulo center of $G$. Let $\bs{\lambda}$ be an irreducible representation of $\bs{J}$. Assume that the intertwining set $I_{G}(\bs{\lambda})$ equals $\bs{J}$. Then the compact induction $\mrind_{\bs{J}}^{G}\bs{\lambda}$ is an irreducible cuspidal representation of $G$, and moreover, any irreducible representation $\pi$ of $G$ containing $\bs{\lambda}$ is cuspidal.
		
	\end{lemma}
	
	\subsection{Bernstein blocks and types}\label{subsectionblockstypes}
	
	A \emph{cuspidal pair} $(M,\rho)$ of $G$ consists of a Levi subgroup $M$ of $G$ and a cuspidal representation $\rho$ of $M$. It is known that for an irreducible representation $\pi$ of $G$, there exists a cuspidal pair $(M,\rho)$, unique up to $G$-conjugacy, and a parabolic subgroup $P$ of $G$ having a Levi factor $M$, such that $\pi$ is a subrepresentation of the parabolic induction $i_{P}^{G}(\rho)$, where $P$ is a parabolic subgroup of $M$. Such a pair is called the \emph{cuspidal support} of $\pi$. 
	
	Moreover, let $\mco$ be the orbit of $\rho$ twisted by unramified characters of $M$. Then $\mfs=(M,\mco)$, up to $G$-conjugacy, is called an \emph{inertial equivalence class} of $G$. The \emph{inertial support} of $\pi$ is the inertial equivalence class of its cuspidal support.
	
	Write $\mrrep_{\mfs}(G)$ for the subcategory of $\mrrep(G)$ consisting of representations whose irreducible subquotients have inertial support $\mfs$. Then, each $\mrrep_{\mfs}(G)$ is a block of $\mrrep(G)$, and we have the block decomposition (\emph{cf.} \cite{renard2010representations}*{Chapitre IV, \S 7.2}, \cite{kaplan2022note})
	$$\mrrep(G)\cong\prod_{\mfs}\mrrep_{\mfs}(G),$$
	where $\mfs$ ranges over all the inertial equivalence classes of $G$.
	
	Let $J$ be an open compact subgroup of $G$ and let $\lambda$ be an irreducible representation of $J$. As in \cite{bushnell1998smooth}*{\S 3}, we may define a special idempotent $e_{\lambda}$ in $\mch(G)$ and the subcategory $\mrrep_{\lambda}(G)=\{V\mid \mch(G)e_{\lambda} V=V\}$ of $\mrrep(G)$. Then, we call $(J,\lambda)$ an $\mfs$-\emph{type}, if for any irreducible representation $\pi$ of $G$, we have that $\pi|_{J}$ contains $\lambda$ if and only if $\pi$ has inertial equivalence class $\mfs$. 
	
	In this case, we have $\mrrep_{\lambda}(G)=\mrrep_{\mfs}(G)$ and an equivalence of categories 
	$$\bs{\mathrm{M}}_{\lambda}:\mrrep_{\mfs}(G)\rightarrow\mrmod(\mch(G,\lambda)),\quad (\pi,V)\mapsto V_{\lambda}.$$
	Such $\bs{\mathrm{M}}_{\lambda}$ is equivariant up to the multiplicative action by unramified characters of $G$ on both $\mrrep_{\mfs}(G)$ and $\mrmod(\mch(G,\lambda))$. 
	
	\subsection{Covering pairs}\label{subsectioncoverpair}
	
	Let $P=MN$ be a parabolic subgroup of $G$, with $M$ being a Levi factor and $N$ being the unipotent radical of $P$. Let $P^{-}$ and $N^{-}$ be the opposite of $P$ and $N$ respectively with respect to $M$. A pair $(J,\lambda)$, consisting of an open subgroup $J$ of $G$ and an irreducible $\lambda$ of $J$, is called \emph{decomposed with respect to $(M,P)$} if
	\begin{itemize}
		\item $J=(J\cap N^{-})\cdot(J\cap M)\cdot(J\cap N)$.
		\item The groups $J\cap N^{-}$ and $J\cap N$ are contained in the kernel of $\lambda$.
	\end{itemize}
	An element $\zeta\in M$ is called \emph{strongly $(P,J)$-positive} if
	\begin{itemize}
		\item $\zeta(J\cap N)\zeta^{-1}\subset (J\cap N)$, $\zeta^{-1}(J\cap N^{-})\zeta\subset (J\cap N^{-}).$
		\item For any compact open subgroups $H_{1}$, $H_{2}$ of $N$, there exists an integer $i\geq 0$ such that $\zeta^{i}H_{1}\zeta^{-i}\subset H_{2}$.
		\item For any compact open subgroups $H_{1}$, $H_{2}$ of $N^{-}$, there exists an integer $i\geq 0$ such that $\zeta^{-i}H_{1}\zeta^{i}\subset H_{2}$.
	\end{itemize}
	
	Consider another pair $(J_{M},\lambda_{M})$, where $J_{M}$ is an open compact subgroup of $M$ and $\lambda_{M}$ is an irreducible representation of $J_{M}$. The pair $(J,\lambda)$ is a ($G$-)\emph{covering pair} of $(J_{M},\lambda_{M})$ if
	\begin{enumerate}
		\item The pair $(J,\lambda)$ is decomposed with respect to $(M,P)$ for every parabolic subgroup $P$ of $G$ with a Levi factor $M$.
		\item $J\cap M=J_{M}$ and $\lambda|_{J_{M}}=\lambda_{M}$.
		\item For every parabolic subgroup $P$ of $G$ with a Levi factor $M$, there exist a strongly $(P,J)$-positive element $z\in Z(M)$ and an invertible element $\phi_{z}$ in $\mch(G,\lambda)$ supported on $JzJ$.
	\end{enumerate}
	Sometimes, the following condition is more convenient to verify, which implies the condition (3) above (\emph{cf.} \cite{bushnell1998smooth}*{Condition 8.2}):
	\begin{enumerate}
		\setcounter{enumi}{3}
		\item Every $\phi\in\mch(G,\lambda)$ is supported on $JMJ$.
	\end{enumerate}

	We remark that in \cite{bushnell1998smooth}, Bushnell and Kutzko use the word ``cover" for such a pair $(J,\lambda)$, which unfortunately conflicts with ``metaplectic covers", the central objects of this article. So we use ``covering pair" instead, translated from the corresponding French words (\emph{cf.} \cite{blondel1999methode}). 
	
	The following theorem illustrates the connection between types and covering pairs.
	
	\begin{theorem}[\cite{bushnell1998smooth}*{Theorem 8.3}]\label{thmtypecovering}
		
		Let $L$ be a Levi subgroup of $M$ and let $\rho$ be a cuspidal representation of $L$. Let $\mfs_{G}$ (resp. $\mfs_{M}$) be the inertial equivalence class of $G$ (resp. $M$) that contains $(L,\rho)$. If $(J_{M},\lambda_{M})$ is a $\mfs_{M}$-type and $(J,\lambda)$ is an covering pair of $(J_{M},\lambda_{M})$, then $(J,\lambda)$ is an $\mfs_{G}$-type. 
		
	\end{theorem}
	
	Fix a parabolic subgroup $P=MN$ of $G$. Let $(J,\lambda)$ be a covering pair of $(J_{M},\lambda_{M})$. In \cite{bushnell1998smooth}*{\S 7}, an embedding of algebras $t_{P}:\mch(M,\lambda_{M})\rightarrow\mch(G,\lambda)$ is defined and uniquely characterized by the following properties:
	\begin{itemize}
		\item For any strongly $(P,J)$-positive element $\zeta\in M$, define $\varphi_{\zeta}\in\mch(M,\lambda_{M})$ (resp. $\phi_{\zeta}\in\mch(G,\lambda)$) such that $\mrsupp(\varphi_{\zeta})=J_{M}\zeta J_{M}$ and $\varphi_{\zeta}(\zeta)=\mathrm{id}_{\lambda_{M}^{\vee}}$ (resp. $\mrsupp(\phi_{\zeta})= J\zeta J$ and $\phi_{\zeta}(\zeta)=\mathrm{id}_{\lambda^{\vee}}).$ 
		Then $t_{P}(\varphi_{\zeta})=\delta_{N}^{1/2}(\zeta)\cdot\phi_{\zeta}$, where $\delta_{N}:M\rightarrow\mbc^{\times}$ denotes the modulus character with respect to $N$. 
		
		\item Let $\mu_{G}$ (resp. $\mu_{M}$) be the Haar measure on $G$ (resp. $M$), such that $\mu_{G}(J)=\mu_{M}(J_{M})=1$. Using these two Haar measures and fixing a complex structure on the representation space of $\lambda^{\vee}$ and $\lambda_{M}^{\vee}$, we define the corresponding hermitian forms $h_{G}$ on $\mch(G,\lambda)$ and $h_{M}$ on $\mch(M,\lambda_{M})$. Then we have \begin{equation}\label{eqheckeherminv}
			h_{M}(\varphi,\varphi)=h_{G}(t_{P}(\varphi),t_{P}(\varphi)),\quad\text{for any}\ \varphi\in\mch(M,\lambda_{M}).
		\end{equation}
		In other words, $t_{P}$ is an isometry between $\mch(M,\lambda_{M})$ and $t_{P}(\mch(G,\lambda))$.
		
		\item If $\mch(G,\lambda)_M:=\{\phi\in \mch(G,\lambda)\mid \mrsupp(\phi)\subset JMJ \}$ is a subalgebra of $\mch(G,\lambda)$, then we have $t_P(\mch(M,\lambda_M))=\mch(G,\lambda)_M$ and the map $t_P$ preserves support of functions (\emph{cf.} \cite{bushnell1998smooth}*{Theorem 7.2.(ii)}), in the sense that 
		$$\mrsupp(t_{P}(\varphi))=J\cdot\mrsupp(\varphi)\cdot J,\quad \varphi\in\mch(M,\lambda_{M}).$$
		
	\end{itemize}
	We remark that the second property follows from \cite{roche1998types}*{\S 5}. In \emph{loc. cit.} the condition (4) for a covering pair is imposed, which simply guarantees that $t_{P}$ is an isomorphism and is not essential for our statement here. Moreover, although in \emph{loc. cit.} the representation $\lambda$ is assumed to be one-dimensional, as already pointed out by the author, the proof is general enough to be adapted to an irreducible finite dimensional representation $\lambda$ of $J$. 
	
	The above homomorphism $t_{P}$ induces a map 
	$$t_{P}^{*}:\mrmod(\mch(G,\lambda))\rightarrow\mrmod(\mch(M,\lambda_{M}))$$
	as the pull-back map of $t_{P}$. 
	Such $t_{P}^{*}$ should be regarded as the normalized Jacquet functor on the Hecke algebra side. Also, the above $t_{P}^{*}$ has a right adjoint
	$$(t_{P})_{*}:\mrmod(\mch(M,\lambda_{M}))\rightarrow \mrmod(\mch(G,\lambda)).$$
	Indeed, for $V\in\mrmod(\mch(M,\lambda_{M}))$, we have $$(t_{P})_{*}(V)=\mrhom_{\mch(M,\lambda_{M})}(\mch(G,\lambda),V).$$
	Correspondingly, such $(t_{P})_{*}$ should be regarded as the normalized parabolic induction functor on the Hecke algebra side. More precisely, we have the following theorem:
	
	\begin{theorem}[\cite{bushnell1998smooth}*{Theorem 7.9, Corollary 8.4}]\label{thmJacquetcomm}
		
		Under the setting of Theorem \ref{thmtypecovering}, we have commutative diagrams:
		$$\xymatrix{
			\mrrep_{\mfs_G}(G) \ar[d]_-{r_{N}} \ar[r]^-{\bs{\mathrm{M}}_{\lambda}} & \mrmod(\mch(G,\lambda)) \ar[d]^-{t_{P}^{*}}  &	\mrrep_{\mfs_{G}}(G) \ar[r]^-{\bs{\mathrm{M}}_{\lambda}} & \mrmod(\mch(G,\lambda))  \\ 
			\mrrep_{\mfs_M}(M)\ar[r]_-{\bs{\mathrm{M}}_{\lambda_{M}}}  & \mrmod(\mch(M,\lambda_{M}))
			&	\mrrep_{\mfs_{M}}(M) \ar[r]_-{\bs{\mathrm{M}}_{\lambda_{M}}} \ar[u]^-{i_{P}^G} & \mrmod(\mch(M,\lambda_{M})) \ar[u]_-{(t_{P})_{*}}
		}$$
		
	\end{theorem}
	
	\begin{remark}
		
		The character $\delta_{N}^{1/2}$ introduced above guarantees that $t_{P}^{*}$ (resp. $(t_{P})_{*}$) is compatible with the \textbf{normalized} Jacquet module functor $r_{N}$ (resp. parabolic induction functor $i_{P}^{G}$). Note that there is a $\delta_{N}^{1/2}$-shift between our definition of $t_{P}$ and that of \cite{bushnell1998smooth}, since in \emph{ibid.} they consider \textbf{non-normalized} parabolic induction and Jacquet module functor.
		
	\end{remark} 
	
	\section{Metaplectic covers of $\mrgl_{r}(F)$}\label{sectionnfoldcover}
	
	\subsection{General theory for a finite central cover}
	
	We refer to \cite{gan2018groups} for the basic facts that we are going to state below, which also serves as an excellent historical survey.
	
	Let $G$ be an $\ell$-group. By an $n$-fold cover of $G$, we mean a central extension of $G$ by $\mu_{n}$ as $\ell$-groups:
	\begin{equation}\label{eqcentralext}
		\xymatrix{0 \ar[r] & \mu_{n} \ar[r]^-{} & \wt{G} \ar[r]^-{\bs{p}} & G \ar[r] & 0}.
	\end{equation}
	We note that the set of equivalence classes of such central extensions is in bijection with the set of (locally constant) 2-cohomology classes $H^{2}(G,\mu_{n})$. More precisely, given an $n$-fold cover $\wt{G}$ of $G$, we fix a continuous map $\bs{s}:G\rightarrow \wt{G}$ such that $\bs{p}\circ\bs{s}=\id$. Then the corresponding cohomology class is related to the 2-cocycle $\sigma:G\times G\rightarrow \mu_{n}$ satisfying
	\begin{equation}\label{eq2cocycle1}
		\bs{s}(g_{1})\bs{s}(g_{2})\sigma(g_{1},g_{2})=\bs{s}(g_{1}g_{2}),\quad g_{1},g_{2}\in G.
	\end{equation}
	For two covers $\wt{G}_{1}$ and $\wt{G}_{2}$ realized by 2-cocycles $\sigma_{1}$ and $\sigma_{2}$ respectively, the Baer sum of $\wt{G}_{1}$ and $\wt{G}_{2}$ is realized by the 2-cocycle $\sigma_{1}\cdot\sigma_{2}$.
	
	By convention, for a closed subgroup $H$ of $G$ we write $\wt{H}$ for the preimage $\bs{p}^{-1}(H)$. We identify $\mu_{n}$ with a central subgroup of $\wt{G}$. 
	
	A \emph{splitting} of $H$ is a continuous group homomorphism $\bs{s}_{H}:H\rightarrow \wt{G}$ satisfying $\bs{p}\circ\bs{s}_{H}=\id$. In general, such a splitting may not exist or may not be unique. If $H$ is a pro-$p$-group and $\mrgcd(n,p)=1$, then the cohomology group $H^{2}(H,\mu_{n})$ is trivial, implying that there exists a unique splitting $\bs{s}_{H}$ of $H$. Thus we identify $H$ with a subgroup $\bs{s}_{H}(H)$ of $\wt{H}$, which we denote by $\,_{s}H$. For a representation $\rho$ of $H$, we similarly write $\,_{s}\rho$ for the corresponding representation of $\,_{s}H$. In general, even if the splittings of $H$ are not unique, we may still define $\,_{s}H$ and $\,_{s}\rho$ as above once we fix a splitting $\bs{s}_{H}$. We refer to \cite{van2017beta} for a similar setting.
	
	The commutator $[\cdot,\cdot]:\wt{G}\times\wt{G}\rightarrow\wt{G}$ factors through $G\times G$, and we denote by $[\cdot,\cdot]_{\sim}:G\times G\rightarrow\wt{G}$ the resulting map. In particular, if $g_{1},g_{2}\in G$ commute, then  
	\begin{equation}\label{eqformcommutator}
		[g_{1},g_{2}]_{\sim}=\sigma(g_{1},g_{2})\sigma(g_{2},g_{1})^{-1}\in\mu_{n}.
	\end{equation} 
	Similarly, the $\wt{G}$-conjugation on $\wt{G}$ factors through $G$, so we may consider the $G$-conjugation on $\wt{G}$. It is clear that $[g_{1}^{g},g_{2}^{g}]_{\sim}=[g_{1},g_{2}]_{\sim}$ for any $g,g_{1},g_{2}\in G$. Also, for $H,H_{1},H_{2}$ closed subgroups of $G$, we define the coset $g\wt{H}:=\bs{s}(g)\wt{H}$ and the double coset $\wt{H_{1}}g\wt{H_{2}}:=\wt{H_{1}}\bs{s}(g)\wt{H_{2}}$, which does not depend on the choice of the section $\bs{s}$.
	
	For $\wt{\rho}\in\mrrep(\wt{H})$, we say that $g\in G$ intertwines $\wt{\rho}$ if $\bs{s}(g)$ intertwines $\wt{\rho}$, which does not depend on the choice of $\bs{s}$. We denote by $I_{G}(\wt{\rho})$ the intertwining set of $\wt{\rho}$ as a subset of $G$. Similarly, it makes sense to define the normalizer $N_{G}(\wt{\rho})$.
	
	Fix a character $\epsilon:\mu_{n}\rightarrow\mbc^{\times}$. A representation $\wt{\pi}$ of $\wt{G}$ is called ($\epsilon$-)\emph{genuine} if $\mu_{n}$ acts by $\epsilon$ on $\wt{\pi}$. We remark that essentially we only need to consider the case where $\epsilon$ is primitive, otherwise any $\epsilon$-genuine representation $\wt{\pi}$ of $\wt{G}$ corresponds to an $\epsilon'$-genuine representation of an $n'$-fold cover of $G$ with $n'$ dividing $n$, where $\epsilon'$ is a primitive character of $\mu_{n'}$. So from now on, we also assume the character $\epsilon$ to be primitive. 
	
	For a representation $\rho$ and a splitting $\bs{s}$ of $H$, we write  $\epsilon\cdot\,_{s}\rho$ for the extension of $\,_{s}\rho$ as an $\epsilon$-genuine representation of $\wt{H}$.
	
	For a representation $\chi$ of $H$ and a genuine representation $\wt{\rho}$ of $\wt{H}$, the tensor product $\wt{\rho}\otimes\chi:=\wt{\rho}\otimes(\chi\circ\bs{p})$ is well-defined as another genuine representation of $\wt{H}$. When $\chi$ is a character, we write $\wt{\rho}\chi$ or $\wt{\rho}\cdot\chi$ instead.
	
	\subsection{Metaplectic covers of $\mrgl_{r}(F)$}\label{subsectionmetacoverGL}
	
	We consider metaplectic covers of $\mrgl_{r}(F)$, which, in the context of this article, are $n$-fold covers arising from Brylinski-Deligne covers (\emph{cf.} \cite{brylinski2001central}). Roughly speaking, Brylinski and Deligne constructed a family of central extensions of a reductive group $G$ over $F$ by the second $K$-group:
	\begin{equation}\label{eqK2extension}
		\xymatrix{0 \ar[r] & K_{2}(F) \ar[r]^-{} & \hat{G} \ar[r]^-{\bs{p}} & G \ar[r] & 0}.
	\end{equation}
	Recall that the abelian group $K_{2}(F)$ could be given by generators of the form $\{x,y\}$ with $x,y\in F^{\times}$, subject only to the following relations (\emph{cf.} \cite{matsumoto1969sous}, \cite{milnor1971introduction}*{Theorem 11.1}): 
	\begin{itemize}
		\item $\{x_{1}x_{2},y\}=\{x_{1},y\}\{x_{2},y\},\ \{y,x_{1}x_{2}\}=\{y,x_{1}\}\{y,x_{2}\}, \quad\text{for any}\ x_{1},x_{2},y\in F^{\times}.$
		\item $\{x,1-x\}=1,\quad \text{for any}\ x\neq 0,1.$
	\end{itemize} 
	Pushing out the map (\emph{cf.} \cite{gan2018groups}*{p3}) $$K_{2}(F)\rightarrow \mu_{n},\quad \{x,y\}\mapsto (x,y)_{n},$$  they got a family of $n$-fold covers $\wt{G}$ of $G$.

	\begin{remark}
		
		As far as I understand, before the work of Brylinski-Deligne, the terminology ``metaplectic covers" means finite central covers of simple simply connected groups studied by Moore, Steinberg, Matsumoto, etc., and their push-outs, pull-backs, Baer sums as finite central covers of general reductive groups. See \cite{gan2018groups} for more details.
		
		We refer to Proposition \ref{propBDcoverproperty} for the properties that are special for a metapletic cover instead of a general cover of $\mrgl_{r}(F)$.
		
	\end{remark}
	
	We follow \cite{gao2019whittaker} to characterize all the metaplectic covers of $G=\mrgl_{r}(F)$. Let $T$ be the diagonal torus of $G$, let $Y=\mrhom(\mbg_{m},T)$ be the cocharacter lattice, and let $W=W(G,T)$ be the corresponding Weyl group. Then the set of $K_{2}(F)$-extensions $\hat{G}$ of $G$ are in bijection with the set of $W$-invariant quadratic forms 
	$$Q:Y\rightarrow \mbz.$$
	This bijection is additive, saying that for two Brylinski-Deligne covers $\hat{G}_{1}$ and $\hat{G}_{2}$ of $G$ with related quadratic forms being $Q_{1}$ and $Q_{2}$ respectively, the quadratic form related to the Baer sum of $\hat{G}_{1}$ and $\hat{G}_{2}$ is $Q_{1}+Q_{2}$.
	
	One way to pin down the above bijection is as follows: Let $B_{Q}$ be the associated bilinear form given by $$B_{Q}(y_{1},y_{2})=Q(y_{1}+y_{2})-Q(y_{1})-Q(y_{2}),\quad y_{1},y_{2}\in Y.$$
	Then for $y_{1},y_{2}\in Y$ and $a,b\in F^{\times}$, we have (\emph{cf.} \cite{brylinski2001central}*{Corollary 3.14})
	$$[y_{1}(a),y_{2}(b)]_{\sim}=\{a,b\}^{B_{Q}(y_{1},y_{2})},$$
	where $[\cdot,\cdot]_{\sim}:T\times T\rightarrow K_{2}(F)$ is the resulting commutator with respect to $\hat{G}$.

	Moreover, we let $e_{1},\dots,e_{r}$ be the canonical basis of $Y$. Then such quadratic forms are in bijection with pairs of integers $(\bs{a},\bs{b})$ given by $\bs{a}=Q(e_{1})$ and $\bs{b}=B_{Q}(e_{1},e_{2})$.
	Likewise, this bijection is additive. 
	
	We consider two special Brylinski-Deligne covers of $G$. 
	First we consider the $K_{2}(F)$-extension $\hat{G}_{1}$ related to the 2-cocycle
	$$G\times G\rightarrow K_{2}(F),\quad (g_{1},g_{2})\mapsto\{\mrdet(g_{1}),\mrdet(g_{2})\}.$$
	In this case, the corresponding pair of integers is $(\bs{a},\bs{b})=(1,2)$. We let $\wt{G}_{1}$ be the corresponding $n$-fold metaplectic cover of $G$, then it is related to the 2-cocycle
	\begin{equation}\label{eq2cocycledet}
		\sigma_{\mrdet}:G\times G\rightarrow\mu_{n},\quad (g_{1},g_{2})\mapsto(\mrdet(g_{1}),\mrdet(g_{2}))_{n}.
	\end{equation}
	
	We further consider the canonical $K_{2}(F)$-extension of $\mrsl_{r+1}(F)$ considered by Matsumoto with respect to the Steinberg symbol $\{\cdot,\cdot\}^{-1}:F^{\times}\times F^{\times}\rightarrow K_{2}(F)$ (\emph{cf.} \cite{matsumoto1969sous}, \cite{milnor1971introduction}*{\S 12}). Via the pull-back of the embedding $$\mrgl_{r}(F)\rightarrow\mrsl_{r+1}(F),\quad g\mapsto(\mrdet(g)^{-1},g),$$ we get a $K_{2}(F)$-extension of $G$, which we denote by $\hat{G}_{2}$. In this case, the corresponding pair of integers is $(\bs{a},\bs{b})=(1,1)$. We let $\wt{G}_{2}$ be the corresponding $n$-fold metaplectic cover of $G=\mrgl_{r}(F)$. In \cite{kazhdan1984metaplectic}*{\S 0.I}, \cite{banks1999block}, a special 2-cocycle $\sigma_{\mathrm{KP}}:G\times G\rightarrow\mu_{n}$ related to the cover $\wt{G}_{2}$ is constructed for each $r$. It is trivial when $r=1$ and satisfies the following block-compatibility:
	\begin{equation}\label{eqBLSblock}
		\begin{aligned}
			\sigma_{\mathrm{KP}}(\mrdiag(g_1,\dots,g_k),\mrdiag(g_1',\dots,g_k'))=
			\bigg[\prod_{i=1}^k\sigma_{\mathrm{KP}}(g_i,g_i')\bigg]\cdot\bigg[\prod_{1\leq i<j\leq k}(\det(g_i),\det(g_j'))_n\bigg],   \end{aligned}
	\end{equation} 
	where $g_{i},g_{i}'\in \mrgl_{r_{i}}(F)$ for each $i=1,\dots,k$ and $r=r_{1}+\cdots+r_{k}$.
	
	Since $(1,2)$ and $(1,1)$ generate $\mbz\times\mbz$, the integral combination of $\hat{G}_{1}$ and $\hat{G}_{2}$ with respect to the Baer sum ranges over all the equivalence classes of Brylinski-Deligne covers of $G$, and the integral combination of $\wt{G}_{1}$ and $\wt{G}_{2}$ ranges over all the $n$-fold metaplectic covers of $G$. In other words, when $(\bs{c},\bs{d})$ ranges over $\mbz/n\mbz\times\mbz/n\mbz$, the corresponding 2-cocycle $\sigma_{\mrdet}^{\bs{c}}\cdot\sigma_{\mathrm{KP}}^{\bs{d}}$ represents all the $n$-fold metaplectic covers of $G$. 
	
	We fix a pair of integers $(\bs{c},\bs{d})$, and we let $\sigma=\sigma_{\mrdet}^{\bs{c}}\cdot\sigma_{\mathrm{KP}}^{\bs{d}}$ be the related 2-cocycle for $r\geq 1$. We let $\wt{G}$ be the related $n$-fold metaplectic cover of $G$. We list various properties of $\sigma$.
	
	\begin{proposition}\label{propBDcoverproperty}
		
		\begin{enumerate}
			\item $\sigma(x,y)=(x,y)_{n}^{\bs{c}}$ for $r=1$ and $x,y\in F^{\times}$.
			\item For $r=r_{1}+\dots+r_{k}$ and $g_{i},g_{i}'\in\mrgl_{r_{i}}(F)$, we have
			\begin{equation}\label{eqBDblock}
				\begin{aligned}
					&\sigma(\mrdiag(g_1,\dots,g_k),\mrdiag(g_1',\dots,g_k'))\\=&
					\bigg[\prod_{i=1}^k\sigma(g_i,g_i')\bigg]\cdot\bigg[\prod_{1\leq i<j\leq k}(\det(g_i),\det(g_j'))_n\bigg]^{\bs{c}+\bs{d}}\cdot\bigg[\prod_{1\leq j<i\leq k}(\det(g_i),\det(g_j'))_n\bigg]^{\bs{c}}.
				\end{aligned}
			\end{equation} 
			\item Keep the notation of (2) and assume $g=\mrdiag(g_{1},\dots,g_{k})$ and $g'=\mrdiag(g_{1}',\dots,g_{k}')$ commute, then
			\begin{equation}\label{eqBDcommutator}
				[g,g']_{\sim}=\prod_{i=1}^{k}[g_{i},g_{i}']_{\sim}\cdot\prod_{i\neq j}(\mrdet(g_{i}),\mrdet(g_{j}))_{n}^{2\bs{c}+\bs{d}}.
			\end{equation}
			\item For $g\in G$ and $z=\lambda I_{r}\in Z(G)$, we have
			\begin{equation}\label{eqcentercommutator}
				[z,g]_{\sim}=(\lambda,\mrdet(g))_{n}^{(2\bs{c}+\bs{d})r-\bs{d}}.
			\end{equation}
			\item Let $r=r_{1}+\dots+r_{k}$ and let $E_{i}/F$ be a field extension of degree $d_{i}$ such that $r_{i}=d_{i}r_{i}'^{2}$ for each $i$. Fix an $F$-algebra embedding $\bigoplus_{i=1}^{k}E_{i}\hookrightarrow\mrm_{r}(F)$. 
			\begin{enumerate}
				\item If $d_{i}=r_{i}$ for each $i$, then for $u=(u_{1},\dots,u_{k}), v=(v_{1},\dots,v_{k})\in \bigoplus_{i=1}^{k}E_{i}^{\times}$, we have
				\begin{equation}\label{eqsumEicommutator}
					[u,v]_{\sim}=\bigg[\prod_{i=1}^{k}(u_{i},v_{i})_{n,E_{i}}^{-\bs{d}}\bigg]\cdot(\mrdet_{F}(u),\mrdet_{F}(v))_{n,F}^{2\bs{c}+\bs{d}}.
				\end{equation}
				\item In general, consider the centralizer of $\bigoplus_{i=1}^{k}E_{i}$ in $\mrm_{r}(F)$, which induces an $F$-algebra embedding $\mrm_{r_{1}'}(E_{1})\times\dots\times\mrm_{r_{k}'}(E_{k})\hookrightarrow\mrm_{r}(F)$. Then for $u=(u_{1},\dots,u_{k})\in \bigoplus_{i=1}^{k}E_{i}^{\times}$ and $v=(v_{1},\dots,v_{k})\in \mrgl_{r_{1}'}(E_{1})\times\dots\times\mrgl_{r_{k}'}(E_{k})$, we have
				\begin{equation}\label{eqsumEiMEicommutator}
					\begin{aligned}
						[u,v]_{\sim}&=\bigg[\prod_{i=1}^{k}(u_{i},\mrdet_{E_{i}}(v_{i}))_{n,E_{i}}^{-\bs{d}}\bigg]\cdot(\mrdet_{F}(u),\mrdet_{F}(v))_{n,F}^{2\bs{c}+\bs{d}}\\
						&=\prod_{i=1}^{k}(\mrdet_{F}(u)^{2\bs{c}+\bs{d}}u_{i}^{-\bs{d}},\mrdet_{E_{i}}(v_{i}))_{n,E_{i}}.
					\end{aligned}
				\end{equation}
			\end{enumerate}
			\item The 2-cocycle $\sigma$ is trivial on $\mrgl_{r}(\mfo_{F})\times\mrgl_{r}(\mfo_{F})$. Thus there exists a splitting $\mrgl_{r}(\mfo_{F})\rightarrow\wt{G}$, and any two splittings differ by a character of $\mfo_{F}^{\times}$ composing with the determinant map. In general, for any open compact subgroup $K$ of $\mrgl_{r}(F)$, there exists a splitting $K\rightarrow\wt{G}$.
			
		\end{enumerate}
		
	\end{proposition}
	
	\begin{proof}
		
		Statement (1) is direct. 
		
		Statement (2)(3)(4) follow from \cite{kaplan2022classification}*{Lemma 4.1} and statement (5a) follows from \cite{kazhdan1984metaplectic}*{Proposition 0.1.5} in the case $\bs{d}=1$. The general cases follow easily from the fact that $\sigma=\sigma_{\mathrm{KP}}^{\bs{d}}\cdot\sigma_{\mrdet}^{\bs{c}}$ and a direct calculation (\emph{cf.} \eqref{eqformcommutator}). 
		
		For statement (5b), we notice that $[u,\cdot]_{\sim}$ is a character of $\mrgl_{r_{1}'}(E_{1})\times\dots\times\mrgl_{r_{k}'}(E_{k})$, which is trivial on the derived subgroup $\mrsl_{r_{1}'}(E_{1})\times\dots\times\mrsl_{r_{k}'}(E_{k})$. So we only need to verify the formula \eqref{eqsumEiMEicommutator} for those $v_{i}$ diagonal in $\mrgl_{r_{i}'}(E_{i})$ for each $i$. Then we may use \eqref{eqsumEicommutator} with respect to an $F$-algebra embedding $$\underbrace{E_{1}\times\dots\times E_{1}}_{r_{1}'\text{-copies}}\times\dots\times \underbrace{E_{k}\times\dots\times E_{k}}_{r_{k}'\text{-copies}}\hookrightarrow \mrm_{r}(F).$$
		
		The first part of statement (6) follows from \cite{kazhdan1984metaplectic}*{Proposition 0.1.2} and the fact that $(\cdot,\cdot)_{n,F}$ is trivial on $\mfo_{F}^{\times}\times\mfo_{F}^{\times}$ when $\mrgcd(q,n)=1$. The second part is trivial. In the final part, since $K$ and $\mrgl_{r}(\mfo_{F})$ are conjugate by some $g\in G$, we simply take the conjugation of the corresponding splitting of $\mrgl_{r}(\mfo_{F})$ as a splitting of $K$.
		
	\end{proof}
	
	From now on, we fix $r\geq 1$, $(\bs{c},\bs{d})\in\mbz\times\mbz$ and $\sigma$ as above. We let $\wt{G}$ be the $n$-fold metaplectic cover of $G=\mrgl_{r}(F)$ corresponding to $\sigma$. For a closed subgroup $H$ of $G$, we define $$H^{(n)}=\{h\in H\mid\mrdet(h)\in F^{\times n}\}.$$
	
	Let $P=MN$ be a parabolic subgroup of $G$ with a Levi factor $M$ and the unipotent radical $N$, where we assume that $M$ is isomorphic to $\mrgl_{r_{1}}(F)\times\dots\times\mrgl_{r_{k}}(F)$ with $r=r_{1}+\dots+r_{k}$. Since $N$ is a pro-$p$-group, there exists a unique splitting of $N$ into $\wt{G}$. Thus we may realize $N$ as a subgroup of $\wt{G}$, which we still denote by $N$. So $\wt{P}=\wt{M}N$
	is a parabolic subgroup of $\wt{G}$ with a Levi factor $\wt{M}$ and the unipotent radical $N$. 
	
	
	Let $H=H_{1}\times\dots\times H_{k}$ be a subgroup of $M$, such that each $H_{i}$ is a closed subgroup of $\mrgl_{r_{i}}(F)$. We call $H$ \emph{block compatible} if $[H_{i},H_{j}]_{\sim}=\{1\}$ for any  $1\leq i<j\leq k$. We notice that $H$ is block compatible in the following two cases (\emph{cf.} \eqref{eqBDcommutator}):
	\begin{itemize}
		\item For any $i$, we have $H_{i}=H_{i}^{(n)}$;
		\item For any $i$, the determinant of every element in $H_{i}$ is in $\mfo_{F}^{\times}$.
	\end{itemize}
	This concept becomes important when considering (exterior) tensor product of genuine representations. Let $\wt{\rho}_{i}$ be genuine representations of $\wt{H_{i}}$ for each $i$. We take the tensor product $\wt{\rho}_{1}\boxtimes\dots\boxtimes\wt{\rho}_{k}$ as a representation of $\wt{H_{1}}\times\dots\times\wt{H_{k}}$, which is trivial on $$\Xi=\{(\zeta_{1},\dots,\zeta_{k})\in\mu_{n}\times\dots\times\mu_{n}\mid \zeta_{1}\dots\zeta_{k}=1\}.$$
	If $H$ is block compatible, then it is clear that $$\wt{H}\cong\wt{H_{1}}\times\dots\times\wt{H_{k}}/\Xi,$$
	so we realize $\wt{\rho}_{1}\boxtimes\dots\boxtimes\wt{\rho}_{k}$  as a genuine representation of $\wt{H}$. 
	
	In particular, we emphasize the following three special classes of covers:
	
	\begin{itemize}
		
		\item (Determinantal covers) When $\bs{d}=0$, we get determinantal covers. We remark that such a cover is not far from the corresponding linear group. Since for any closed subgroup $H$ of $G$, the group $\wt{H^{(n)}}$ is isomorphic to $H^{(n)}\times \mu_{n}$. Moreover $H/H^{(n)}$ is of finite index. So the representation theory of $\wt{G}$ is easily deduced from that of $G$.
		
		\item (The Kazhdan-Patterson covers) When $\bs{d}=1$, we get Kazhdan-Patterson covers (KP-covers for short) (\emph{cf.} \cite{kazhdan1984metaplectic}) Such covers are natural, in the sense that they could be regarded as the Baer sum of the pull-back of the canonical cover of $\mrsl_{r+1}(F)$ and a determinantal cover. Moreover, there exists a functorial lift from representations of a KP-cover to that of $G$, see for instance \cite{flicker1986metaplectic},  \cite{zou2022metaplectic}, etc.
		
		\item (Savin's cover) When $\bs{c}=-1$ and $\bs{d}=2$, we get a special cover constructed by Gordan Savin (S-cover for short), see for instance \cite{gao2019whittaker}*{\S 4.1}. It can be realized as follows: first we construct the canonical cover of the symplectic group $\mrsp_{2r}(F)$ with respect to the Steinberg symbol $(\cdot,\cdot)_{n}^{-1}$, and then after identifying $G$ with the Siegel Levi subgroup of $\mrsp_{2r}(F)$ we get the S-cover of $G$. 
		Every Levi subgroup of $G$ is block compatible, which can be seen from formula \eqref{eqBDcommutator} since $2\bs{c}+\bs{d}=0$.
		
	\end{itemize} 
	
	\section{Strata and simple characters}\label{sectionstrata}
	
	In this section, we recall the theory of strata and simple characters related to $G=\mrgl_{r}(F)$, developed by Bushnell and Kutzko. Our main reference is \cite{bushnell129admissible}, \cite{bushnell1999semisimple} and \cite{secherre2008representations}. We also recommend \cite{bushnell2019arithmetic} as an excellent survey. 
	
	\subsection{Lattice sequences and lattice chains}
	
	Fix an $r$-dimensional vector space $V$ over $F$, and a basis $\{v_{1},\dots,v_{r}\}$ of $V$ under which we identify $G$ with $\mraut_{F}(V)$. We write $A=\mrend_{F}(V)\cong\mrm_{r}(F)$. For $\beta\in A$, we define $\psi_{\beta}(x)=\psi_{F}(\mrtr(\beta(x-1)))$ for $x\in A$.
	
	An \emph{$\mfo_{F}$-lattice sequence} $\Lambda=(\Lambda_{k})_{k\in\mbz}$ of $V$ is a sequence of $\mfo_{F}$-lattices of $V$, such that
	\begin{itemize}
		\item $\Lambda_{k+1}\subset\Lambda_{k}$ for any $k\in\mbz$.
		\item There exists a positive integer $e=e(\Lambda|\mfo_{F})$, called the period of $\Lambda$ over $\mfo_{F}$, such that $\Lambda_{k+e}=\mfp_{F}\Lambda_{k}$ for any $k\in\mbz$.
	\end{itemize} 
	When $\Lambda_{k+1}\subsetneq\Lambda_{k}$ for each $k\in\mbz$, we call $\Lambda$ a \emph{strict} lattice sequence, or a \emph{lattice chain} of $V$. We denote by $\msl(V,\mfo_{F})$ (resp. $\msl^{+}(V,\mfo_{F})$) the set of $\mfo_{F}$-lattice sequences (resp. $\mfo_{F}$-lattice chains) of $V$.
	
	We may realize a lattice sequence $\Lambda$ as a function defined on $\mbr$ by imposing $\Lambda_{x}=\Lambda_{\ceil{x}}$ for $x\in\mbr$.
	
	For $\Lambda\in\msl(V,\mfo_{F})$, we define a lattice sequence $\mfA(\Lambda)\in\msl(A,\mfo_{F})$ by
	$$\mfA_{k}(\Lambda)=\{a\in A\mid a\Lambda_{l}\subset\Lambda_{l+k},\ l\in\mbz \},\quad k\in\mbz.$$
	In particular, $\mfa=\mfA_{0}(\Lambda)$ is a hereditary order in $A$, and $\mfp_{\mfa}=\mfA_{1}(\Lambda)$ is the Jacobson radical of $\mfa$. 
	We define the valuation map $$v_{\Lambda}:A\rightarrow \mbz,\ x\mapsto\mathrm{max}\{k\in\mbz\mid x\in\mfA_{k}(\Lambda)\}$$ with the convention $v_{\Lambda}(0)=\infty$. We write $U(\Lambda)=\mfA_{0}^{\times}$ and $U_{k}(\Lambda)=1+\mfA_{k}(\Lambda)$ for $k\geq 1$. 
	
	Up to the choice of an $F$-basis of $V$, we may realize $\mfa$ as a standard hereditary order in $\mrm_{r}(F)$, meaning that there exists a composition $r_{1}+\dots+r_{t}=r$ such that
	$$\mfa=\{(a_{ij})_{1\leq i,j\leq t}\mid a_{ij}\in\mrm_{r_{i}\times r_{j}}(\mfo_{F})\ \text{for}\ 1\leq i\leq j\leq t\ \text{and}\ a_{ij}\in\mrm_{r_{i}\times r_{j}}(\mfp_{F})\ \text{for}\ 1\leq j< i\leq t\}.$$
	
	If $\Lambda$ is strict, then $\mfA_{k}(\Lambda)=\mfp_{\mfa}^{k}$ for any $k\in\mbz$. Indeed, it is not hard to show that $\msl^{+}(V,\mfo_{F})\rightarrow \msl^{+}(A,\mfo_{F}),\ \Lambda\rightarrow \mfA(\Lambda)$ is a bijection, so we may somehow recover the corresponding lattice chain from a given hereditary order $\mfa$ in $A$.
	
	Let $E$ be a subfield of $A$ over $F$. Then we may realize $V$ as a vector space over $E$, which we denote by $V_{E}$. Let $B=\mrend_{E}(V)$, which is  an $F$-subalgebra of $A$. Then $B^{\times}$ is the centralizer of $E^{\times}$ in $G$. Let $\psi_{E}$ be a character of $E$ of conductor $\mfp_{E}$.  Then there exists a bi-$B$-module homomorphism $s:A\rightarrow B$ such that 
	$$\psi_{E}(\mrtr_{E}(s(a)b))=\psi_{F}(\mrtr_{F}(ab)),\quad a\in A,\ b\in B,$$
	called a \emph{tame corestriction} on $A$.
	
	A sequence $\Lambda\in\msl(V,\mfo_{F})$ is called \emph{$E$-pure} if it is normalized by $E^{\times}$. In this case, each $\Lambda_{k}$ is an $\mfo_{E}$-lattice. Moreover, $\Lambda$ could be realized as an element in $\msl(V_{E},\mfo_{E})$, which we denote by $\Lambda_{E}$. By definition,  $\mfA_{k}(\Lambda_{E})=\mfA_{k}(\Lambda)\cap B$ for each $k\in\mbz$. Let $\mfb=\mfA_{0}(\Lambda_{E})$, which is a hereditary order in $B$. Let $\mfp_{\mfb}=\mfA_{1}(\Lambda_{E})$.
	
	We consider a decomposition $V=\bigoplus_{i=1}^{t} V^{i}$ of $F$-vector spaces. It is called an \emph{$E$-decomposition} if each $V^{i}$ is $E$-stable. In this case,  
	we have $V_{E}=\oplus_{i=1}^{t}V_{E}^{i}$. 
	
	Let $M=\prod_{i=1}^{t}\mraut_{F}(V^{i})\subset G$, let $P=MN$ be a parabolic subgroup of $G$ having a Levi factor $M$ and the unipotent radical $N$. For an $E$-decomposition $V=\bigoplus_{i=1}^{t} V^{i}$, we get a parabolic subgroup $P_{E}=P\cap B^{\times}$ of $B^{\times}$, having a Levi factor $M_{E}=M\cap B^{\times}$ and the unipotent radical $N_{E}=N\cap B^{\times}$.
	
	Let $\Lambda^{i}\in\msl(V^{i},\mfo_{F})$ for $i=1,\dots t$. Then the direct sum $\Lambda=\bigoplus_{i=1}^{t}\Lambda^{i}$ is well-defined as an element in $\msl(V,\mfo_{F})$. On the other hand, we say that a decomposition $V=\bigoplus_{i=1}^{t} V^{i}$ \emph{conforms with} $\Lambda\in\msl(V,\mfo_{F})$, if there exists $\Lambda^{i}\in\msl(V^{i},\mfo_{F})$  for each $i$ such that $\Lambda=\bigoplus_{i=1}^{t}\Lambda^{i}$. In this case, we necessarily have $\Lambda^{i}=\Lambda\cap V^{i}$. Furthermore, if $V=\bigoplus_{i=1}^{t} V^{i}$ is an $E$-decomposition and each $\Lambda^{i}$ is $E$-pure, we get the decomposition $\Lambda_{E}=\bigoplus_{i=1}^{t}\Lambda_{E}^{i}$. 
	
	\subsection{Strata}
	
	A \emph{stratum} in $A$ is a $4$-tuple $[\Lambda,u,l,\beta]$, where $\Lambda\in\msl(V,\mfo_{F})$, and $u,l$ are integers such that $0\leq l\leq u-1$, and $\beta\in\mfA_{-u}(\Lambda)$.
	
	If the stratum we consider is strict, saying that $\Lambda$ is strict, then conventionally we often write $[\mfa,u,l,\beta]$ for the same stratum to emphasize this assumption, where $\mfa=\mfA_{0}(\Lambda)$. Conventionally, if an object $\Gamma(\Lambda,\beta,\dots)$ is defined for a stratum with $\Lambda$ strict, we often write $\Gamma(\mfa,\beta,\dots)$ instead.
	
	Two strata $[\Lambda,u,l,\beta_{1}]$ and $[\Lambda,u,l,\beta_{2}]$ are called \emph{equivalent} if $\beta_{1}-\beta_{2}\in\mfA_{-l}(\Lambda)$. 
	
	For a stratum $[\Lambda,u,l,\beta]$ in $A$ such that $0\leq l<u\leq 2l+1$, we notice that $\psi_{\beta}$ becomes a character of $U^{l+1}(\Lambda)/U^{u+1}(\Lambda)$, which depends only on the equivalence class of $[\Lambda,u,l,\beta]$.
	
	Assume that $[\Lambda,u,l,\beta]$ is \emph{pure}, which means that $E:=F[\beta]$ is a field and $\Lambda$ is $E$-pure. In this case, we let $B$ be the centralizer of $E$ in $A$. Once a related tame corestriction $s:A\rightarrow B$ is fixed, a \emph{derived stratum} of $[\Lambda,u,l,\beta]$ is a stratum $[\Lambda_{E},l,l-1,s(c)]$ in $B$ with $c\in\mfA_{-l}(\Lambda)$.
	
	For $k\in\mbz$, we define
	$$\mfn_{k}(\beta,\Lambda)=\{x\in\mfa\mid \beta x-x\beta\in \mfA_{k}(\Lambda)\}.$$ 
	If $E\neq F$, there exists a maximal integer $k$ such that $\mfn_{k}(\beta,\Lambda)$ is not contained in $\mfb+\mfp_{\mfa}$, which we denote by $k_{0}(\beta,\Lambda)$. Then $k_{0}(\beta,\Lambda)\geq -u$ since $\mfn_{-u}(\beta,\Lambda)=\mfa$. If $E=F$, we let $k_{0}(\beta,\Lambda)=-\infty$. 
	Then we call a pure stratum $[\Lambda,u,l,\beta]$ \emph{simple} if $l<-k_{0}(\beta,\Lambda)$. 
	
	Assume that we have an $E$-decomposition $V=\bigoplus_{i=1}^{t}V^{i}$. Fix a simple stratum $[\Lambda,u,l,\beta]$ in $A$ and assume that $\Lambda$ conforms with this decomposition. Then each $[\Lambda^{i},u,l,\beta]$ is a simple stratum in $A^{i}=\mrend_{F}(V^{i})$ for $i=1,\dots,t$ (\emph{cf.} \cite{secherre2008representations}*{Proposition 1.20}). 
	
	In particular, we consider a stratum of the form $[\Lambda,u,u-1,b]$. Let $e=e(\Lambda|\mfo_{F})$. The element $\varpi_{F}^{u/\mrgcd(u,e)}b^{e/\mrgcd(u,e)}$ is in $\mfa^{\times}$. So its characteristic polynomial has coefficients in $\mfo_{F}$, whose reduction module $\mfp_{F}$ in $\bs{k}[X]$ is called the characteristic polynomial of $[\Lambda,u,u-1,b]$ and denoted by $\varphi_{b}$. 
	
	Then a stratum $[\Lambda,u,u-1,b]$ is called
	\begin{itemize}
		\item \emph{fundamental} if $\varphi_{b}$ is not a power of $X$.
		\item \emph{split} if $\varphi_{b}$ has at least two different irreducible factors.
		\item \emph{minimal} if it is simple, or in other words $k_{0}(b,\Lambda)=-u$. 
	\end{itemize}
	
	Finally, we also allow the occurrence of a ``null" stratum $[\Lambda,0,0,\beta]$, where $\Lambda$ is a lattice chain such that the corresponding hereditary order $\mfa$ is maximal in $A$, and $\beta\in \mfo_{F}$. In this case we have $E=F$ and $A=B$.
	
	\subsection{Simple characters}\label{subsectionsimplecharacters}
	
	Let $[\Lambda,u,0,\beta]$ be a simple stratum in $A=\mrend_{F}(V)$ and let $E=F[\beta]$. We first assume $\Lambda$, as well as the occurring lattice sequences, to be strict. As in \cite{bushnell129admissible}*{\S 3} and \cite{secherre2008representations}*{\S  2}, 
	\begin{enumerate}
		\item We define two sub-$\mfo_{F}$-orders $\mfH(\beta,\Lambda)\subset\mfJ(\beta,\Lambda)$ of $\mfa$, which depend only on the equivalence class of $[\Lambda,u,0,\beta]$. Each of them is filtered by a sequence of bilateral ideals correspondingly:
		$$\mfH^{k}(\beta,\Lambda)=\mfH(\beta,\Lambda)\cap\mfA_{k}(\Lambda),\quad \mfJ^{k}(\beta,\Lambda)=\mfJ(\beta,\Lambda)\cap\mfA_{k}(\Lambda),\quad k\geq 1.$$
		We denote by $H(\beta,\Lambda)$ (resp. $J(\beta,\Lambda)$) the subgroup of invertible elements in $\mfH(\beta,\Lambda)$ (resp. $\mfJ(\beta,\Lambda)$), then similarly each of them is filtered by a sequence of open compact subgroups
		$$H^{k}(\beta,\Lambda)=H(\beta,\Lambda)\cap U^{k}(\Lambda),\quad J^{k}(\beta,\Lambda)=J(\beta,\Lambda)\cap U^{k}(\Lambda),\quad k\geq 1.$$
		We also have $J(\beta,\Lambda)=U(\Lambda_{E})J^{1}(\beta,\Lambda)$.
		
		\item For each $0\leq l\leq k_{0}(\beta,\Lambda)-1$, we define a finite set $\mcc(\Lambda,l,\beta)$ of characters of $H^{l+1}(\beta,\Lambda)$, called \emph{simple characters} of level $l$ attached to $[\Lambda,u,0,\beta]$. We write $\mcc(\Lambda,\beta)=\mcc(\Lambda,0,\beta)$ for short. We remark that $\mcc(\Lambda,l,\beta)$ depends on the choice of $\psi_{F}$. 
		
		\item Let $v=-k_{0}(\beta,\Lambda)>0$. For $k\geq 1$, we define
		\begin{equation}\label{eqomegadef}
			\mfm_{k}(\beta,\Lambda)=\mfA_{k}(\Lambda)\cap\mfn_{k-v}(\beta,\Lambda)+\mfJ^{\ceil{v/2}}(\beta,\Lambda)\quad\text{and}\quad\Omega_{k}(\beta,\Lambda)=1+\mfm_{k}(\beta,\Lambda).
		\end{equation}
		Then for $\theta\in\mcc(\Lambda,l,\beta)$ we have
		\begin{equation}\label{eqsimcharintertwine}
			I_{G}(\theta)=\Omega_{v-l}(\beta,\Lambda)B^{\times}\Omega_{v-l}(\beta,\Lambda).
		\end{equation}
		In particular, if $l=0$ we have
		\begin{equation}\label{eqsimcharintertwinelevel0}
			I_{G}(\theta)=J(\beta,\Lambda)B^{\times}J(\beta,\Lambda).
		\end{equation}
		
		\item Let $V$, $V'$ be two finite dimensional $F$-vector spaces, let $A=\mrend_{F}(V)$ and $A'=\mrend_{F}(V')$, let $E=F[\beta]$ and $E'=F[\beta']$ be subfields of $A$ and $A'$ respectively with an isomorphism $E\rightarrow E'$ mapping $\beta$ to $\beta'$. Let  $[\Lambda,u,0,\beta]$ (resp. $[\Lambda',u',0,\beta']$) be a strict simple stratum in $A$ (resp. $A'$). Then we have a \emph{transfer} map as a bijection
		$$t_{\mfa,\mfa',\beta,\beta'}:\mcc(\Lambda,l,\beta)\rightarrow\mcc(\Lambda',l',\beta'),$$
		where we assume $0\leq l\leq u-1$, $0\leq l\leq u'-1$ and $\floor{l/e(\Lambda|\mfo_{E})}=\floor{l'/e(\Lambda'|\mfo_{E})}$. 
		
		In particular, in the following cases, this transfer map could be further characterized.
		\begin{enumerate}
			\item Assume $V=V'$, $l=l'$, and there exists $g\in\mraut_{F}(V)$ such that $\beta'=\beta^{g}$ and $\mfa'=\mfa^{g}$. Then the transfer map is given by the $g$-conjugation.
			
			\item Assume $V=V'$ and $\beta=\beta'$. Then for $\theta\in\mcc(\Lambda,l,\beta)$, the transfer of $\theta$ to $\mcc(\Lambda',l',\beta)$ coincides with $\theta$ on $H^{l+1}(\beta,\Lambda)\cap H^{l'+1}(\beta,\Lambda')$.
			
			\item Assume that $V=V'\oplus V''$ is an $E$-decomposition, and there exists an $E$-pure lattice chain $\Lambda''$ in $V''$ such that $\Lambda=\Lambda'\oplus\Lambda''$. Then for any $l\leq u-1$ we have $H^{l}(\beta,\Lambda)\cap \mraut_{F}(V')=H^{l}(\beta',\Lambda')$.  Moreover, for $l'=l$ the transfer map is given by the restriction to $H^{l}(\beta',\Lambda')$.
			
		\end{enumerate}
		
		\item Let $V=\bigoplus_{i=1}^{t}V^{i}$ be an $E$-decomposition with $E=F[\beta]$. Let $\Lambda^{i}$ be an $E$-pure lattice chain\footnote{Be aware that in general $\Lambda^{i}=\Lambda\cap V^{i}$ is not necessarily strict for $i=1,\dots,t$, although right now we assume that it is so to define $H^{l+1}(\beta,\Lambda^{i})$ and $\theta_{i}\in\mcc(\Lambda^{i},l,\beta)$. We will explain in the next paragraph how to deal with the general case.} in $V^{i}$ for each $i$, such that $\Lambda=\bigoplus_{i=1}^{t}\Lambda^{i}$. Then each $[\Lambda^{i},u,0,\beta]$ is a simple stratum in $A^{i}=\mrend_{F}(V^{i})$. Let $M=\prod_{i=1}^{t}\mraut_{F}(V^{i})$ be the corresponding Levi subgroup of $G$, and let $P=MN$ be a parabolic subgroup of $G$ having a Levi factor $M$. Then for $\theta\in\mcc(\Lambda,l,\beta)$,
		\begin{enumerate}
			\item  The pair $(H^{l+1}(\beta,\Lambda),\theta)$ is decomposed with respect to $(M,P')$ for any parabolic subgroup $P'$ of $G$ with a Levi factor $M$ (\emph{cf.} \S \ref{subsectioncoverpair}). 
			
			\item Moreover, we have
			\begin{equation}\label{eqthetadecom}
				H^{l+1}(\beta,\Lambda)\cap M=\prod_{i=1}^{t}H^{l+1}(\beta,\Lambda^{i})\quad\text{and}\quad \theta_{M}:=\theta\rest_{H^{l+1}(\beta,\Lambda)\cap M}=\theta_{1}\otimes\dots\otimes\theta_{t},
			\end{equation}
			with each $\theta_{i}$ being the transfer of $\theta$ to $\mcc(\Lambda^{i},l,\beta)$.
		\end{enumerate}

	\end{enumerate}
	
	Now we consider a general lattice sequence $\Lambda$, and we explain how to generalize the above theory in this case. We refer to \cite{secherre2008representations}*{\S 2} for the missing details.
	
	We construct an $E$-decomposition of $F$-vector spaces $V''=V\oplus V'$. Write $A''=\mrend_{F}(V'')$, which contains $A=\mrend_{F}(V)$ as a sub-$F$-algebra. For $e=e(\Lambda|\mfo_{F})$, we choose a strict $E$-pure lattice sequence $\Lambda'$ in $V'$ of period $e$. Then
	the $E$-pure lattice sequence $\Lambda''=\Lambda\oplus\Lambda'$ in $V''$ is also strict and of period $e$. Moreover, $[\Lambda'',u,0,\beta]$ (after identifying $\beta$ with an element in $A''$) is a simple stratum in $A''$.
	
	We define $$\mfH(\beta,\Lambda)=\mfH(\beta,\Lambda'')\cap A\quad\text{and}\quad\mfJ(\beta,\Lambda)=\mfJ(\beta,\Lambda'')\cap A$$ 
	as $\mfo_{F}$-orders in $A$. Then, we define $\mfH^{k}(\beta,\Lambda)$, $\mfJ^{k}(\beta,\Lambda)$, $H(\beta,\Lambda)$, $J(\beta,\Lambda)$, $H^{k}(\beta,\Lambda)$, $J^{k}(\beta,\Lambda)$, $\mfm_{k}(\beta,\Lambda)$, $\Omega_{k}(\beta,\Lambda)$ as above.
	
	We have $H^{l+1}(\beta,\Lambda'')\cap A=H^{l+1}(\beta,\Lambda)$ for any $l\geq 0$. Moreover, for $0\leq l<-k_{0}(\beta,\Lambda'')$, we define simple characters of level $l$ attached to $[\Lambda,u,0,\beta]$ to be the restriction of characters in $\mcc(\Lambda'',l,\beta)$ to $H^{l+1}(\beta,\Lambda)$. We denote by $\mcc(\Lambda,l,\beta)$ the set of such characters. 
	
	We remark that the above notation is well-defined, i.e. independent of the choice of $V'$, $V''$ and $\Lambda'$, and it is compatible with the original theory when $\Lambda$ is strict. Moreover, all the results for simple characters listed above remain valid.
	
	For any simple stratum $[\Lambda,u,0,\beta]$ in $A$, we may define a related strict simple stratum $[\mfa,u',0,\beta]$ in $A$, where $\mfa=\mfA_{0}(\Lambda)$ and $u'=-v_{\mfa}(\beta)$. 
	
	We also briefly recall the concept ``endo-class" introduced by Bushnell and Henniart \cite{bushnell1996local}, although it will not be essentially used in this article. Let $V_{1}$, $V_{2}$ be two finite dimensional vector spaces over $F$. Let $[\mfa_{1},u_{1},0,\beta_{1}]$ (resp. $[\mfa_{2},u_{2},0,\beta_{2}]$) be a strict simple stratum in $A_{1}=\mrend_{F}(V_{1})$ (resp. $A_{2}=\mrend_{F}(V_{2})$). We call  $\theta_{1}\in\mcc(\mfa_{1},0,\beta_{1})$ and $\theta_{2}\in\mcc(\mfa_{2},0,\beta_{2})$ \emph{endo-equivalent} if 
	there exist an $F$-vector space $V$ containing $V_{1}$ and $V_{2}$, and two simple strata $[\mfa_{1}',u_{1},0,\beta_{1}']$ and $[\mfa_{2}',u_{2},0,\beta_{2}']$ in $\mrend_{F}(V)$, with $F[\beta_{1}]\cong F[\beta']$ and $F[\beta_{2}]\cong F[\beta_{2}']$ mapping $\beta_{1}$ to $\beta_{1}'$ and $\beta_{2}$ to $\beta_{2}'$, such that the transfers $t_{\mfa_{1},\mfa_{1}',\beta_{1},\beta_{1}'}(\theta_{1})$ and $t_{\mfa_{2},\mfa_{2}',\beta_{2},\beta_{2}'}(\theta_{2})$ are conjugate by $G=\mraut_{F}(V)$.
	
	We notice that such an equivalence relation is independent of various choices above. Thus we introduced an equivalence relation on 
	$$\bigcup_{[\mfa,u,0,\beta]}\mcc(\mfa,0,\beta)$$
	where the union ranges over all the strata in $\mrm_{r}(F)$ with $r$ ranging over all positive integers. Such an equivalence class is called an \emph{endo-class}.
	
	Finally, for a ``null" stratum $[\Lambda,0,0,\beta]$ in $A$ and $\beta\in\mfo_{F}$, by convention we have $\mfH(\beta,\mfa)=\mfJ(\beta,\mfa)=\mfa$, $\mfH^{i}(\beta,\mfa)=\mfJ^{i}(\beta,\mfa)=\mfp_{\mfa}^{i}$, $H(\beta,\mfa)=J(\beta,\mfa)=U(\mfa)$, $H^{i}(\beta,\mfa)=J^{i}(\beta,\mfa)=U^{i}(\mfa)$ and $\mcc(\mfa,\beta)=\{1\}$. In this case, the corresponding endo-class is trivial.
	
	\subsection{Heisenberg representations and $\beta$-extensions}\label{subsectionHeisenbergbeta}
	
	The reference here is \cite{bushnell129admissible}*{\S 5.1, 5.2, 7.1, 7.2} and \cite{secherre2005representations}. Let $[\mfa,u,0,\beta]$ be a strict simple stratum in $A=\mrend_{F}(V)$, where we denote by $\Lambda$ the lattice chain related to $\mfa$. Still we write $E=F[\beta]$ and $B=\mrend_{E}(V_{E})$. Let $\theta\in\mcc(\mfa,\beta)$. 
	
	There exists a unique irreducible representation $\eta$ of $J^{1}(\beta,\mfa)$, called the \emph{Heisenberg} representation of $\theta$, whose restriction to $H^{1}(\beta,\mfa)$ contains $\theta$. The restriction of $\eta$ to $H^{1}(\beta,\mfa)$ is isomorphic to $[J^{1}(\beta,\mfa):H^{1}(\beta,\mfa)]^{1/2}\cdot\theta$. Also we have
	\begin{equation}\label{eqintertwineeta}
		I_{G}(\eta)=I_{G}(\theta)=J(\beta,\mfa)B^{\times}J(\beta,\mfa).
	\end{equation} 
	Moreover, each related intertwining space of $\eta$ is of dimension $1$.
	
	A representation $\kappa$ of $J(\beta,\mfa)=U(\mfb)J^{1}(\beta,\mfa)$ is called a \emph{$\beta$-extension} of $\eta$ if its restriction to $J^{1}(\beta,\mfa)$ is $\eta$, and moreover $I_{G}(\kappa)=I_{G}(\eta)$. We notice that such $\kappa$ exists, and all the other $\beta$-extensions of $\eta$ are of the form $\kappa\chi$, where $\chi$ ranges over characters of $J(\beta,\mfa)/J^{1}(\beta,\mfa)\cong U(\mfb)/U^{1}(\mfb)$ that factor through $\mrdet_{E}:B^{\times}\rightarrow E^{\times}$.
	
	We define $\bs{J}(\beta,\mfa)=N_{B^{\times}}(U(\mfb))J(\beta,\mfa)$ as a compact modulo center subgroup of $G$. In particular, if $\mfb$ is a maximal hereditary order, then $\bs{J}(\beta,\mfa)=E^{\times}J(\beta,\mfa)$. Notice that $\bs{J}(\beta,\mfa)$ is indeed the normalizer of the simple character $\theta$, thus it does not depend on the choice of the simple stratum $[\mfa,u,0,\beta]$ but only on $\theta$.
	
	Now we consider $E$-pure lattice chains $\Lambda'$ and $\Lambda''$ in $A$, such that the corresponding hereditary orders $\mfa'$ and $\mfa''$ satisfy $\mfa'\subset\mfa\subset\mfa''$. Let $u'=-v_{\mfa'}(\beta)$ and $u''=-v_{\mfa''}(\beta)$ be two positive integers. Then $[\mfa',u',0,\beta]$ and $[\mfa'',u'',0,\beta]$ are simple strata in $A$. Let $\theta'\in\mcc(\mfa',\beta)$ and $\theta''\in\mcc(\mfa'',\beta)$ be the corresponding transfers of $\theta$, and let $\eta'$ and $\eta''$ be the Heisenberg representations of $\theta'$ and $\theta''$ respectively. Then, there exist unique $\beta$-extensions $\kappa'$ of $\eta'$ and  $\kappa''$ of $\eta''$, such that 
	\begin{equation}\label{eqrelatedbetaext}
		\begin{aligned}
			&\mrind_{U(\mfb')J^{1}(\beta,\mfa)}^{U(\mfb')U^{1}(\mfa')}(\kappa\rest_{U(\mfb')J^{1}(\beta,\mfa)})\cong\mrind_{J(\beta,\mfa')}^{U(\mfb')U^{1}(\mfa')}(\kappa')\\
			\text{and}\quad&\mrind_{U(\mfb)J^{1}(\beta,\mfa'')}^{U(\mfb)U^{1}(\mfa)}(\kappa''\rest_{U(\mfb)J^{1}(\beta,\mfa'')})\cong\mrind_{J(\beta,\mfa)}^{U(\mfb)U^{1}(\mfa)}(\kappa)
		\end{aligned}
	\end{equation}
	as irreducible representations (\emph{cf.} \cite{bushnell129admissible}*{\S 5.2.14}). We call $\kappa'$ (resp. $\kappa''$) the $\beta$-extension related to $\kappa$. 
	
	Finally we consider an $E$-decomposition $V=\bigoplus_{i=1}^{t}V^{i}$ which conforms with $\Lambda$. Let $M=\prod_{i=1}^{t}\mraut_{F}(V^{i})$ be the corresponding Levi subgroup of $G$. We fix a parabolic subgroup $P=MN$ of $G$ having a Levi factor $M$. 
	
	Write $\Lambda^{i}=\Lambda\cap V^{i}$, $\mfa^{i}=\mfA_{0}(\Lambda^{i})$ and $\mfb^{i}=\mfA_{0}(\Lambda_E^{i})$ for $i=1,\dots,t$.
	
	Let $\theta_{i}\in\mcc(\Lambda^{i},\beta)$ be defined as in \eqref{eqthetadecom}. Let $\eta_{i}$ be the Heisenberg representation of $\theta_{i}$ as a representation of $J^{1}(\beta,\Lambda^{i})$, and let $\eta_{M}=\eta_{1}\boxtimes\dots\boxtimes\eta_{t}$ be an irreducible representation of $$J^{1}_{M}(\beta,\mfa):=J^{1}(\beta,\mfa)\cap M=J^{1}(\beta,\Lambda^{1})\times\dots\times J^{1}(\beta,\Lambda^{t}).$$ 
	We define 
	\begin{equation}\label{eqJ1P}
		J^{1}_{P}(\beta,\mfa)=H^{1}(\beta,\mfa)(J^{1}(\beta,\mfa)\cap P)=(H^{1}(\beta,\mfa)\cap N^{-})\cdot(J^{1}(\beta,\mfa)\cap M)\cdot(J^{1}(\beta,\mfa)\cap N)
	\end{equation}
	as a pro-$p$-subgroup of $J^{1}(\beta,\mfa)$, then there exists a unique irreducible representation $\eta_{P}$ of $J^{1}_{P}(\beta,\mfa)$, such that
	\begin{itemize}
		\item $(J^{1}_{P}(\beta,\mfa),\eta_{P})$ is decomposed with respect to $(M,P')$ for any parabolic  subgroup $P'$ of $G$ with a Levi factor $M$.
		
		\item $\eta_{P}\rest_{J^{1}_{M}(\beta,\mfa)}=\eta_{M}$.
		
		\item $\mrind_{J^{1}_{P}(\beta,\mfa)}^{J^{1}(\beta,\mfa)}\eta_{P}=\eta$.
	\end{itemize}
	
	To proceed, we further assume that the $E$-decomposition $V=\bigoplus_{i=1}^{t}V^{i}$ is \emph{properly subordinate} to $\Lambda$, meaning that (\emph{cf.} \cite{stevens2008supercuspidal}*{Definition 5.1})
	\begin{itemize}
		\item it conforms with $\Lambda$;
		\item for $i=1,\dots,t$ and $j\in\mbz$, we have
		$$\Lambda_{tj}^i=\Lambda_{tj+1}^i=\dots=\Lambda_{tj+i-1}^i\supsetneq \Lambda_{tj+i}^i=\dots=\Lambda_{t(j+1)}^i;$$
		\item for each $i=1,\dots,t$, the hereditary order $\mfb^{i}$ in $\mrend_E(V_E^i)$ is maximal.
	\end{itemize}
	In this case, the hereditary order $\mfb$ is $B^{\times}$-conjugate to the standard hereditary order in $B=\mrend_{E}(V_{E})$ with respect to the composition $(\mrdim_{E}(V_{E}^{1}),\dots,\mrdim_{E}(V_{E}^{t}))$ of $\mrdim_{E}(V_{E})$. Also, for each $i$ we necessarily have (\emph{cf.} \cite{bushnell129admissible}*{Proposition 7.1.12, Theorem 7.1.14})
	$$H^{1}(\beta,\Lambda^i)=H^{1}(\beta,\mfa^i),\ J^{1}(\beta,\Lambda^i)=J^{1}(\beta,\mfa^i),\ J(\beta,\Lambda^i)=J(\beta,\mfa^i)\quad \text{and}\quad \mcc(\Lambda^i,\beta)=\mcc(\mfa^i,\beta),$$
	and $$J_{M}(\beta,\mfa):=J(\beta,\mfa)\cap M=J(\beta,\mfa^{1})\times\dots\times J(\beta,\mfa^{t})$$ 
	as a subgroup of $M$, and
	\begin{equation}\label{eqJ0P}
		J_{P}(\beta,\mfa):=H^{1}(\beta,\mfa)(J(\beta,\mfa)\cap P)=(H^{1}(\beta,\mfa)\cap N^{-})\cdot(J(\beta,\mfa)\cap M)\cdot(J^{1}(\beta,\mfa)\cap N)
	\end{equation}
	as a subgroup of $J(\beta,\mfa)$. 
	
	Let $\kappa_{P}$ be the representation of $J_{P}(\beta,\mfa)$ defined on the space of $J^{1}(\beta,\mfa)\cap N^{-}$-fixed vectors in $\eta$. Then 
	\begin{itemize}
		\item $(J_{P}(\beta,\mfa),\kappa_{P})$ is decomposed with respect to $(M,P')$ for any parabolic subgroup $P'$ of $G$ with a Levi factor $M$.
		
		\item $\kappa_{P}\rest_{J_{M}(\beta,\mfa)}=\kappa_{M}:=\kappa_{1}\boxtimes\dots\boxtimes\kappa_{t}$, where $\kappa_{i}$ is a $\beta$-extension of $\eta_{i}$ for each $i$. In particular, $\kappa_{M}$ does not depend on the choice of $P$.
		
		\item $\mrind_{J_{P}(\beta,\mfa)}^{J(\beta,\mfa)}\kappa_{P}=\kappa$.
	\end{itemize}
	In the case $\mrdim_{F}(V^{1})=\dots=\mrdim_{F}(V^{t})$, we necessarily have $\theta_{1}=\dots=\theta_{t}$, $\eta_{1}\cong\dots\cong\eta_{t}$ and $\kappa_{1}\cong\dots\cong\kappa_{t}$.
	
	Finally, we notice that when considering a null stratum $[\mfa,0,0,\beta]$, our simple character $\theta$ and Heisenberg representation $\eta$ are identity characters, and the corresponding $\beta$-extensions we are considering are characters of $U(\mfa)/U^{1}(\mfa)$. We choose the identity character to be our initial $\beta$-extension $\kappa$. We say that we are in the ``level 0" case once we consider a null stratum as so.
	
	\section{Simple types}\label{sectionsimpletypes}
	
	In this section, we introduce simple types. These are types related to certain inertial equivalence classes of $\wt{G}=\wt{\mrgl_{r}(F)}$ in the sense of \S \ref{subsectionblockstypes}, in particular, those inertial equivalence classes containing a discrete series representation for KP-covers and the S-cover.
	
	\subsection{Homogeneous types}\label{subsectionhomotype}
	
	Let $[\mfa,u,0,\beta]$ be a strict simple stratum in $A=\mrend_{F}(V)$, with $V$ being an $r$-dimensional vector space over $F$ and $\Lambda$ being the corresponding lattice chain of $\mfa$. As before, let $E=F[\beta]$, $B=\mrend_{E}(V_{E})$ and $\mfb=\mfa\cap A$. Let $d=[E:F]$ and $m=r/d$. Let $d=ef$, where $e$ denotes the ramification index and $f$ denotes the unramified degree of $E/F$. Let $\bs{l}$ be the residue field of $E$.
	
	Consider a containment of $E$-pure hereditary orders $\amin\subset\mfa\subset\amax$, such that $\bmin=B\cap\amin$ is a minimal hereditary order, and $\bmax=B\cap\amax$ is a maximal hereditary order in $B$. Let $\umin=-v_{\amin}(\beta)$ and  $\umax=-v_{\amax}(\beta)$. Then $[\amin,\umin,0,\beta]$ and $[\amax,\umax,0,\beta]$ are also simple strata in $A$. Let $\Lambdamin$ and $\Lambdamax$ be the corresponding lattice chains with respect to $\amin$ and $\amax$ respectively.
	
	Consider an $E$-decomposition $V=\bigoplus_{i=1}^{t}V^{i}$ that conforms with $\Lambda$, $\Lambdamax$, $\Lambdamin$, and in particular is properly subordinate to $\Lambda$. Let $P=MN$ be a corresponding parabolic subgroup of $V$ with a Levi factor $M$ being $\prod_{i=1}^{t}\mraut_{F}(V^{i})$. We write $r_{i}=\mrdim_{F}(V^{i})$ and $m_{i}=\mrdim_{E}(V^{i}_{E})=\mrdim_{F}(V^{i})/d$ for each $i$. Then $m_{1}+\dots+m_{t}=m$ and $r_{1}+\dots+r_{t}=r$. 
	
	For $i=1,\dots,t$, let $A^{i}=\mrend_{F}(V^{i})$ (resp. $B^{i}=\mrend_{E}(V_{E}^{i})$) which is identified with a subalgebra of $A$ (resp. $B$) via the $i$-th block diagonal embedding. Let $\mfa^{i}=A^{i}\cap \mfa$ and $\mfb^{i}=B^{i}\cap \mfb$. Let $\mcg^{i}=\mrgl_{m_{i}}(\bs{l})\cong U(\mfb^{i})/U^{1}(\mfb^{i})$.
	
	Consider $$H^{1}=H^{1}(\beta,\mfa)\subset J^{1}=J^{1}(\beta,\mfa)\subset J=J(\beta,\mfa)$$
	as subgroups of $U(\mfa)$, 
	$$\Honemin=H^{1}(\beta,\amin)\subset \Jonemin=J^{1}(\beta,\amin)\subset\Jmin=J(\beta,\amin)$$
	as subgroups of $U(\amin)$, and
	$$\Honemax=H^{1}(\beta,\amax)\subset\Jonemax=J^{1}(\beta,\amax)\subset\Jmax=J(\beta,\amax)$$ as subgroups of $U(\amax)$. 
	Then, 
	$$\mcg=\mrgl_{m}(\bs{l})\cong U(\bmax)/U^{1}(\bmax)\cong\Jmax/\Jonemax.$$
	
	Moreover $\mcp\cong U(\mfb)/U^{1}(\bmax)$
	is a parabolic subgroup of $\mcg$ having a Levi factor is
	$$\mcm=\mcg^{1}\times\dots\times\mcg^{t}\cong U(\mfb)/U^{1}(\mfb)\cong J/J^{1}.$$
	Moreover, 
	$$\wt{\mcg}=\wt{U(\bmax)}/\,_{s}U^{1}(\bmax)=\wt{\Jmax}/\,_{s}\Jonemax$$ is an $n$-fold central extension of $\mcg$. We define $\wt{\mcp}$ and 
	$$\wt{\mcm}=\wt{U(\mfb)}/\,_{s}U^{1}(\mfb)\cong\wt{J}/\,_{s}J^{1}$$ 
	as the preimage in $\wt{\mcg}$ of $\mcp$ and $\mcm$ respectively.

	Choose a maximal open compact subgroup $K$ of $G$ that contains $U(\mfa)$. Fix a splitting $\bs{s}:K\rightarrow\wt{G}$ of $K$, then the restriction of $\bs{s}$ to $U(\mfa)$ is also a splitting. In general, we may also assume that $K$ contains $U(\mfa_{\max})$, then such $\bs{s}$ also induces splittings of $U(\amin)$ and $U(\amax)$. Such an $\bs{s}$ also induces a splitting of $\wt{\mcg}$.
	Then, we have an identification
	$\wt{\mcg}=\mu_{n}\times\,_{s}\mcg$ depending on the choice of $K$ and $\bs{s}$.
	
	Let $\theta\in\mcc(\mfa,\beta)$, let $\eta$ be the Heisenberg representation of $\theta$ and let $\kappa$ be a $\beta$-extension of $\eta$. Let $\wt{\kappa}$ be the pull-back of $\kappa$ as a (non-genuine) representation of $\wt{J}$. 
	
	Let $\varrho=\varrho_{1}\boxtimes\dots\boxtimes\varrho_{t}$ be a cuspidal representation of $\mcm$ that also extends trivially to a representation of $\mcp$, where each $\varrho_{i}$ is a cuspidal representation of $\mcg^{i}$. Let $\,_{s}\varrho$ be the corresponding representation of $\,_{s}\mcm$, and let $\epsilon\cdot\,_{s}\varrho$ be the extension of $\,_{s}\varrho$ to $\wt{\mcm}=\mu_{n}\times \,_{s}\mcm$ with $\mu_{n}$ acting by $\epsilon$. Then $\wt{\rho}:=\mrinf_{\wt{\mcm}}^{\wt{J}}(\epsilon\cdot\,_{s}\varrho)$ is a genuine irreducible representation of $\wt{J}$. 
	
	Let $\wt{\lambda}=\wt{\kappa}\otimes\wt{\rho}$, which is a genuine irreducible representation of $\wt{J}$.
	
	\begin{definition}
		
		A \emph{homogeneous type}\footnote{Here, the adjective ``homogeneous" follows from \cite{bushnell1999semisimple}*{\S 7}. Indeed, it is expected that these homogeneous types are indeed types in the sense of \S \ref{subsectionblockstypes}. Moreover, they are related to those inertial equivalence classes with the related cuspidal representations  having the same endo-class. However, this will not be discussed in this article and will be written down elsewhere.} of $\wt{G}$ is the pair $(\wt{J},\wt{\lambda})$ defined as above.
		
	\end{definition}
	
	We verify that this definition is independent of the choice of $K$ and the splitting $\bs{s}$. Indeed, let $\bs{s}'$ be a splitting of another maximal compact subgroup $K'$ of $G$ that contains $U(\mfa)$. Let $(\wt{J},\wt{\lambda}')$ be the pair  constructed as above, but with the splitting $\bs{s}'$ replaced by $\bs{s}$. More precisely, we have $\wt{\lambda}=\wt{\kappa}\otimes\wt{\rho}'$ with $\wt{\rho}'=\mrinf_{\wt{\mcm}}^{\wt{J}}(\epsilon\cdot\,_{s'}\varrho)$. We notice that $$\chi_{\bs{s}|\bs{s}'}:x\mapsto\bs{s}'(x)/\bs{s}(x)$$ defines a character of $\mcm\cong J/J^{1}$. Moreover, by definition we have $$\wt{\rho}'=\mrinf_{\wt{\mcm}}^{\wt{J}}(\epsilon\cdot\,_{s}(\varrho\cdot\chi_{\bs{s}|\bs{s}'}))\quad \text{and}\quad \wt{\lambda}'\cong\wt{\lambda}\cdot\chi_{\bs{s}|\bs{s}'}.$$ It means that $(\wt{J},\wt{\lambda}')$ is also a homogeneous type with respect to the pair $(K,\bs{s})$. From now on in considering a homogeneous type, we may fix any $K$ and $\bs{s}$ as above.
	
	Finally, we construct two pairs of representations related to $(\wt{J},\wt{\lambda})$. Let $$J_{M}=J\cap M,\ J_{M}^{1}=J^{1}\cap M,\ J_{P}^{1}=(H^{1}\cap N^{-})(J^{1}\cap P),\ J_{P}=(H^{1}\cap N^{-})(J\cap P)$$ and let $\kappa_{P}$, $\kappa_{M}$ be defined as in \S \ref{subsectionHeisenbergbeta}. Let $\wt{\kappa}_{P}$ (resp. $\wt{\kappa}_{M}$) be the pull-back of $\kappa_{P}$ (resp. $\kappa_{M}$) to $\wt{J_{P}}$ (resp. $\wt{J_{M}}$) as a non-genuine irreducible representation. Since $$J_{P}/J_{P}^{1}\cong J_{M}/J_{M}^{1}\cong\mcm,$$ 
	we may realize $\wt{\rho}$ as genuine irreducible representations of $\wt{J_{P}}$ and $\wt{J_{M}}$ by taking the restriction, denoted by $\wt{\rho}_{P}$ and $\wt{\rho}_{M}$ respectively. In other words, we have $\wt{\rho}_{P}=\mrinf_{\wt{\mcm}}^{\wt{J_{P}}}(\epsilon\cdot\,_{s}\varrho)$ and $\wt{\rho}_{M}=\mrinf_{\wt{\mcm}}^{\wt{J_{M}}}(\epsilon\cdot\,_{s}\varrho)$. Let $\wt{\lambda}_{P}=\wt{\kappa}_{P}\otimes\wt{\rho}_{P}$ (resp. $\wt{\lambda}_{M}=\wt{\kappa}_{M}\otimes\wt{\rho}_{M}$), which is a genuine irreducible representation of $\wt{J_{P}}$ (resp. $\wt{J_{M}}$). 
	
	Using the statements in \S \ref{subsectionHeisenbergbeta}, it is direct to verify that
	\begin{itemize}
		\item $(\wt{J_{P}},\wt{\lambda}_{P})$ is decomposed with respect to $(M,P')$ for every parabolic subgroup $P'$ of $G$ with a Levi factor $M$.
		\item $\wt{J_{P}}\cap \wt{M}=\wt{J_{M}}$ and $\wt{\lambda}_{P}\rest_{\wt{J_{M}}}=\wt{\lambda}_{M}$.
		\item $\mrind_{\wt{J_{P}}}^{\wt{J}}\wt{\lambda}_{P}\cong\wt{\lambda}$.
	\end{itemize} 
	
	
	As a result, the following lemma is valid (\emph{cf.} \cite{bushnell129admissible}*{Proposition 4.1.3}).
	
	\begin{lemma}\label{lemmaisoheckelambdaPlambda}
		
		We have a canonical isomorphism of Hecke algebras:
		$$\mch(\wt{G},\wt{\lambda}_{P})\cong\mch(\wt{G},\wt{\lambda}).$$ 
		It preserves the support, in the sense that, for $\phi_P\in \mch(\wt{G},\wt{\lambda}_{P})$ supported on $\wt{J_P}g\wt{J_P}$, the related function $\phi\in\mch(\wt{G},\wt{\lambda})$ is supported on $\wt{J}g\wt{J}$.
		
	\end{lemma}
	
	We may further explain $\wt{\lambda}_{M}$. As before,  $\Lambda^{i}=\Lambda\cap V^{i}$ is a lattice sequence of $V^{i}$, and $[\Lambda^{i},u,0,\beta]$ is a simple stratum in $A^{i}$ for each $i=1,\dots,t$, and we consider the related strict simple stratum $[\mfa^{i},u_{i},0,\beta]$ in $A^{i}$.
	Let $\kappa_{i}$ be defined as in \S \ref{subsectionHeisenbergbeta}, let $\wt{\kappa}_{i}$ be its pull-back as a non-genuine representation of $\wt{J(\beta,\mfa^{i})}$, and let $\wt{\rho}_{i}=\mrinf^{\wt{J(\beta,\mfa^{i})}}_{\wt{\mcg_{i}}}(\epsilon\cdot\,_{s}\varrho_{i})$. Let $\wt{\lambda}_{i}=\wt{\kappa}_{i}\otimes\wt{\rho}_{i}$ be a genuine irreducible representation of $\wt{J(\beta,\mfa^{i})}$. Then by definition we have 
	$$\wt{\lambda}_{M}\cong\wt{\lambda}_{1}\boxtimes\dots\boxtimes\wt{\lambda}_{t}$$
	as a representation of
	$$\wt{J_{M}}=\bs{p}^{-1}(J(\beta,\mfa^{1})\times\dots\times J(\beta,\mfa^{t})).$$
	Here, the related tensor product makes sense since $J_{M}$ is block compatible. Indeed, each pair $(\wt{J(\beta,\mfa^{i})},\wt{\lambda}_{i})$ is a so-called maximal simple type of $\wt{(A^{i})^{\times}}$, a concept to be introduced later (\emph{cf.} Definition \ref{defsimpletype}). By \eqref{eqthetadecom}, for $i=1,\dots,t$  the simple characters $\theta_{i}$ contained in $\kappa_{i}$  are transfers of each other. Thus they lie in the same endo-class.
	
	

	\subsection{Intertwining set of a homogeneous type}
	
	Our next goal is to study the intertwining set of $\wt{\lambda}$.
	
	\begin{proposition}\label{propIglambda}
		
		We have $I_{G}(\wt{\lambda})=JI_{B^{\times}}(\wt{\rho})J.$
		
	\end{proposition}
	
	\begin{proof}
		
		Since $I_{G}(\wt{\kappa})=JB^{\times}J$ and $\mrdim_{\mbc}\mrhom_{\wt{J}^{g}\cap\wt{J}}(\wt{\kappa}^{g},\wt{\kappa})=1$ for $g\in I_G(\wt{\kappa})$, the result follows from the same argument of \cite{bushnell129admissible}*{Proposition 5.3.2}.
		
	\end{proof}
	
	Then we need to study $I_{B^{\times}}(\wt{\rho})$. We introduce more notation. 
	
	Fix a certain $E$-basis of $V_{E}$ to identify $B$ with $\mrm_{m}(E)$, such that $\bmax$ is identified with $\mrm_{m}(\mfo_{E})$,  and $\mfb$ is identified with the standard hereditary order in $B$ with respect to the composition $m_{1}+\dots+m_{t}=m$, and $\bmin$ is identified with the standard minimal hereditary order in $B$. 
	
	Under this basis, let 
	$$T(B)=\{\mrdiag(\varpi_{E}^{s_{1}},\dots,\varpi_{E}^{s_{m}})\in B^{\times}\mid s_{i}\in \mbz,\ i=1,\dots,m\}$$
	and
	$$T(\mfb)=\{\mrdiag(\varpi_{E}^{s_{1}}I_{m_{1}},\dots,\varpi_{E}^{s_{t}}I_{m_{t}})\in B^{\times}\mid s_{i}\in \mbz,\ i=1,\dots,t\},$$
	where we fix a uniformizer $\varpi_{E}$ of $E$. Let 
	$M^{0}(\mfb)=M\cap U(\mfb)$ and $M^{1}(\mfb)=M\cap U^{1}(\mfb)$, then $\mcm\cong M^{0}(\mfb)/M^{1}(\mfb)$. Also, $T(\mfb)$ commutes with both $M^{0}(\mfb)$ and $M^{1}(\mfb)$.
	
	The restriction of $\wt{\rho}$ to $\wt{M^{0}(\mfb)}$ is the inflation $\mrinf_{\wt{\mcm}}^{\wt{M^{0}(\mfb)}}(\epsilon\cdot\,_{s}\varrho)$.
	
	Let $W_{0}(B)$ be the subgroup of permutation matrices in $B^{\times}\cong\mrgl_{m}(E)$ which is embedded in $\bmax^{\times}$, and let $W_{0}(\mfb)$ be the stabilizer of $T(\mfb)$ in $W_{0}(B)$. Let $W(B)=W_{0}(B)\ltimes T(B)$ be the semi-direct product as a subgroup of $B^{\times}$. Let $W(\mfb)=W_{0}(\mfb)\ltimes T(\mfb)$, which is a subgroup of $W(B)$. By direct verification, the normalizer of $M^{0}(\mfb)$ in $W(B)$ is $W(\mfb)$.
	
	We have the Bruhat decomposition
	$$B^{\times}=U(\bmin)W(B)U(\bmin).$$
	Since $U(\bmin)\subset U(\mfb)$ normalizes $\wt{\rho}$, to study $I_{B^{\times}}(\wt{\rho})$ we only need to calculate $I_{W(B)}(\wt{\rho})$.
	
	\begin{proposition}\label{proprhoWBintertwine}
		
		An element $w\in W(B)$ intertwines $\wt{\rho}$ if and only if it normalizes $M^{0}(\mfb)$ and  $\wt{\rho}\rest_{\wt{M^{0}(\mfb)}}$.
		Moreover, for $g\in I_{B^{\times}}(\wt{\rho})$ we have
		$$\mrdim_{\mbc}\mrhom_{\wt{J}^{g}\cap \wt{J}}(\wt{\rho}^{g},\wt{\rho})=1.$$
		
	\end{proposition}
	
	\begin{proof}
		
		For $w\in I_{W(B)}(\wt{\rho})$, we necessarily have that 
		\begin{equation}\label{eqrhowintertwine}
			\mrhom_{\wt{M^{0}(\mfb)}\cap\wt{M^{0}(\mfb)^{w}}}(\wt{\rho},\wt{\rho}^{w})\neq 0.
		\end{equation}

		The argument of the ``only if" part follows from the proof of \cite{bushnell129admissible}*{Proposition 5.5.5}, which we explain as follows.  If $w$ does not normalize $M^{0}(\mfb)$, then as in \emph{loc. cit.} there exist a certain $i=1,\dots,t$, a standard hereditary order $\mfb'$ in $B^{i}=\mrend_{E}(V_{E}^{i})$ with respect to the decomposition $V_{E}=\bigoplus_{i=1}^{t}V_{E}^{i}$, such that $$U^{1}(\mfb^{i})\subsetneq U^{1}(\mfb')\subset U(\mfb^{i})\quad \text{and}\quad U^{1}(\mfb')\subset M^{1}(\mfb)^{w}.$$ 
		We consider $\mcn'= U^{1}(\mfb')/U^{1}(\mfb^{i})\cong\,_{s}U^{1}(\mfb')/\,_{s}U^{1}(\mfb^{i})$, which is a proper unipotent subgroup of $\mcg^{i}$. Then, restricting \eqref{eqrhowintertwine} to $\,_{s}U^{1}(\mfb')$ and modulo $\,_{s}U^{1}(\mfb^{i})$, we have
		$$\mrhom_{\mcn'}(\varrho_{i},1)\neq 0,$$
		contradicting the fact that $\varrho_{i}$ is cuspidal. So $w$ normalizes $M^{0}(\mfb)$.
		
		On the other hand, assume $w\in W(B)$ normalizes $M^{0}(\mfb)$ and $\wt{\rho}\rest_{\wt{M^{0}(\mfb)}}$. We claim that 
		\begin{equation}\label{eqhomrhorhow}
			\mrhom_{\wt{J}\cap\wt{J}^{w}}(\wt{\rho},\wt{\rho}^{w})
			=\mrhom_{\wt{M^{0}(\mfb)}}(\wt{\rho},\wt{\rho}^{w})=\mbc.
		\end{equation}
		It suffices to show the first equality. Remark that we have decompositions
		$$J=M^{0}(\mfb)J^1\quad\text{and}\quad J^w=M^{0}(\mfb)^wJ^{1w}=M^{0}(\mfb)J^{1w},$$
		thus we have
		$$J\cap J^w=M^{0}(\mfb)(J^{1}\cap M^{0}(\mfb)J^{1w}).$$
		We claim that $$J^{1}\cap M^{0}(\mfb)J^{1w}=J^{1}\cap J^{1w}$$ which suffices to show the first equality of \eqref{eqhomrhorhow}, since both $\wt{\rho}$ and $\wt{\rho}^w$ are trivial on $\wt{J^1}\cap \wt{J^{1w}}$. Assume that there exist $m_0\in M^{0}(\mfb)$, $g_1\in J^{1}\subset U^1(\mfa)$, $g_2\in J^{1w}\subset U^{1}(\mfa)^w$, such that $m_0=g_1g_2$. It suffices to show that $m_0\in M^{1}(\mfb)$, or a fortiori, it suffices to show that
		$$M^0(\mfb)\cap U^{1}(\mfa)w^{-1}U^{1}(\mfa)w=M^1(\mfb).$$
		Using \cite{bushnell129admissible}*{Theorem 1.6.1}, we have
		$$M^0(\mfb)\cap U^{1}(\mfa)w^{-1}U^{1}(\mfa)w=M^0(\mfb)\cap U^{1}(\mfb)w^{-1}U^{1}(\mfb)w.$$
		Thus, we only need to show that\footnote{We remark that such a claim has already been implicitly used in \cite{grabitz2001level}*{Lemma 1.5}.}
		$$U^{0}(\mfb)^{w}\cap U^{1}(\mfb)=M^0(\mfb)U^{1}(\mfb)^w\cap U^{1}(\mfb)\subset U^{1}(\mfb)^w.$$
		To show the last claim, suppose we have $g\in U^{0}(\mfb)$ such that $g^{w}\in U^{0}(\mfb)^{w}\cap U^{1}(\mfb)$. Consider the Iwahori decomposition $$g=n_1mn_2\quad \text{in}\quad U^{0}(\mfb)=(U^{0}(\mfb)\cap N^{-})(U^{0}(\mfb)\cap M)(U^{0}(\mfb)\cap N),$$ thus $$g^{w}=n_1^{w}m^{w}n_2^{w}\quad \text{in}\quad U^{0}(\mfb)=(U^{0}(\mfb)\cap N^{-w})(U^{0}(\mfb)\cap M)(U^{0}(\mfb)\cap N^{w})$$
		is the related Iwahori decomposition of $g^{w}$. Since $g^{w}\in U^{1}(\mfb)$, both $m^{w}$ and $m$ are in $U^{1}(\mfb)\cap M$. Thus $g\in U^{1}(\mfb)$ and $g^{w}\in U^{1}(\mfb)^{w}$.

		Finally, the second statement of the proposition follows from \eqref{eqhomrhorhow}.
		
	\end{proof}
	
	As a direct corollary of Proposition \ref{propIglambda} and Proposition \ref{proprhoWBintertwine}, 
	
	\begin{corollary}\label{corintertwingspace}
		
		We have 
		$$\mrdim_{\mbc}\mrhom_{\wt{J}^{g}\cap \wt{J}}(\wt{\lambda}^{g},\wt{\lambda})=1$$
		for any $g\in I_{G}(\wt{\lambda})$. Thus there exists a unique $\phi\in\mch(\wt{G},\wt{\lambda})$ up to a scalar that is supported on $\wt{J}g\wt{J}$.
		
	\end{corollary}
	
	We first study $I_{W_{0}(B)}(\wt{\rho})$. We may regard $W_{0}(B)$ as a subgroup of $\mcg$ modulo $U^{1}(\bmax)$.
	
	\begin{proposition}\label{propW0Brho}
		
		We have that $w\in I_{W_{0}(B)}(\wt{\rho})$ if and only if $w$ normalizes $\varrho=\varrho_{1}\boxtimes\dots\boxtimes\varrho_{t}$. In particular, $I_{W_{0}(B)}(\wt{\rho})$ is contained in $W_{0}(\mfb)$.
		
	\end{proposition}
	
	\begin{proof}
		
		If $w$ normalizes $\varrho$, then it also normalizes $\wt{\rho}\rest_{\wt{M^{0}(\mfb)}}$. It is because both $M^{0}(\mfb)$ and $w$ are in $K$, thus we may simply take the splitting $\bs{s}$ of $K$. Thus by Proposition \ref{proprhoWBintertwine} $w$ intertwines $\wt{\rho}$. 
		
		Conversely, if $w\in W_{0}(B)$ intertwines $\wt{\rho}$, by Proposition \ref{proprhoWBintertwine} we have that $w$ normalizes $M^{0}(\mfb)$ and is in $W_{0}(\mfb)$. 
		So
		$$0\neq\mrhom_{\wt{J}^{w}\cap\wt{J}}(\wt{\rho}^{w},\wt{\rho})\cong\mrhom_{\wt{M^{0}(\mfb)}}(\wt{\rho}^{w},\wt{\rho})\cong \mrhom_{\wt{\mcm}}((\epsilon\cdot\varrho)^{w},\epsilon\cdot\varrho)\cong\mrhom_{\mcm}(\varrho^{w},\varrho).$$
		Thus $w$ normalizes $\varrho$.
		
	\end{proof}
	
	We also study  $I_{T(B)}(\wt{\rho})$. The normalizer of $M^{0}(\mfb)$ in $T(B)$ is $T(\mfb)$, so we only need to investigate the intertwining set $I_{T(\mfb)}(\wt{\rho})$. 
	
	Before that, we need to say something for $\varrho=\varrho_{1}\boxtimes\dots\boxtimes\varrho_{t}$. 
	
	Let $l_{i}$ be the maximal positive integer dividing $n$, such that $\varrho_{i}$ is isomorphic to its twist by a character of $\mcg^{i}$ of order $l_{i}$. 
	
	We may further use the theory of Green \cite{green1955characters} or the Deligne-Lusztig theory \cite{deligne1976representations} to describe $l_{i}$. For each $i$, let $\bs{l}_{m_{i}}$ be the finite extension of degree $m_{i}$ over $\bs{l}$, and let $\xi_{i}$ be a regular character of $\bs{l}_{m_{i}}^{\times}$ that constructs $\varrho_{i}$ in the sense of the above two theories. Here, by regular we mean that the orbit of $\xi_{i}$ under the action of the Galois group $\mathrm{Gal}(\bs{l}_{m_{i}}/\bs{l})$ is of cardinality $m_i$.
	
	Fix a Frobenius element in $\mathrm{Gal}(\bs{l}_{m_{i}}/\bs{l})$ as the $q_{\bs{l}}$'s power map, where $q_{\bs{l}}$ denotes the cardinality of $\bs{l}$. Then there exists a smallest positive integer $o_{i}$, such that $\xi_{i}^{q_{\bs{l}}^{o_{i}}}/\xi_{i}$ is a character of $\bs{l}_{m_{i}}^{\times}$ of order dividing $n$. An easy exercise implies that $o_{i}l_{i}=m_{i}$, and in particular $l_{i}$ divides $m_{i}$.
	
	\begin{proposition}\label{Tmfbrho}
		
		Consider the following linear congruence equations:
		\begin{equation}\label{eqcongruence}
			l_{i}[(s_{1}r_{1}+\dots+s_{t}r_{t})(2\bs{c}+\bs{d})-s_{i}\bs{d}]\equiv 0\quad(\mathrm{mod}\ n),\quad i=1,\dots t.
		\end{equation}
		Then the intertwining set $I_{T(\mfb)}(\wt{\rho})=I_{T(\mfb)}(\wt{\rho}\rest_{\wt{M^{0}(\mfb)}})$ equals
		\begin{equation}\label{eqTli}
			T(\varrho):=\{\mrdiag(\varpi_{E}^{s_{1}}I_{m_{1}},\dots,\varpi_{E}^{s_{t}}I_{m_{t}})\mid s_{1},\dots,s_{t}\in \mbz\ \text{satisfy}\ \eqref{eqcongruence}\}.
		\end{equation}
		
	\end{proposition}
	
	\begin{proof}
		
		Let $h=\mrdiag(\varpi_{E}^{s_{1}}I_{m_{1}},\dots,\varpi_{E}^{s_{t}}I_{m_{t}})$ that intertwines $\wt{\rho}$. By definition, $h$ commutes with $M^{0}(\mfb)$. So we have that $\wt{\rho}^{h}=\wt{\rho}\cdot \chi_{h}$, where $\chi_{h}:=\epsilon([h,\cdot]_{\sim})$ is a character of $M^{0}(\mfb)$ of order dividing $n$. For $g=\mrdiag(g_{1},\dots,g_{t})\in M^{0}(\mfb)$, using \eqref{eqsumEiMEicommutator} we have
		\begin{equation}
			\begin{aligned}
				[h,g]_{\sim}&=[\mrdiag(\varpi_{E}^{s_{1}}I_{m_{1}},\dots,\varpi_{E}^{s_{t}}I_{m_{t}}),\mrdiag(g_{1},\dots,g_{t})]_{\sim}\\
				&=\prod_{i=1}^{t}(\mrdet_{F}(h)^{2\bs{c}+\bs{d}}\varpi_{E}^{-s_{i}\bs{d}},\mrdet_{E}(g_{i}))_{n,E}\\
				&=\prod_{i=1}^{t}(\varpi_{E}^{(s_{1}r_{1}+\dots+s_{t}r_{t})(2\bs{c}+\bs{d})-s_{i}\bs{d}},\mrdet_{E}(g_{i}))_{n,E}.
			\end{aligned}
		\end{equation}
		Let $\chi_{\varpi_{E}}=\epsilon((\varpi_{E},\cdot)_{n,E})$ be a character of $\mfo_{E}^{\times}$ of order $n$. Then
		we have 
		\begin{equation}\label{eqchih}
			\chi_{h}(g)=\prod_{i=1}^{t}\chi_{\varpi_{E}}(\mrdet_{E}(g_{i}))^{(s_{1}r_{1}+\dots+s_{t}r_{t})(2\bs{c}+\bs{d})-s_{i}\bs{d}}.
		\end{equation}
		Noting that $(1+\mfp_{E})^{n}=1+\mfp_{E}$ when $\mrgcd(q,n)=1$, so $\chi_{\varpi_{E}}$ is trivial on $1+\mfp_{E}$. Thus it can be regarded as a character of $\bs{l}^{\times}\cong\mfo_{E}^{\times}/1+\mfp_{E}$. Similarly, $\chi_{h}$ can be regarded as a character of $\mcm\cong M^{0}(\mfb)/M^{1}(\mfb)$. 
		
		Then $h$ intertwines $\wt{\rho}$ 
		if and only if $\wt{\rho}\rest_{\wt{M^{0}}(\mfb)}\cdot\chi_{h}\cong\wt{\rho}\rest_{\wt{M^{0}}(\mfb)}$, if and only if $\varrho\cdot\chi_{h}\cong\varrho$, if and only if
		$$\varrho_{i}\cdot(\chi_{\varpi_{E}}\circ\mrdet_{\bs{l}})^{(s_{1}r_{1}+\dots+s_{t}r_{t})(2\bs{c}+\bs{d})-s_{i}\bs{d}}\cong\varrho_{i},\quad i=1,\dots,t.$$
		Since $\chi_{\varpi_{E}}\circ\mrdet_{\bs{l}}$ is a character of order $n$, the above relation is equivalent to the following linear congruence equations:
		$$(s_{1}r_{1}+\dots+s_{t}r_{t})(2\bs{c}+\bs{d})-s_{i}\bs{d}\equiv 0\quad(\mathrm{mod}\ n/l_{i}),\quad i=1,\dots ,t.$$
		So the result follows.
		
	\end{proof}
	
	In general, the solution of \eqref{eqcongruence} could be quite messy. But in the following two cases, it is rather neat.
	
	\begin{corollary}\label{corsolcongeq}
		
		The solution of equations \eqref{eqcongruence} is
		
		\begin{itemize}
			\item if $\wt{G}$ is a KP-cover, then	$$s_{i}=k_{i}n/l_{i}+k(2\bs{c}+1)n/\mrgcd(n,2r\bs{c}+r-1),\quad i=1,\dots,t,$$
			where $k\in\mbz$ and $k_{i}\in\mbz$ for $i=1,\dots,t$.
			
			\item if $\wt{G}$ is the S-cover, then $$s_{i}=k_{i}n/\mrgcd(n,2l_{i}),\quad i=1,\dots,t,$$ 
			where $k_{i}\in\mbz$ for $i=1,\dots,t$.
		\end{itemize}
		
	\end{corollary}
	
	\begin{proof}
		
		If $\wt{G}$ is a KP-cover, then $\bs{d}=1$. Write $\Gamma=s_{1}r_{1}+\cdots+s_{t}r_{t}$. We multiply the $i$-th equation in \eqref{eqcongruence} by $r_{i}/l_{i}$ and sum them together, which induces
		$$(2\bs{c}r+\bs{c}-1)\Gamma\equiv 0\quad(\mathrm{mod}\ n).$$
		Also, let $s_{i}'=s_{i}-(2\bs{c}+1)\Gamma$, then \eqref{eqcongruence} becomes
		$$l_{i}s_{i}'\equiv 0 \quad(\mathrm{mod}\ n).$$
		So solving these equations for $s_{i}'$ and $\Gamma$, we get the desired result.
		
		If $\wt{G}$ is the S-cover, then $\bs{d}=2$ and $2\bs{c}+\bs{d}=0$. So \eqref{eqcongruence} becomes
		$$2l_{i}s_{i}\equiv 0 \quad(\mathrm{mod}\ n),\quad i=1,\dots,t,$$
		which can also be easily solved.
		
	\end{proof}
	
	We also calculate the intertwining set $I_{M}(\wt{\lambda}_{M})$ and the normalizer $N_{M}(\wt{\lambda}_{M})$.
	
	\begin{proposition}\label{propintertwinelambdaM}
		
		We have $I_{M}(\wt{\lambda}_{M})=N_{M}(\wt{\lambda}_{M})=T(\varrho)J_{M}$, where $T(\varrho)$ is given by Proposition \ref{Tmfbrho}. 
		
	\end{proposition}
	
	\begin{proof}
		
		First we notice that $I_{M}(\wt{\kappa}_{M})=J_{M}(M\cap B^{\times})J_{M}$ and $\mrhom_{\wt{J_{M}^{g}}\cap\wt{J_{M}}}(\wt{\kappa}_{M}^{g},\wt{\kappa}_{M})\cong\mbc$ for $g\in I_{M}(\wt{\kappa}_{M})$. This simply follows from the fact that $J_{M}=J(\beta,\mfa^{1})\times\dots\times J(\beta,\mfa^{t})$ and $\kappa_{M}=\kappa_{1}\boxtimes\dots\boxtimes\kappa_{t}$, where $\kappa_{i}$ is a $\beta$-extension of $J(\beta,\mfa^{i})$ for $i=1,\dots,t$. Then, as in Proposition \ref{propIglambda} we have $$I_{M}(\wt{\lambda}_{M})=J_{M}I_{M\cap B^{\times}}(\wt{\rho}\rest_{M^{0}(\mfb)})J_{M}.$$ 
		Using the Cartan decomposition, we have $$M\cap B^{\times}=(U(\mfb)\cap M)T(B)(U(\mfb)\cap M).$$ 
		Since $U(\mfb)\cap M$ is contained in $J_{M}$, we have 
		$$I_{M}(\wt{\lambda}_{M})=J_{M}I_{T(B)}(\wt{\rho}\rest_{M^{0}(\mfb)})J_{M}.$$ 
		Finally, since the normalizer of $M^{0}(\mfb)$ in $T(B)$ is $T(\mfb)$, by Proposition \ref{proprhoWBintertwine} we have $$I_{M}(\wt{\lambda}_{M})=J_{M}I_{T(\mfb)}(\wt{\rho}\rest_{M^{0}(\mfb)})J_{M}=T(\varrho)J_{M}.$$ 
		
		Also, since the group $\bs{J}_{M}=T(\mfb)J_{M}$ contains $I_{M}(\wt{\lambda}_{M})=T(\varrho)J_{M}$, we also have $N_{M}(\wt{\lambda}_{M})=T(\varrho)J_{M}$.
		
	\end{proof}
	
	\subsection{Simple types}\label{subsectionsimpletypes}
	
	We keep the notation of previous subsections. We study a special class of homogeneous types, the so-called simple types.
	
	\begin{definition}\label{defsimpletype}
		
		A \emph{twisted simple type} of $\wt{G}$ is a homogeneous type $(\wt{J},\wt{\lambda})$ as before, satisfying the following properties:
		
		\begin{itemize}
			
			\item $\wt{\lambda}=\wt{\kappa}\otimes\wt{\rho}$ with $\wt{\rho}=\mrinf_{\wt{\mcm}}^{\wt{J}}(\epsilon\cdot\,_{s}\varrho)$.
			
			\item $m_{1}=\dots=m_{t}$. Then we write $m_{0}=m/t$ and $r_{0}=r/t$.
			
			\item There exist an irreducible cuspidal representation $\varrho_{0}$ of $\mrgl_{m_{0}}(\bs{l})$ and $g_{0}\in T(\mfb)$, such that $\varrho$ is isomorphic to $(\varrho_{0}\boxtimes\dots\boxtimes\varrho_{0})\chi_{g_{0}}$ as a representation of $\mcm=\mcg^{1}\times\dots\times\mcg^{t}$, where $\mcg^{i}=\mrgl_{m_{0}}(\bs{l})$ for $i=1,\dots,t$, and $\chi_{g_{0}}:=\epsilon([g_{0},\cdot]_{\sim})$ is a character of $\mcm\cong M^{0}(\mfb)/M^{1}(\mfb)$.
			
		\end{itemize}
		
		If $g_{0}=1$, we call the corresponding $(\wt{J},\wt{\lambda})$ a \emph{simple type} of $\wt{G}$.
		
		If the corresponding hereditary order $\mfb$ is maximal, or equivalently $t=1$, we call the corresponding $(\wt{J},\wt{\lambda})$ a \emph{maximal simple type} of $\wt{G}$.
		
		If we fix $\wt{\kappa}$ and $\varrho_{0}$ and let $g_{0}$ range over $T(\mfb)$, then the corresponding $\wt{\lambda}$ forms a finite set of representations of $\wt{J}$, which we denote by $[\wt{\lambda}]$ and call the \emph{weak equivalence class} of $\wt{\lambda}$.
		
	\end{definition}
	
	Still, we verify that the definition of twisted simple types is independent of the choice of $K$ and $\bs{s}$. Let $K'$ be another maximal open compact subgroup of $G$ that contains $U(\mfa)$, and let $\bs{s}'$ be a splitting of $K'$. 
	
	We first prove the following claim.
	
	\begin{lemma}
		
		There exists $g\in W(\mfb)$ such that $U(\mfb)^{g}=U(\mfb)$ and $K'=K^{g}$.
		
	\end{lemma}
	
	\begin{proof}
		
		Let $\mfa$ be the corresponding hereditary order in $A$ in defining our twisted simple type. 
		
		We first notice that when $g$ ranges over the normalizer $N_{G}(U(\mfa))$, the corresponding $K^{g}$ ranges over all the maximal open compact subgroup of $G$ that contains $U(\mfa)$. To see this, we may without loss of generality assume that $\mfa$ is a standard hereditary order. In this case, the normalizer $N_{G}(U(\mfa))$ is generated by $U(\mfa)$ and the element $$\Pi_\mfa=\begin{pmatrix}
			& I_{m_{0}f(et-1)} \\
			\varpi_{F}I_{m_{0}f} &
		\end{pmatrix}\in\mrgl_{r}(F)$$	
		Thus in this case $K$, $K^{\Pi_\mfa},\dots,K^{\Pi_\mfa^{et-1}}$ are all the maximal compact subgroups containing $U(\mfa)$.
		
		We also claim that there exists $h\in W(\mfb)\cap N_{B^{\times}}(U(\mfb))$ such that $N_{G}(U(\mfa))=\pairangone{h}U(\mfa)$. Indeed, let $\Lambda$ be the corresponding lattice chain of $\mfa$. Then $N_{G}(U(\mfa))$ is generated by $U(\mfa)$ and an element $h$ in $G$ that maps $\Lambda^{i}$ to $\Lambda^{i+1}$ for each $i\in\mbz$ (In the above standard case, it could be the element $\Pi_{\mfa}$ as above). Or, since $\Lambda$ is $E$-pure, such $h$ can be chosen to be an element in $B^{\times}$ that maps $\Lambda_{E}^{i}$ to $\Lambda_{E}^{i+1}$ for each $i\in\mbz$. Such $h$, by definition, can be chosen as an element in $W(\mfb)\cap N_{B^{\times}}(U(\mfb))$. 
		
		Thus we may choose $g\in N_{G}(U(\mfa))$, such that $K'=K^{g}$. Since $U(\mfa)\subset K$, we may also assume that $g\in\pairangone{h}$. Thus we also have $U(\mfb)^{g}=U(\mfb)$, which finishes the proof.
		
	\end{proof}
	
	Let $(\wt{J},\wt{\lambda}')$ be the corresponding twisted simple type constructed using the pair $(K',\bs{s}')$ instead of $(K,\bs{s})$. More precisely, we have $\wt{\lambda}'=\wt{\kappa}\otimes\wt{\rho}'$, where $\wt{\rho}'=\mrinf_{\wt{\mcm}}^{\wt{J}}(\epsilon\cdot\,_{s'}\varrho)$. We would like to show that $(\wt{J},\wt{\lambda}')$ is a twisted simple type with respect to the pair $(K,\bs{s})$. We write $\varrho=(\varrho_{0}\boxtimes\dots\boxtimes\varrho_{0})\chi_{g_{0}}$ with $g_{0}\in T(\mfb)$.
	
	Choose $g$ as in the lemma, and write $g=w_{0}h$ with $w_{0}\in W_{0}(\mfb)$ and $h\in T(\mfb)$. We write $h'=g_{0}^{-1}w_{0}^{-1}g_{0}w_{0}h\in T(\mfb)$, since $g_{0}, w_{0}^{-1}g_{0}w_{0}, h\in T(\mfb)$. By definition, we have $$\wt{\rho}'=\mrinf_{\wt{\mcm}}^{\wt{J}}((\epsilon\cdot\,_{s}\varrho)^{g}\chi_{\bs{s}^{g}|\bs{s}'})=\mrinf_{\wt{\mcm}}^{\wt{J}}((\epsilon\cdot\,_{s}\varrho)^{h'}\chi_{\bs{s}^{g}|\bs{s}'}).$$
	The character $\chi_{\bs{s}^{g}|\bs{s}'}(x)=\bs{s}'(x)/\bs{s}^{g}(x)$ of $K'$ can be written as a character $\chi$ of $\mfo_{F}^{\times}/1+\mfp_{F}\cong\bs{k}^{\times}$ composing with the determinant $\mrdet_{F}$. Let $$\varrho'=\varrho_{0}(\chi\circ\mrn_{\bs{l}/\bs{k}}\circ\mrdet_{\bs{l}})\boxtimes\cdots\boxtimes \varrho_{0}(\chi\circ\mrn_{\bs{l}/\bs{k}}\circ\mrdet_{\bs{l}}),$$ then by direct verification we have $\varrho\chi_{\bs{s}^{g}|\bs{s}'}=\varrho'\chi_{g_0}$ and $\wt{\rho}'=\mrinf_{\wt{\mcm}}^{\wt{J}}(\epsilon\cdot\,_{s}(\varrho'\chi_{g_{0}h'}))$. Thus $(\wt{J},\wt{\lambda}')$ is also a twisted simple type with respect to the pair $(K,\bs{s})$. Thus in considering a twisted simple type, we may fix any $K$ and $\bs{s}$ as above. 
	
	\begin{remark}
		
		However, the definition of a simple type indeed depends on the choice of the maximal compact subgroup $K$ containing $J$ and the splitting $\bs{s}$ of $K$.
		
	\end{remark}
	
	In the rest of this subsection, let $(\wt{J},\wt{\lambda})$ be a twisted simple type of $\wt{G}$. We further study the intertwining set of $\wt{\lambda}$. 
	
	Let $l_{0}$ be the maximal positive integer dividing $n$, such that $\varrho_{0}$ is isomorphic to its twist by a certain character of $\mrgl_{m_{0}}(\bs{l})$ of order $l_{0}$. As before, $l_{0}$ divides $m_{0}$.
	
	Let $T(r_{0},m_{0},l_{0};t)$ be the group $T(\varrho)$ in \eqref{eqTli} with $l_{1}=\dots=l_{t}=l_{0}$, $m_{1}=\dots=m_{t}=m_{0}$ and $r_{1}=\dots=r_{t}=r_{0}$. 
	Then $T(r_{0},m_{0},l_{0};t)$ is normalized by $W(\mfb)$. 
	
	Let $$W(r_{0},m_{0},l_{0};t)=W_{0}(\mfb)\ltimes T(r_{0},m_{0},l_{0};t),$$ 
	which is a subgroup of $W(\mfb)$.

	\begin{proposition}\label{propintertwinelambdasimple}
		
		We have $I_{G}(\wt{\lambda})=JW(r_{0},m_{0},l_{0};t)^{g_{0}}J$.
		
	\end{proposition}
	
	\begin{proof}
		
		Using Proposition \ref{propIglambda} and Proposition \ref{proprhoWBintertwine}, we only need to show that the normalizer $N_{W(\mfb)}(\wt{\varrho}\rest_{M(\mfb)})$ equals $W(r_{0},m_{0},l_{0};t)^{g_{0}}$. Since $\wt{\varrho}=(\epsilon\cdot\,_{s}(\varrho_{0}\boxtimes\dots\boxtimes\varrho_{0}))^{g_0}$, by extracting the $g_0$-conjugation we may assume $\varrho=\varrho_{0}\boxtimes\dots\boxtimes\varrho_{0}$ and $g_{0}=1$ without loss of generality. Then the proposition follows from Proposition \ref{propW0Brho}, Proposition \ref{Tmfbrho}, and the fact that $W_{0}(\mfb)$ normalizes $\varrho$.
		
	\end{proof}
	
	We give a more concrete description of $W(r_{0},m_{0},l_{0};t)$. As before, we choose an $E$-basis of $V_{E}=V_{E}^{1}\oplus\cdots\oplus V_{E}^{t}$, such that $\bmax$, $\mfb$ and $\bmin$ are standard hereditary orders of $B\cong\mrm_{m}(E)$. Under such basis, $W_{0}(\mfb)$ is identified with the group of transposition matrices of the form $$(\delta_{\varsigma(i)j}I_{m_{0}})_{1\leq i,j\leq t},$$ 
	where $\varsigma\in \mfS_{t}$, and $\delta_{ij}=1$ if $i=j$ and $0$ otherwise. In particular, \begin{equation}\label{eqStW0bident}
		\mcf_{0}:\mfS_{t}\rightarrow W_{0}(\mfb),\quad \varsigma\mapsto (\delta_{\varsigma(i)j}I_{m_{0}})_{1\leq i,j\leq t}
	\end{equation}
	is an isomorphism. Let $\varsigma_{i}$ denote the transposition of $i$ and $i+1$ and let $\sigma_{i}:=\mcf_{0}(\varsigma_{i})$ for $i=1,\dots,t-1$. Then $\{\sigma_{1},\dots,\sigma_{t-1}\}$ is a set of generators of $W_{0}(\mfb)$. 
	
	Consider the following linear congruence equations:
	\begin{equation}\label{eqcongruencesimple}
		l_{0}[(s_{1}+\dots+s_{t})r_{0}(2\bs{c}+\bs{d})-s_{i}\bs{d}]\equiv 0\quad(\mathrm{mod}\ n),\quad i=1,\dots,t.
	\end{equation}
	Using \eqref{eqTli}, we have
	\begin{equation}\label{eqTr0m0l0t}
		T(r_{0},m_{0},l_{0};t)=\{\mrdiag(\varpi_{E}^{s_{1}}I_{m_{0}},\dots,\varpi_{E}^{s_{t}}I_{m_{0}})\mid s_{1},\dots,s_{t}\in \mbz\ \text{satisfies}\ \eqref{eqcongruencesimple}\}.
	\end{equation}
	Consider the solutions of \eqref{eqcongruencesimple} such that $s_{1}=\dots=s_{t-1}=0$. Then it is required that $$l_{0}[(2\bs{c}+\bs{d})r_{0}-\bs{d}]s_{t}\equiv 0\ (\mathrm{mod}\ n)\quad \text{and}\quad l_{0}\bs{d}s_{t}\equiv 0\ (\mathrm{mod}\ n),$$ so $s_{t}=kn_{r_{0},l_{0}}$ for $k\in\mbz$, where 
	\begin{equation}\label{eqnr0l0}
		n_{r_{0},l_{0}}:=n/\mrgcd(n,(2\bs{c}+\bs{d})r_{0}l_{0},\bs{d}l_{0}).
	\end{equation} 
	In particular, \begin{equation}\label{eqnr0l0KPS}n_{r_{0},l_{0}}=\begin{cases}
			n/l_{0}&\quad \text{if}\ \wt{G}\ \text{is a KP-cover},\\
			n/\mrgcd(n,2l_{0}) &\quad \text{if}\ \wt{G}\ \text{is the S-cover}.
		\end{cases}
	\end{equation}
	We also consider the solutions of \eqref{eqcongruencesimple} such that $s:=s_{1}=\dots=s_{t}$. Then it is required that
	$$l_{0}[(2\bs{c}+\bs{d})r-\bs{d}]s\equiv0\ (\text{mod}\ n).$$
	So $s_{1}=\dots=s_{t}=kd_{r,l_{0}}$ for $k\in\mbz$, where 
	\begin{equation}\label{eqdrl0}
		d_{r,l_{0}}=n/\mrgcd(n,l_{0}(2\bs{c}r+\bs{d}r-\bs{d})).
	\end{equation} 
	In particular,
	\begin{equation}\label{eqdrl0KPS}
		d_{r,l_{0}}=\begin{cases}n/(l_{0}\mrgcd(n,2\bs{c}r+r-1))&\quad \text{if}\ \wt{G}\ \text{is a KP-cover},\\
			n/\mrgcd(n,2l_{0}) &\quad \text{if}\ \wt{G}\ \text{is the S-cover}.
		\end{cases}
	\end{equation}
	Note that for the KP-cover, we use the fact that $l_{0}$ divides $r$ and $\mrgcd(l_{0},2\bs{c}r+r-1)=1$. 
	
	By definition, $d_{r,l_{0}}$ divides $n_{r_{0},l_{0}}$. We write \begin{equation}\label{eqsr-l0}
		s_{r_{0},l_{0}}=n_{r_{0},l_{0}}/d_{r,l_{0}}.
	\end{equation} 
	For most cases, we write $n_{0}=n_{r_{0},l_{0}}$, $d_{0}=d_{r,l_{0}}$ and $s_{0}=s_{r_{0},l_{0}}$ for short.
	
	Let $\varsigma'$ be the permutation $1\mapsto 2\mapsto\dots \mapsto t-1\mapsto t\mapsto 1$. Let \begin{equation}\label{eqPizeta}
		\Pi_{E}=\mrdiag(\underbrace{I_{m_{0}},\dots,I_{m_{0}},\varpi_{E}^{n_{0}}I_{m_{0}}}_{t\text{-terms}})\mcf_{0}(\varsigma')\quad\text{and}\quad
		\zeta_{E}
		=\mrdiag(\underbrace{\varpi_{E}^{d_{0}}I_{m_{0}},\dots,\varpi_{E}^{d_{0}}I_{m_{0}}}_{t\text{-terms}}).
	\end{equation}
	be elements in $W(r_{0},m_{0},l_{0};t)$. By definition, we have $\Pi_{E}^{t}=\zeta_{E}^{s_{0}}$. 
	
	The elements $\sigma_{1},\dots,\sigma_{k-1}$, $\Pi_{E}$ and $\zeta_{E}$ generate a normal subgroup of $W(r_{0},m_{0},l_{0};t)$ of finite index, which we denote by $W'(r_{0},m_{0},l_{0};t)$. 
	
	We often write $T_{0}=T(r_{0},m_{0},l_{0};t)$, $W_{0}=W(r_{0},m_{0},l_{0};t)$ and $W_{0}'=W'(r_{0},m_{0},l_{0};t)$ for short.
	
	\begin{proposition}\label{propKPSW=W'}
		
		If $\wt{G}$ is either a KP-cover or the S-cover, then $W_{0}=W_{0}'$.
		
	\end{proposition}
	
	\begin{proof}
		
		Using Proposition \ref{Tmfbrho}, Corollary \ref{corsolcongeq}, \eqref{eqnr0l0KPS},
		\eqref{eqdrl0KPS}, we may verify directly that $\sigma_{1},\dots,\sigma_{k-1}$, $\Pi_{E}$ and $\zeta_{E}$ generate $W_0$. 
		
	\end{proof}
	
	In particular, for the S-cover we have $s_{0}=1$ and $\Pi_{E}^{t}=\zeta_{E}$, hence $\sigma_{1}\dots,\sigma_{t-1}$ and $\Pi_{E}$ generate $W_{0}$.
	
	\subsection{Main results for homogeneous and simple types}\label{subsectionmainhomsimtype}
	
	We list our main results for homogeneous and (twisted) simple types.
	
	First we let $(\wt{J},\wt{\lambda})$ be a homogeneous type of $\wt{G}$, and we construct the related pairs $(\wt{J_{P}},\wt{\lambda}_{P})$ and $(\wt{J_{M}},\wt{\lambda}_{M})$ as before.
	
	An inertial equivalence class $\mfs_{M}=(\wt{M},\mco_{M})$ of $\wt{M}$ is called \emph{cuspidal} if $\mco_{M}$ consists of (genuine) cuspidal representations of $\wt{M}$.
	
	\begin{theorem}\label{thmlambdaMtype}
		
		The pair $(\wt{J_{M}},\wt{\lambda}_{M})$ is a type related to a cuspidal inertial equivalence class of $\wt{M}$.
		
	\end{theorem}
	
	\begin{proof}
		
		We define $\bs{J}_{M}=T(\mfb)J_{M}$ as a subgroup in $M$. By Proposition \ref{propintertwinelambdaM}, 
		$\bs{J}_{M}$ contains the normalizer and the intertwining set $I_{M}(\wt{\lambda}_{M})=N_{M}(\wt{\lambda}_{M})$. Let $Z(\wt{M})$ be the center of $\wt{M}$. Then both $\wt{\bs{J}_{M}}$ and $N_{\wt{M}}(\wt{\lambda}_{M})$ contain $Z(\wt{M})$. Moreover, the quotient $\wt{\bs{J}_{M}}/Z(\wt{M})\wt{J_{M}}$ 
		is a finite abelian group. 
		
		We first extend $\wt{\lambda}_{M}$ to an irreducible representation $\wt{\lambda}_{M}'$ of $Z(\wt{M})\wt{J_{M}}$, which could be done by choosing a compatible central character. Using \cite{gelbart1982indistinguishability}*{Lemma 2.1}, Frobenius reciprocity and the Mackey formula, we may easily deduce that:
		\begin{itemize}
			\item There exists an irreducible representation $\wt{\bs{\lambda}}_{M}$ of $\wt{\bs{J}_{M}}$, whose restriction to $Z(\wt{M})\wt{J_{M}}$ contains $\wt{\lambda}_{M}'$.
			
			\item There exist positive integers $k$ and $d$, such that  $\car{N_{M}(\wt{\lambda})/\bs{p}(Z(\wt{M}))J_{M}}=dk^{2}$, and 
			\begin{equation}\label{eqrestLambdaM}
				\wt{\bs{\lambda}}_{M}\rest_{Z(\wt{M})\wt{J_{M}}}\cong\bigoplus_{g\in\bs{J}_{M}/N_{M}(\wt{\lambda}_{M})}k\cdot\wt{\lambda}_{M}'^{g}.
			\end{equation}
			\item There exists a finite set $\{\chi_{1},\dots,\chi_{d}\}$ of characters of $J_{M}$ that are trivial on $\bs{p}(Z(\wt{M}))$, such that $\wt{\bs{\lambda}}_{M}\cdot\chi_{i}$ are pairwise inequivalent for $i=1,\dots,d$. Moreover
			\begin{equation}\label{eqlambdaM'ind}
				\mrind_{Z(\wt{M})\wt{J_{M}}}^{\wt{\bs{J}_{M}}}\wt{\lambda}_{M}'\cong\bigoplus_{i=1}^{d}k\cdot(\wt{\bs{\lambda}}_{M}\cdot\chi_{i}).
			\end{equation}
			
		\end{itemize}
		
		For $g\in I_{M}(\wt{\bs{\lambda}}_{M})$, restricting $\wt{\bs{\lambda}}_{M}^{g}$ and $\wt{\bs{\lambda}}_{M}$ to $\wt{J_{M}^{g}}\cap\wt{J_{M}}$ and using \eqref{eqrestLambdaM}, there exists $g'\in\bs{J}_{M}$ such that $g$ intertwines $\wt{\lambda}_{M}^{g'}$ with $\wt{\lambda}_{M}$. Thus $g'g\in I_{M}(\wt{\lambda}_{M})$, implying that $g\in\bs{J}_{M}$. So the intertwining set $I_{M}(\wt{\bs{\lambda}}_{M})$ equals $\bs{J}_{M}$. Using Lemma \ref{lemmacomindcusp}, the compact induction $\wt{\pi}=\mrind_{\wt{\bs{J}_{M}}}^{\wt{M}}\wt{\bs{\lambda}}_{M}$ is a genuine irreducible cuspidal representation of $\wt{M}$. 
		
		Let $\mfs_{M}$ be the corresponding inertial equivalence class of $\wt{M}$ that contains $\wt{\pi}$. Let $M^{0}$ denote the subgroup of $M$ as the intersection of the kernel of all the unramified characters of $M$. Then by definition $Z(\wt{M})\cap\wt{M^{0}}$ is contained in $\wt{J_{M}}$. Moreover, we have $J_{M}=\bs{J}_{M}\cap M^{0}$, which induces an embedding $\bs{J}_{M}/J_{M}\hookrightarrow M/M^{0}$. For each $i=1,\dots,d$, the character $\chi_{i}$, regarded as character of $\bs{J}_{M}/J_{M}$, can be extended to a character of $M/M^{0}$, which is still denoted by $\chi_{i}$. Then, taking the compact induction functor $\mrind_{\wt{\bs{J}_{M}}}^{\wt{M}}$ for \eqref{eqlambdaM'ind}, we get 
		$$\mrind_{Z(\wt{M})\wt{J_{M}}}^{\wt{M}}\wt{\lambda}_{M}'\cong\bigoplus_{i=1}^{d}k\cdot(\wt{\pi}\chi_{i}).$$
		Thus by \cite{bushnell1998smooth}*{Proposition 5.2}, $(\wt{J}_{M},\wt{\lambda}_{M})$ is an $\mfs_{M}$-type, which finishes the proof.

	\end{proof}
	
	We further assume $(\wt{J},\wt{\lambda})$ to be a simple type of $\wt{G}$.
	
	\begin{theorem}\label{thmJPcoveringpair}
		
		Let $(\wt{J},\wt{\lambda})$ be a simple type of $\wt{G}$. Then $(\wt{J}_{P},\wt{\lambda}_{P})$ is a covering pair of $(\wt{J_{M}},\wt{\lambda}_{M})$, and both $(\wt{J},\wt{\lambda})$ and $(\wt{J}_{P},\wt{\lambda}_{P})$ are types of $\wt{G}$.
		
	\end{theorem}
	
	Theorem \ref{thmJPcoveringpair} will be proved in the next section.
	
	\begin{remark}
		
		Indeed, Theorem \ref{thmJPcoveringpair} is expected to be true for homogeneous types, which, however, will not be discussed in this article and will be written down elsewhere.
		
	\end{remark}
	
	In particular, if $(\wt{J},\wt{\lambda})$ is a maximal simple type of $\wt{G}$, then it is a type related to a cuspidal inertial equivalence class of $\wt{G}$.
	
	Finally, we have the following theorem, stating that intertwining of twisted simple types implies conjugacy of corresponding weak equivalence classes:
	
	\begin{theorem}\label{thmsimpletypeconjugacy}
		
		Let $(\wt{J},\wt{\lambda})$, $(\wt{J}',\wt{\lambda}')$ be twisted simple types of $\wt{G}$, attached to hereditary orders $\mfa$, $\mfa'$ in $A$. Suppose that $\mfa$ and $\mfa'$ are conjugate by $G$, and $\wt{\lambda}$ intertwines with $\wt{\lambda}'$  in $G$. Then, there exists $g\in G$ such that $J'=J^{g}$ and $[\wt{\lambda}']\cong[\wt{\lambda}^{g}]$. 
		
		\begin{proof}
			
			The argument follows from that of \cite{bushnell129admissible}*{Theorem 5.7.1} (also see \cite{secherre2012smooth}*{Theorem 6.1}). As in \emph{loc. cit.}, up to $G$-conjugacy we may assume that there exist an $r$-dimensional vector space $V$ over $F$, and two strict simple stratum $[\mfa,u,0,\beta]$ and $[\mfa,u,0,\beta']$ in $A=\mrend_{F}(V)$ such that
			\begin{enumerate}
				
				\item $E=F[\beta]$ and $E'=F[\beta']$ are fields of degree $d$ over $F$, having the same residue field denoted by $\bs{l}$. Write $m=r/d$. Let $B=\mrend_{E}(V_{E})$, $\mfb=B\cap \mfa$. We may also assume that $H^{1}=H^{1}(\beta,\mfa)=H^{1}(\beta',\mfa)$, $J^{1}=J^{1}(\beta,\mfa)=J^{1}(\beta',\mfa)$ and $J=J(\beta,\mfa)=J(\beta',\mfa)$.
				
				\item There exist a simple character $\theta\in\mcc(\mfa,\beta)\cap\mcc(\mfa,\beta')$, the Heisenberg representation $\eta$ of $\theta$, a $\beta$-extension $\kappa$ of $\eta$, two cuspidal representations $\varrho$ and $\varrho'$ of $\mcm\cong U(\mfb)/U^{1}(\mfb)\cong J/J^{1}$ and the corresponding inflations $\wt{\rho}=\mrinf_{\wt{\mcm}}^{\wt{J}}(\epsilon\cdot\,_{s}\varrho)$ and $\wt{\rho}'=\mrinf_{\wt{\mcm}}^{\wt{J}}(\epsilon\cdot\,_{s}\varrho')$, such that $\wt{\lambda}\cong\wt{\kappa}\otimes\wt{\rho}$ and $\wt{\lambda}'\cong\wt{\kappa}\otimes\wt{\rho}'$.
				
				\item Write $\mcm=\mcg^{1}\times\dots\times\mcg^{t}$, where $\mcg^{1}\cong\dots\cong\mcg^{t}\cong\mrgl_{m_{0}}(\bs{l})$ with $m_{0}=m/t$. Then, there exist a cuspidal representation $\varrho_{0}$ of $\mrgl_{m_{0}}(\bs{l})$ and $g_{0}\in T(\mfb)$, such that $\varrho\cong(\varrho_{0}\boxtimes\dots\boxtimes\varrho_{0})\chi_{g_{0}}$. 
				
			\end{enumerate} 
			
			Let $x\in G$ that intertwines $\wt{\lambda}$ and $\wt{\lambda}'$. Restricting to $\,_{s}H^{1}\cap \,_{s}H^{1x}$, we deduce that $x$ intertwines $\,_{s}\theta$, then $x\in JB^{\times}J$. Using the Bruhat decomposition, we have $JB^{\times}J=JU(\mfb_{m})W(B)U(\mfb_{m})J=JW(B)J$. So we may assume $x\in W(B)$ without loss of generality. 
			Since $$\mrhom_{\wt{J}^{x}\cap\wt{J}}(\wt{\lambda}^{x},\wt{\lambda}')=\mrhom_{\wt{J}^{x}\cap\wt{J}}(\wt{\rho}^{x},\wt{\rho}')\neq 0,$$ 
			Proposition \ref{proprhoWBintertwine}, or more precisely its proof, shows that $x$ normalizes $M^{0}(\mfb)$, $M^{1}(\mfb)$ and $\mcm$. Up to changing the $U(\mfb)$-$U(\mfb)$ double coset of $x$, we may further assume that $x\in W(\mfb)$. Thus we have
			\begin{equation}\label{eqintertwinerhorho'}
				0\neq\mrhom_{\wt{J}^{x}\cap\wt{J}}(\wt{\lambda}^{x},\wt{\lambda}')=\mrhom_{\wt{J}^{x}\cap\wt{J}}(\wt{\rho}^{x},\wt{\rho}')= \mrhom_{\wt{\mcm}}((\epsilon\cdot\,_{s}\varrho)^{x},\epsilon\cdot\,_{s}\varrho').
			\end{equation}
			
			Write $x=wx'$ with $w\in W_{0}(\mfb)$ and $x'\in T(\mfb)$. Then $h=w^{-1}g_{0}wx'$ is also in $T(\mfb)$. Moreover, we have $$(\epsilon\cdot\,_{s}\varrho)^{x}\cong(\epsilon\cdot\,_{s}(\varrho_{0}\boxtimes\dots\boxtimes\varrho_{0}))^{g_{0}x}=(\epsilon\cdot\,_{s}(\varrho_{0}\boxtimes\dots\boxtimes\varrho_{0}))^{wh}=(\epsilon\cdot\,_{s}(\varrho_{0}\boxtimes\dots\boxtimes\varrho_{0}))\chi_{h}.$$
			Thus \eqref{eqintertwinerhorho'} implies that $\varrho'\cong(\varrho_{0}\boxtimes\dots\boxtimes\varrho_{0})\chi_{h}$, meaning that $[\wt{\lambda}']= [\wt{\lambda}]$. So the proof is finished.
			
		\end{proof}
		
	\end{theorem}
	
	We extract the following corollary from the argument above.
	
	\begin{corollary}\label{corintertwinesimpletype}
		
		Let $(\wt{J},\wt{\lambda})$ and $(\wt{J},\wt{\lambda'})$ be twisted simple types of $\wt{G}$ attached to the same hereditary order. If $x\in G$ intertwines $\wt{\lambda}$ and $\wt{\lambda'}$, then $x\in JW(\mfb)J$.
		
	\end{corollary}
	
	\begin{remark}
		
		The above theorem \emph{cannot} be relaxed to ``intertwining implies conjugation", as in \cite{bushnell129admissible}*{Theorem 5.7.1} when $n=1$. Indeed, two twisted simple types in the same weak equivalence class are intertwined by an element in $T(\mfb)$. However it is easy to cook up an example that two weakly equivalent twisted simple types are not necessarily conjugate by $G$. 
		
	\end{remark}
	
	\section{Calculation of Hecke algebra}
	
	In this section, we keep the notation of Section \ref{sectionsimpletypes}. Our main goal is to prove Theorem \ref{thmJPcoveringpair} by giving a precise description of the Hecke algebra $\mch(\wt{G},\wt{\lambda})$, where $(\wt{J},\wt{\lambda})$ is a simple type. Along the way, we show that $\mch(\wt{G},\wt{\lambda})$ is indeed an affine Hecke algebra of type A if $\wt{G}$ is a KP-cover or the S-cover.
	
	\subsection{Finite and affine Hecke algebras of type A}\label{subsectionaffineHeckeA}
	
	In this part, we recall known results on finite and affine Hecke algebras of type A (\emph{cf.} \cite{bushnell129admissible}*{\S 5.4} and \cite{solleveld2021affine}*{\S 1}). We fix a real number $z>1$ and two positive integers $t$, $s$.
	
	First we define the finite Hecke algebra
	$$\mch_{0}(t,z)=\mbc[1,[\varsigma_{i}]\mid i=1,\dots,t-1],$$
	where the generators $1$, $[\varsigma_{i}]$ are subject to the following relations
	\begin{enumerate}
		\setcounter{enumi}{-1}
		\item 1 is the unit element.
		\item $[\varsigma_{i}][\varsigma_{j}]=[\varsigma_{j}][\varsigma_{i}]$, $1\leq i,j\leq t-1$, $\abs{i-j}\geq 2$.
		\item $[\varsigma_{i}][\varsigma_{i+1}][\varsigma_{i}]=[\varsigma_{i+1}][\varsigma_{i}][\varsigma_{i+1}]$, $i=1,\dots,t-2$.
		\item $([\varsigma_{i}]-z)([\varsigma_{i}]+1)=0$, $i=1,\dots,t-1$.
		
	\end{enumerate}	
	
	We identify $\varsigma_{i}$ with the transposition of $i$ and $i+1$ in $\mfS_{t}$. For any $\varsigma\in\mfS_{t}$, we consider a reduced decomposition $$\varsigma=\prod_{i=1}^{l}\varsigma_{i}',$$ where  $\varsigma'_{i}\in\{\varsigma_{1},\dots\varsigma_{t-1}\}$ such that $l$ is minimal. Then, the element
	$$[\varsigma]:=\prod_{i=1}^{l}[\varsigma_{i}']\in\mch_{0}(t,z)$$
	is independent of the choice of the above decomposition. Here, by convention we write $[1]=1$. The underlying vector space of the finite Hecke algebra is exactly 
	$$\mbc\pairangone{[\varsigma]\mid\varsigma\in\mfS_{t}}$$
	which is $t!$-dimensional. 
	
	Then we define the linear affine Hecke algebra
	$$\mch(t,z)=\mbc[1,[\varsigma_{i}],[\Pi]\mid i=1,\dots,t-1],$$
	where the generators $1,[\varsigma_{i}],[\Pi]$ are subject to the relations (0)-(3) and 
	\begin{enumerate}
		\setcounter{enumi}{3}
		\item $[\Pi]$ is invertible with respect to the multiplication.
		\item $[\Pi][\varsigma_{i}]=[\varsigma_{i-1}][\Pi]$, $i=2,\dots,t-1$.
		\item $[\Pi]^{2}[\varsigma_{1}]=[\varsigma_{t-1}][\Pi]^{2}$.
	\end{enumerate}
	In particular, we define $[\varsigma_{0}]=[\Pi][\varsigma_{1}][\Pi]^{-1}=[\Pi]^{-1}[\varsigma_{t-1}][\Pi]$.
	
	We consider the (extended) affine Weyl group $\Waff{t}=\mbz^{t}\rtimes\mfS_{t}$, where the semi-direct product is given by $(k_{\varsigma(1)},\dots,k_{\varsigma(t)})\cdot\varsigma=\varsigma\cdot(k_{1},\dots,k_{t})$ for $(k_{1},\dots,k_{t})\in\mbz^{t}$ and $\varsigma\in\mfS_{t}$. 
	
	Still, we identify $\varsigma_{i}$ with elements in $\mfS_{t}\subset \Waff{t}$ as above and $\Pi$ with the element $$(0,\dots,0,1)\cdot \varsigma'\in \Waff{t},$$ 
	where $\varsigma'$ denotes the permutation $1\mapsto 2\mapsto \dots \mapsto t-1\mapsto t\mapsto 1$. We define $\varsigma_{0}=\Pi\varsigma_{1}\Pi^{-1}\in\Waff{t}$.
	
	For any $w\in\Waff{t}$, similarly we may consider a decomposition
	$$w=\Pi^{a}\cdot\prod_{i=1}^{l}\varsigma_{i}',$$
	where $a\in\mbz$ is an integer uniquely determined by $w$, and $\varsigma'_{i}\in\{\varsigma_{0},\varsigma_{1},\dots\varsigma_{t-1}\}$ such that $l$ is minimal. 
	The element
	$$[w]:=[\Pi]^{a}\cdot\prod_{i=1}^{l}[\varsigma_{i}']\in\mch(t,z)$$
	is independent of the choice of the above decomposition. Then, the underlying vector space of the linear affine Hecke algebra is 
	$$\mbc\pairangone{[w]\mid w\in\Waff{t}}.$$
	In particular, $[\Pi]^{t}$ is central in $\mch(t,z)$. We write $$\mca(t):=\mbc[[\lambda]\mid \lambda\in\mbz^{t}\subset\Waff{t}],$$ which turns out to be a commutative algebra isomorphic to the polynomial algebra $$\mbc[X_{1},X_{1}^{-1},\dots,X_{t},X_{t}^{-1}],$$
	where $X_{i}=[(0,\dots,0,1,0,\dots,0)]$ with 1 in the $i$-th row. Then we have an $\mbc$-algebra isomorphism
	$$\mch(t,z)\cong\mca(t)\otimes_{\mbc}\mch_{0}(t,z).$$
	The results listed above could be found in \cite{bushnell129admissible}*{\S 5.4}.
	
	We also slightly generalize the above discussion to define the twisted affine Hecke algebra
	$$\wt{\mch}(t,s,z)=\mbc[1,[\varsigma_{i}],[\Pi],[\zeta]\mid i=1,\dots,t-1],$$
	where the generators $1,[\varsigma_{i}],[\Pi],[\zeta]$ are subject to the relations (0)-(6) and
	\begin{enumerate}
		\setcounter{enumi}{6}
		\item $[\zeta]$ is central and invertible with respect to the multiplication.
		\item $[\zeta]^{s}=[\Pi]^{t}$.
	\end{enumerate}	
	In particular, $\mch(t,z)$ is a subalgebra of $\wt{\mch}(t,s,z)$ of index $s$. If $s=1$, we simply have $\wt{\mch}(t,s,z)=\mch(t,z)$. 
	
	Later on we will see that for KP-covers and the S-cover, the Hecke algebra of a simple type is exactly a twisted affine Hecke algebra. Indeed, for the $S$-cover we have $s=1$, so it is reduced to the linear case.
	
	We consider the twisted affine Weyl group $$\Wafftwist{t}{s}=\mbz_{1/s}^{t}\rtimes\mfS_{t},$$ 
	where $\mbz_{1/s}^{t}$ denotes the sub-lattice of $\mbq^{t}$ generated by $\mbz^{t}$ and $(1/s,1/s,\dots,1/s)$, and the semi-direct product is similarly given as before. We identify $\varsigma_{i}$ and $\Pi$ with elements in $\Waff{t}\subset\Wafftwist{t}{s}$ as before, and $\zeta$ with $(1/s,1/s,\dots,1/s)\in\mbz_{1/s}^{t}\subset\Wafftwist{t}{s}$.
	
	For any $w\in\Wafftwist{t}{s}$, we may consider a decomposition
	$$w=\Pi^{a}\cdot\zeta^{b}\cdot\prod_{i=1}^{l}\varsigma_{i}',$$
	where $a\in\mbz$ and $b\in\{0,1,\dots s-1\}$ are uniquely determined by $w$, and $\varsigma'_{i}\in\{\varsigma_{0},\varsigma_{1},\dots\varsigma_{t-1}\}$ such that $l$ is minimal. 
	Then, the element
	\begin{equation}\label{eqwdecompHtilde}
		[w]:=[\Pi]^{a}\cdot[\zeta]^{b}\cdot\prod_{i=1}^{l}[\varsigma_{i}']\in\wt{\mch}(t,s,z)
	\end{equation}
	is independent of the choice of the above decomposition. The underlying vector space of the twisted affine Hecke algebra is 
	$$\mbc\pairangone{[w]\mid w\in\Wafftwist{t}{s}}.$$
	We write $$\wt{\mca}(t,s):=\mbc[[\lambda]\mid \lambda\in\mbz_{1/s}^{t}\subset\Wafftwist{t}{s}],$$ 
	which turns out to be a commutative algebra isomorphic to the polynomial algebra $$\mbc[X_{1},X_{1}^{-1},\dots,X_{t},X_{t}^{-1},Z,Z^{-1}\mid Z^{s}=X_{1}\dots X_{t}],$$
	where $X_{i}$ is as before and $Z=[(1/s,\dots,1/s)]$. Then we have an $\mbc$-algebra isomorphism
	$$\wt{\mch}(t,s,z)=\wt{\mca}(t,s)\otimes_{\mbc}\mch_{0}(t,z).$$
	In particular, $\mca(t)$ is a subalgebra of $\wt{\mca}(t,s)$ of index $s$. 
	
	We define a canonical hermitian form $\pairang{\cdot}{\cdot}:\wt{\mch}(t,s,z)\rightarrow \mbc$ by linearity, such that $[w]\in \Wafftwist{s}{t}$ forms an orthogonal basis, and moreover
	$$\pairang{[w]}{[w]}=z^{l},$$
	where the length $l$ is defined as in \eqref{eqwdecompHtilde}. Restricting to $\mch(t,z)$, $\wt{\mca}(t,s,z)$ and $\wt{\mca}(t)$, we get the corresponding hermitian forms.
	
	\begin{remark}
		
		The affine Hecke algebras $\mch(t,z)$ and $\wt{\mch}(t,s,z)$ are indeed affine Hecke algebras related to a based root datum of type A (in the sense of \cite{solleveld2021affine}*{\S 1.3} for instance), where $\mca(t)$ (resp. $\wt{\mca}(t,s)$) is the lattice part (i.e. $\mbc[X]$ in \emph{loc. cit.}), and $\mch_{0}(t,z)$ is the finite part (i.e. $\mch(W,q)$ in \emph{loc. cit.}).
		
	\end{remark} 
	
	Finally, we define the induction functor
	$$\mrInd_{\mca(t)}^{\mch(t,z)}:\mrmod(\mca(t))\rightarrow\mrmod(\mch(t,z))$$
	given by the tensor product $(\mch(t,z)\otimes_{\mca(t)}\cdot)$ or the Hom-functor $\mrhom_{\mca(t)}(\mch(t,z),\cdot)$, and the induction functor 
	$$\mrInd_{\wt{\mca}(t,s)}^{\wt{\mch}(t,s,z)}:\mrmod(\wt{\mca}(t,s))\rightarrow\mrmod(\wt{\mch}(t,s,z))$$
	given by the tensor product $(\wt{\mch}(t,s,z)\otimes_{\wt{\mca}(t,s)}\cdot)$ or the Hom-functor $\mrhom_{\wt{\mca}(t,s)}(\wt{\mch}(t,s,z),\cdot)$. In particular, they map finite dimensional modules to finite dimensional modules. Clearly, we have a commutative diagram:
	\begin{equation}\label{eqcommrestparindhecke}
		\xymatrix{
			\mrmod(\wt{\mca}(t,s))\ar[r]^-{\mrInd_{\wt{\mca}(t,s)}^{\wt{\mch}(t,s,z)}} \ar[d]_-{\rest_{\mca(t)}} & \mrmod(\wt{\mch}(t,s,z)) \ar[d]^-{\rest_{\mch(t,z)}}\\
			\mrmod(\mca(t))  \ar[r]_-{\mrInd_{\mca(t)}^{\mch(t,z)}}    & \mrmod(\mch(t,z)) 
		}
	\end{equation}
	
	\begin{remark}\label{remresirr}
		
		In particular, since $[\zeta]$ is contained in the center of $\wt{\mch}(t,s,z)$, it is easy to verify that the restriction of a finite dimensional representation $\pi$ of $\wt{\mch}(t,s,z)$ to $\mch(t,z)$ is irreducible, if and only if $\pi$ is irreducible.
		
	\end{remark}

	\subsection{Finite part of $\mch(\wt{G},\wt{\lambda})$}
	
	In this part, we study the subalgebra of $\mch(\wt{G},\wt{\lambda})$ arising from a cuspidal representation of a finite general linear group, where $(\wt{J},\wt{\lambda})$ is a simple type of $\wt{G}$. 
	
	Fix a maximal open compact subgroup $K$ of $G$ that contains $U(\amax)$ and a splitting $\bs{s}$ of $K$, such that our simple type is of the form $\wt{\lambda}=\wt{\kappa}\otimes\wt{\rho}$, where $\wt{\kappa}$ is the pull-back of a $\beta$-extension, $\wt{\rho}=\mrinf_{\wt{\mcm}}^{\wt{J}}(\epsilon\cdot\,_{s}\varrho)$ with $\varrho=\varrho_{0}\boxtimes\dots\boxtimes\varrho_{0}$ being a cuspidal representation of $\mcm\cong U(\mfb)/U^{1}(\mfb)$, and $\varrho_{0}$ is a cuspidal representation of $\mrgl_{m_{0}}(\bs{l})$.
	
	Let $J'=U(\mfb)\Jonemax$, which is a subgroup of $\Jmax$. So we have $J'/\Jonemax\cong U(\mfb)/U^{1}(\bmax)\cong \mcp$ and $\Jmax/\Jonemax\cong U(\bmax)/U^{1}(\bmax)=\mcg$. We may extend $\varrho$ to an irreducible representation of $\mcp$, which we still denote by $\varrho$.
	
	As in \S \ref{subsectionHeisenbergbeta}, we construct a $\beta$-extension $\kappamax$ of $\Jmax$ related to $\kappa$, and we let $\tildekappamax$ be its pull-back to $\wt{\Jmax}$. Let $\wt{\rho}'=\mrinf_{\wt{\mcm}}^{\wt{J'}}(\epsilon\cdot\,_{s}\varrho)$, which is a genuine irreducible representation of $\wt{J'}$. Finally, let $\wt{\lambda}'=\tildekappamax\rest_{\wt{J'}}\otimes\wt{\rho}'$, which is a genuine irreducible representation of $\wt{J}'$.
	
	\begin{lemma}
		
		We have $\mrind_{\wt{J'}}^{\wt{U(\mfb)}\wt{U^{1}(\mfa)}}\wt{\lambda}'\cong\mrind_{\wt{J}}^{\wt{U(\mfb)}\wt{U^{1}(\mfa)}}\wt{\lambda}$. Then we obtain a canonical isomorphism 
		\begin{equation}\label{eqheckelambda'heckelambda}
			\mch(\wt{G},\wt{\lambda}')\cong\mch(\wt{G},\wt{\lambda}),
		\end{equation} 
		mapping a function $\phi'$ supported on $\wt{J'}y\wt{J'}$ to a function $\phi$ supported on $\wt{J}y\wt{J}$ for any $y\in B^\times$.
		
	\end{lemma}
	
	\begin{proof}
		
		We regard $\wt{\rho}'$ (resp. $\wt{\rho}$) as a genuine representation of $\wt{U(\mfb)}\wt{U^{1}(\mfa)}$ whose restriction to $\,_{s}U^{1}(\mfa)$ is trivial. Then using \eqref{eqrelatedbetaext}, we have
		$$\mrind_{\wt{J'}}^{\wt{U(\mfb)}\wt{U^{1}(\mfa)}}\wt{\lambda}'\cong(\mrind_{\wt{J'}}^{\wt{U(\mfb)}\wt{U^{1}(\mfa)}}\tildekappamax\rest_{\wt{J'}})\otimes\wt{\rho}'\cong(\mrind_{\wt{J}}^{\wt{U(\mfb)}\wt{U^{1}(\mfa)}}\wt{\kappa})\otimes\wt{\rho}\cong\mrind_{\wt{J}}^{\wt{U(\mfb)}\wt{U^{1}(\mfa)}}\wt{\lambda}.$$
		The second statement follows from a similar argument as \cite{bushnell129admissible}*{Proposition 5.5.13}.
		
	\end{proof}
	
	We consider a subalgebra $\mch(\wt{\Jmax},\wt{\lambda}')$ of $\mch(\wt{G},\wt{\lambda}')\cong\mch(\wt{G},\wt{\lambda})$. We write $\bs{q}_{0}=q^{m_{0}f}$ to simplify our notation, where $f$ is the unramified degree of $E/F$. Indeed, $\bs{q}_{0}$ is the cardinality of the field  $\bs{l}_{m_{0}}$, where $\bs{l}_{m_{0}}/\bs{l}$ is of degree $m_{0}$.
	
	\begin{proposition}
		
		We have
		\begin{equation}\label{eqHeckeisofinite}
			\mch_{0}(t,\bs{q}_{0})\cong\mch(\mcg,\varrho)\cong\mch(\wt{U(\bmax)},\wt{\rho}'\rest_{\wt{U(\mfb)}})\cong\mch(\wt{\Jmax},\wt{\lambda}').
		\end{equation}
		Moreover, let $\varsigma\in\mfS_{t}$, let $[\varsigma]$ be the corresponding element in $\mch_{0}(t,\bs{q}_{0})$, let $\phi_{\varsigma}\in\mch(\wt{\Jmax},\wt{\lambda}')$ be the image of $[\varsigma]$ via the isomorphism \eqref{eqHeckeisofinite}. Then the support of $\phi_{\varsigma}$ is $\wt{J'}\mcf_{0}(\varsigma)\wt{J'}$ (\emph{cf.} \eqref{eqStW0bident}).
		
	\end{proposition}
	
	\begin{proof}
		
		The argument of \cite{bushnell129admissible}*{Proposition 5.6.4} works here without change.
		
	\end{proof}
	
	Composing with \eqref{eqheckelambda'heckelambda}, we get an embedding of algebras
	\begin{equation}\label{eqmonoH0H}
		\Psi_{0}:\mch_{0}(t,\bs{q}_{0})\hookrightarrow\mch(\wt{G},\wt{\lambda}).
	\end{equation}
	Indeed, $\Psi_{0}(\mch_{0}(t,\bs{q}_{0}))$ is the subalgebra of $\mch(\wt{G},\wt{\lambda})$ consisting of functions supported on $\wt{J}W_{0}(\mfb)\wt{J}$. This is because $J'\backslash\Jmax/J'\cong U^{0}(\mfb) \backslash U^0(\bmax)/U^{0}(\mfb)$, which is represented by elements in $W_0(\mfb)$.

	\subsection{Proof of Theorem \ref{thmJPcoveringpair}}
	
	Let $(\wt{J},\wt{\lambda})$ be a simple type of $\wt{G}$ as above. We define the related pairs $(\wt{J_{M}},\wt{\lambda}_{M})$ and $(\wt{J_{P}},\wt{\lambda}_{P})$ as in \S \ref{subsectionhomotype}. Then we have a natural isomorphism $\mch(\wt{G},\wt{\lambda}_{P})\cong\mch(\wt{G},\wt{\lambda})$.
	
	Let $\Pi_{E}$, $\zeta_{E}$, $W_{0}=W(r_{0},m_{0},l_{0};t)$ and $W_{0}'=W'(r_{0},m_{0},l_{0};t)$ be defined as in \S \ref{subsectionsimpletypes}. We extend the isomorphism (\ref{eqStW0bident}) to an isomorphism \begin{equation}\label{eqFWaffisom}
		\mcf:\Wafftwist{t}{s_{0}}\cong W_{0}'
	\end{equation}
	by imposing $\Pi_{E}=\mcf(\Pi)$ and $\zeta_{E}=\mcf(\zeta)$. We refer to  \eqref{eqnr0l0}, \eqref{eqdrl0} and \eqref{eqsr-l0} for the definition of $s_{0}=s_{r_{0},l_{0}}$.
	
	The main goal is to extend $\Psi_{0}$ to an embedding $\wt{\mch}(t,s_{0},\bs{q}_{0})\hookrightarrow\mch(\wt{G},\wt{\lambda})$.
	
	\begin{theorem}\label{thmcalhecke}
		
		\begin{enumerate}
			
			\item There exists a non-zero element $\phi_{\Pi}\in\mch(\wt{G},\wt{\lambda})$ supported on $\wt{J}\Pi_{E}\wt{J}$, which is unique up to a scalar.
			
			\item There exists a non-zero element $\phi_{\zeta}\in\mch(\wt{G},\wt{\lambda})$ supported on $\wt{J}\zeta_{E}\wt{J}$, unique up to a scalar in $\mu_{s_{0}}$, such that $\phi_{\zeta}^{s_{0}}=\phi_{\Pi}^{t}$. Such $\phi_{\zeta}$ is central in $\mch(\wt{G},\wt{\lambda})$.
			
			\item For any $\phi_{\Pi}$, $\phi_{\zeta}$ as above, there exists a unique algebra homomorphism \begin{equation}\label{eqhecekembed}
				\Psi:\wt{\mch}(t,s_{0},\bs{q}_{0})\rightarrow\mch(\wt{G},\wt{\lambda}),
			\end{equation}
			which extends $\Psi_{0}$ and satisfies $\Psi([\Pi])=\phi_{\Pi}$ and $\Psi([\zeta])=\phi_{\zeta}$.
			
		\end{enumerate}
		
		Every $\Psi$ as above is an embedding of algebras and preserves the support of functions, in the sense that, for $w\in\Wafftwist{t}{s_{0}} $, the function $\Psi([w])$ is non-zero and supported on $\wt{J}\mcf(w)\wt{J}$. Thus the image of $\Psi$ consists of functions in $\mch(\wt{G},\wt{\lambda})$ supported on $\wt{J}W_{0}'\wt{J}$.
		
	\end{theorem}
	
	\begin{remark}\label{remPsilambdaP}
		
		By abuse of notation, it is convenient to regard $\Psi$ as an embedding
		$$\wt{\mch}(t,s_{0},\bs{q}_{0})\hookrightarrow\mch(\wt{G},\wt{\lambda}_{P}),$$
		by composing the original $\Psi$ with the isomorphism in Lemma \ref{lemmaisoheckelambdaPlambda}.
		
	\end{remark}
	
	\begin{proof}
		
		The proof follows from \cite{secherre2005types}*{Th\'eor\`eme 4.6}, which we shall sketch here. Since $\Pi_{E}$ (resp. $\zeta_{E}$) intertwines $\wt{\lambda}$, using Corollary \ref{corintertwingspace} the space of functions supported on $\wt{J}\zeta_{E}\wt{J}$ (resp. $\wt{J}\Pi_{E}\wt{J}$) is one-dimensional. So up to a scalar in $\mbc^{\times}$, there exists a unique $\phi_{\Pi}\in\mch(\wt{G},\wt{\lambda})$ supported on $\wt{J}\Pi_{E}\wt{J}$. Similarly, up to a scalar in $\mu_{s_{0}}$, there exists a unique $\phi_{\zeta}\in\mch(\wt{G},\wt{\lambda})$ supported on $\wt{J}\zeta_{E}\wt{J}$, such that $\phi_{\zeta}^{s_{0}}=\phi_{\Pi}^{t}$. We need the following two lemmas:
		
		\begin{lemma}[\cite{secherre2005types}*{Lemma 4.12}]
			
			For any $w\in W(\mfb)$, we have:
			$$JwJ\Pi_{E}J\cap W(\mfb)=w\Pi_{E}\quad\text{and}\quad J\Pi_{E}JwJ\cap W(\mfb)=\Pi_{E}w$$
			and
			$$JwJ\Pi_{E}^{-1}J\cap W(\mfb)=w\Pi_{E}^{-1}\quad\text{and}\quad J\Pi_{E}^{-1}JwJ\cap W(\mfb)=\Pi_{E}^{-1}w.$$
			
		\end{lemma}
		
		\begin{lemma}
			
			For any $w\in W(\mfb)$, we have:
			$$JwJ\zeta_{E}J=Jw\zeta_{E}J=J\zeta_{E}wJ=J\zeta_{E}JwJ$$
			and
			$$JwJ\zeta_{E}^{-1}J=Jw\zeta_{E}^{-1}J=J\zeta_{E}^{-1}wJ=J\zeta_{E}^{-1}JwJ.$$
			
		\end{lemma}
		
		\begin{proof}
			
			It follows easily from the fact that $\zeta_{E}$ normalizes $J$ and $\zeta_{E}$ commutes with $w$.
			
		\end{proof}
		
		As a corollary of the above two lemma, we have:
		
		\begin{corollary}\label{corphifinvolution}
			
			\begin{enumerate}
				
				\item For any $w\in W_{0}$ and $ f\in\mch(\wt{G},\wt{\lambda})$ supported on $\wt{J}w\wt{J}$, the functions $\phi_{\Pi_{E}}\ast f$, $f\ast\phi_{\Pi_{E}}$ and $\phi_{\zeta_{E}}\ast f=f\ast\phi_{\zeta_{E}}$ are supported on
				$\wt{J}\Pi_{E}w\wt{J},\ \wt{J}w\Pi_{E}\wt{J},\ \wt{J}\zeta_{E}w\wt{J}$
				respectively.
				
				\item The elements $\phi_{\Pi_{E}}$ and $\phi_{\zeta_{E}}$ are invertible in $\mch(\wt{G},\wt{\lambda})$.
				
			\end{enumerate}
			
		\end{corollary}
		
		\begin{proof}
			
			The statement (1) follows from Proposition \ref{propintertwinelambdasimple} and the above two lemmas. 
			
			For statement (2), we consider $\phi_{\Pi_{E}}\ast \phi_{\Pi_{E}^{-1}}$ and $\phi_{\zeta_{E}}\ast \phi_{\zeta_{E}^{-1}}$. Using statement (1) both of them are supported on $\wt{J}$. Using Corollary \ref{corintertwingspace}, there exist complex numbers $c_{1},c_{2}$ such that $\phi_{\Pi_{E}}\ast \phi_{\Pi_{E}^{-1}}=c_{1}\cdot e$ and $\phi_{\zeta_{E}}\ast \phi_{\zeta_{E}^{-1}}=c_{2}\cdot e$, where $e$ denotes the unit element in $\mch(\wt{G},\wt{\lambda})$. Then we may show that both $c_{1}$ and $c_{2}$ are non-zero following a similar argument to \cite{secherre2005types}*{Lemme 4.14}. So statement (2) is also proved.
			
		\end{proof}
		
		Come back to the original proof. We define $\Psi$ as in statement (3). We need to show that $\Psi$ indeed extends to a homomorphism of algebras. It  remains to show the following lemma, which follows from the same argument as \cite{secherre2005types}*{Lemme 4.15}.
		
		\begin{lemma}
			We have $\phi_{\Pi_{E}}\ast\Psi([\varsigma_{i}])\ast\phi_{\Pi_{E}}^{-1}=\Psi([\varsigma_{i-1}])$ for $i=2,\dots,t-1$ and $\phi_{\Pi_{E}}^{2}\ast\Psi([\varsigma_{1}])\ast\phi_{\Pi_{E}}^{-2}=\Psi([\varsigma_{t-1}])$.
			
		\end{lemma}
		
		Finally, we need to show that $\Psi$ is an embedding and preserves the support. We only need to show that for $w\in\Wafftwist{t}{s_{0}} $, the function $\Psi([w])$ is non-zero and supported on $\wt{J}\mcf(w)\wt{J}$. Then the injectivity follows from Corollary \ref{corintertwingspace}. Using \cite{secherre2005types}*{Proposition 4.16}, whose statement and proof can be transplanted here without change, the function $\Psi([w'])$ is non-zero and supported on $\wt{J}\mcf(w')\wt{J}$ for $w'\in\Waff{t}$. Since every $w\in\Wafftwist{t}{s_{0}}$ can be written as $\zeta^{b}w'$ with $b\in\{0,1,\dots,s-1\}$ and $w'\in\Waff{t}$, using Corollary \ref{corphifinvolution} the function $\Psi([w])=\phi_{\zeta}^{b}\ast\Psi([w'])$ is non-zero and supported on $\wt{J}\mcf(w)\wt{J}$. So we finish the proof of Theorem \ref{thmcalhecke}.
		
	\end{proof}
	
	\begin{remark}
		
		Theorem \ref{thmcalhecke} dates back to \cite{bushnell129admissible}*{Theorem 5.6.6}, which unfortunately cannot be adapted here. Indeed, in \emph{loc. cit.} they used the fact that $\Pi_{E}$ normalizes $\wt{\lambda}$, which is not true here if $n_{0}=n_{r_{0},l_{0}}>1$.
		
	\end{remark}
	
	As a corollary of Proposition \ref{propKPSW=W'} and Theorem \ref{thmcalhecke}, we have
	
	\begin{corollary}\label{corcalHeckeKPS}
		
		When $\wt{G}$ is either a KP-cover or the S-cover, (\ref{eqhecekembed}) is an isomorphism of algebras
		$$\Psi:\wt{\mch}(t,s_{0},\bs{q}_{0})\cong\mch(\wt{G},\wt{\lambda}).$$
		
	\end{corollary}
	
	The following corollary follows from the same proof as \cite{bushnell129admissible}*{Corollary 5.6.17}.
	
	\begin{corollary}\label{corpsiisometry}
		
		Under the assumption of Corollary \ref{corcalHeckeKPS}, we may choose $\Psi$ such that it is an isometry with respect to the hermitian form of $\wt{\mch}(t,s_{0},\bs{q}_{0})$ and $\mch(\wt{G},\wt{\lambda})$. 
		
		More precisely, choose $\Psi$ such that $\Psi([\Pi])\ast\overline{\Psi([\Pi])}=1$ and $\Psi([\zeta])\ast\overline{\Psi([\zeta])}=1$. Then for the canonical hermitian form $\pairang{\cdot}{\cdot}$ of $\wt{\mch}(t,s_{0},\bs{q}_{0})$ and the hermitian form $\mrdim(\wt{\lambda})^{-1}\cdot h_{G}$ of $\mch(\wt{G},\wt{\lambda})$, we have
		$$\mrdim(\wt{\lambda})^{-1}\cdot h_{G}(\Psi(x),\Psi(y))=\pairang{x}{y},\quad x,y\in\wt{\mch}(t,s_{0},\bs{q}_{0}).$$
		
	\end{corollary}

	Finally we are able to prove Theorem \ref{thmJPcoveringpair}.
	
	\begin{proof}[Proof of Theorem \ref{thmJPcoveringpair}]
		
		To show that $(\wt{J_{P}},\wt{\lambda}_{P})$ is a covering pair of $(\wt{J_{M}},\wt{\lambda}_{M})$, it remains to verify the condition (3) in \S \ref{subsectioncoverpair}.
		
		Indeed, given  a parabolic subgroup $P'$ of $G$ with a Levi factor $M$, we may choose an element $z\in Z(\wt{M})\cap\wt{J} W_{0}'\wt{J}$ that is strongly $(\wt{P'},\wt{J})$-positive. For instance, if $P'$ denotes the standard upper triangular parabolic group with respect to the $E$-basis of $V_{E}$ we fixed before, the element $$z=\bs{s}(\mrdiag(\varpi_{F}^{ns_{1}}I_{r_{0}},\varpi_{F}^{ns_{2}}I_{r_{0}},\dots,\varpi_{F}^{ns_{t}}I_{r_{0}}))$$
		satisfies the conditions we want for any sequence of integers $s_{1}>s_{2}>\dots>s_{t}$. Then, using Lemma \ref{lemmaisoheckelambdaPlambda} and Theorem \ref{thmcalhecke}, there exists an invertible element $\phi_{z}$ in $\mch(\wt{G},\wt{\lambda}_{P})$ supported on $\wt{J_{P}}z\wt{J_{P}}$, verifying the condition (3). 
		
		So $(\wt{J_{P}},\wt{\lambda}_{P})$ is a covering pair of $(\wt{J_{M}},\wt{\lambda}_{M})$. The rest follows from Theorem \ref{thmtypecovering}, Theorem \ref{thmlambdaMtype} and the fact $\mrind_{\wt{J}_{P}}^{\wt{J}}\wt{\lambda}_{P}\cong\wt{\lambda}$.
		
	\end{proof}
	
	\subsection{Lattice part of $\mch(\wt{G},\wt{\lambda})$}
	
	We also study another subalgebra of $\mch(\wt{G},\wt{\lambda})$. Since $(\wt{J_{P}},\wt{\lambda}_{P})$ is a covering pair of $(\wt{J_{M}},\wt{\lambda}_{M})$, we have an embedding
	$$t_{P}:\mch(\wt{M},\wt{\lambda}_{M})\rightarrow\mch(\wt{G},\wt{\lambda}_{P}).$$
	Composing with the isomorphism $\mch(\wt{G},\wt{\lambda}_{P})\cong\mch(\wt{G},\wt{\lambda})$ given by Lemma \ref{lemmaisoheckelambdaPlambda}, we get an embedding
	$$\mch(\wt{M},\wt{\lambda}_{M})\rightarrow\mch(\wt{G},\wt{\lambda}),$$
	which we still denote by $t_{P}$. We first study $\mch(\wt{M},\wt{\lambda}_{M})$.
	
	\begin{lemma}\label{lemmacomuheckecal}
		
		$\mch(\wt{M},\wt{\lambda}_{M})$ is  isomorphic to the commutative group algebra $\mbc[T_{0}]$. 
		
	\end{lemma}
	
	\begin{proof}
		
		By Proposition \ref{propintertwinelambdaM}, we have $N_{M}(\wt{\lambda})=I_{M}(\wt{\lambda})=T_{0}J_{M}$. Since $T_{0}$ is abelian and normalizes $J_{M}$, it is easy to verify that  $\mch(\wt{M},\wt{\lambda}_{M})$ is commutative with respect to the convolution. Since $T_{0}$ is a free abelian group of rank $t$, we may choose free generators $b_{1}=I_{r},b_{2},\dots,b_{t}\in T_{0}$. Choose functions $\phi_{b_{1}},\dots,\phi_{b_{t}}\in \mch(\wt{M},\wt{\lambda}_{M})$ such that $\phi_{b_{i}}$ is supported on $b_{i}\wt{J_{M}}$ for each $i$, and $\phi_{b_{1}}\ast \phi_{b_{1}}=\phi_{b_{1}}$. Then we get an isomorphism of algebras $\mbc[T_{0}]\rightarrow\mch(\wt{M},\wt{\lambda})$ which maps $b_{i}$ to $\phi_{b_{i}}$. 
		
	\end{proof}
	
	\begin{remark}
		
		If $\mch(\wt{G},\wt{\lambda})_{\wt{M}}$ is a subalgebra of $\mch(\wt{G},\wt{\lambda})$ (\emph{cf.} \S \ref{subsectioncoverpair}), which is true for KP-covers and the S-cover (\emph{cf.} Corollary \ref{corcalHeckeKPS}), then the map $t_{P}$ preserves the support. Thus, the subalgebra $t_{P}(\mch(\wt{M},\wt{\lambda}_{M}))$ of $\mch(\wt{G},\wt{\lambda})$ consists of functions in $\mch(\wt{G},\wt{\lambda})$ supported on $\wt{J}T_{0}\wt{J}$.
		
	\end{remark} 
	
	We expect the following conjecture in general.
	
	\begin{conjecture}\label{conjHGlambdaaffine}
		
		$\mch(\wt{G},\wt{\lambda})$ is an affine Hecke algebra of type $A$. More precisely, we have an isomorphism of algebras
		$$t_{P}(\mch(\wt{M},\wt{\lambda}_{M}))\otimes_{\mbc}\Psi_{0}(\mch_{0}(t,\bs{q}_{0}))\cong \mch(\wt{G},\wt{\lambda})$$
		where $t_{P}(\mch(\wt{M},\wt{\lambda}_{M}))$ corresponds to the lattice part and $\Psi_{0}(\mch_{0}(t,\bs{q}_{0}))$ corresponds to the finite part of an affine Hecke algebra of type $A$, and the product structure on the left-hand side is defined accordingly.
		
	\end{conjecture}
	
	When $\wt{G}$ is either a KP-cover or the S-cover, by Corollary \ref{corcalHeckeKPS} the conjecture is verified.
	
	\begin{remark}
		
		It is expected that all the results listed in Section 7 hold for twisted simple types as well with minor changes. However, it seems that our proof above cannot be adapted to this case directly.
		
	\end{remark}

	\section{Maximal simple types}
	
	In this section, we further study maximal simple types and state our main theorem of constructing cuspidal representations of $\wt{G}$, as well as its Levi subgroup $\wt{M}$. 
	
	\subsection{Maximal simple types of $\wt{G}$}
	
	We keep the notation of Section \ref{sectionsimpletypes}. In particular, we call the hereditary order $\mfa$ in $A$, or the related simple stratum $[\mfa,u,0,\beta]$ \emph{maximal} if the corresponding hereditary order $\mfb=B\cap \mfa$ is a maximal hereditary order in $B$ with respect to the containment. 
	
	Fix a strict maximal simple stratum $[\mfa,u,0,\beta]$ as before. We construct groups $H^{1}$, $J^{1}$, $J$ and related representations $\theta$, $\eta$, $\kappa$. In this case,  
	$$J/J^{1}\cong U(\mfb)/U^{1}(\mfb)\cong\mcg=\mrgl_{m}(\bs{l}),$$ 
	where $\bs{l}$ denotes the residue field of $E$, and $d=[E:F]$ and $r=md$. Let $\wt{\kappa}$ be the pull-back of $\kappa$ as an irreducible representation of $\wt{J}$. Also, we fix a maximal open compact subgroup $K$ of $G$ that contains $U(\mfa)$ and a splitting $\bs{s}$ of $K$. Let $\varrho$ be a cuspidal representation of $\mcg$. Then the inflation $\wt{\rho}=\mrinf_{\wt{\mcg}}^{\wt{J}}(\epsilon\cdot\,_{s}\varrho)$ is a genuine irreducible representation of $\wt{J}$. Let $\wt{\lambda}=\wt{\kappa}\otimes\wt{\rho}$.
	
	\begin{definition}
		
		A \emph{maximal simple type} of $\wt{G}$ consists of a pair $(\wt{J},\wt{\lambda})$ with $\wt{J}$ and $\wt{\lambda}$ defined as above. Or equivalently, a maximal simple type is a simple type related to a maximal simple stratum.
		
	\end{definition}
	
	We write $\bs{J}=\bs{J}(\beta,\mfa):=E^{\times}J$ as a subgroup of $G$.
	
	\begin{definition}
		
		An \emph{extended maximal simple type} (EMST for short) extending $(\wt{J},\wt{\lambda})$ consists of a pair $(\wt{\bs{J}},\wt{\bs{\lambda}})$, where $\wt{\bs{J}}$ is the pull-back of $\bs{J}$ as above, and $\wt{\bs{\lambda}}$ is a genuine irreducible representation of $\wt{\bs{J}}$ whose restriction to $\wt{J}$ contains $\wt{\lambda}$.
		
	\end{definition}
	
	We list several facts that can be easily deduced from Lemma \ref{lemmacomindcusp} and the argument of Theorem \ref{thmlambdaMtype}.
	
	\begin{proposition}\label{propEMSTG}
		\begin{enumerate}
			\item An EMST $(\wt{\bs{J}},\wt{\bs{\lambda}})$ exists and can be constructed as in Theorem \ref{thmlambdaMtype}.
			
			\item  The normalizer of $\wt{\bs{\lambda}}$ in $G$ is $\bs{J}$.
			
			\item Any two simple types $\wt{\lambda}$ and $\wt{\lambda}'$ contained in  $\wt{\bs{\lambda}}\rest_{\wt{J}}$ are conjugate by $\bs{J}$.
			
			\item  The compact induction $\mrind_{\wt{\bs{J}}}^{\wt{G}}\wt{\bs{\lambda}}$ is a genuine cuspidal representation of $\wt{G}$. 
		\end{enumerate}
	\end{proposition}  
	
	We state our main theorem for genuine cuspidal representations of $\wt{G}$.
	
	\begin{theorem}\label{thmcuspidalconstruction}
		
		\begin{enumerate}
			
			\item Given an EMST of $\wt{G}$, the compact induction $\wt{\pi}=\mrind_{\wt{\bs{J}}}^{\wt{G}}\wt{\bs{\lambda}}$ is a genuine cuspidal representation of $\wt{G}$.
			
			\item Every genuine cuspidal representation of $\wt{G}$ is constructed as in (1).
			
			\item Two EMSTs of $\wt{G}$ are intertwined if and only if they are conjugate by $G$. So the EMST constructing $\wt{\pi}$ is unique up to $G$-conjugacy.
			
		\end{enumerate}
		
	\end{theorem}
	
	We have proved statement (1). Statement (2) will be done in Section \ref{sectionexhaustion}. Now we prove statement (3).
	
	Given two maximal simple types $(\wt{J},\wt{\lambda})$ and $(\wt{J}',\wt{\lambda}')$ of $\wt{G}$ and two corresponding EMSTs $(\wt{\bs{J}},\wt{\bs{\lambda}})$ and $(\wt{\bs{J}}',\wt{\bs{\lambda}}')$, we need to show that $\wt{\bs{\lambda}}$ and $\wt{\bs{\lambda}}'$ are conjugate by $G$ if they are intertwined. Taking the restriction, we deduce that $\wt{\lambda}$ and $\wt{\lambda}'$ are also intertwined. Thus by Theorem \ref{thmsimpletypeconjugacy}, there exists $g\in G$ such that $J'=J^{g}$, $\bs{J}'=\bs{J}^{g}$ and $[\wt{\lambda}']=[\wt{\lambda}^{g}]$. In the maximal case, all the simple types in the class $[\wt{\lambda}]$ can be realized by conjugation of elements in $E^{\times}\subset\bs{J}$, so we may further assume that $\wt{\lambda'}=\wt{\lambda}^{g}$. 
	
	Suppose $h\in G$ intertwines $\wt{\bs{\lambda}}'$ and $\wt{\bs{\lambda}}^{g}$. Then there exists $h'\in\bs{J}'$, such that $h$ intertwines $\wt{\lambda}'^{h'}$ and $\wt{\lambda}^{g}=\wt{\lambda}'$. Using Corollary \ref{corintertwinesimpletype} and the fact that $W(\mfb)=\pairangone{\varpi_{E}}$ in the maximal case, we have $h\in \bs{J}'$. Thus 
	$$0\neq\mrhom_{\wt{\bs{J}}'^{h}\cap\wt{\bs{J}}'}(\wt{\bs{\lambda}}'^{h},\wt{\bs{\lambda}}^{g})=\mrhom_{\wt{\bs{J}}'}(\wt{\bs{\lambda}}',\wt{\bs{\lambda}}^{g}),$$
	implying that  $\wt{\bs{\lambda}}'\cong\wt{\bs{\lambda}}^{g}$. So statement (3) is proved.
	
	Summing up Theorem \ref{thmlambdaMtype}, Theorem \ref{thmsimpletypeconjugacy} and Theorem \ref{thmcuspidalconstruction},
	
	\begin{corollary}\label{cormaximalsimpletypescorrep}
		
		We have a bijection between the set of $G$-conjugacy classes of weak equivalence classes of maximal simple types $(\wt{J},\wt{\lambda})$ and the set of cuspidal inertial equivalence classes $\mfs$ of $\wt{G}$, such that $(\wt{J},\wt{\lambda})$ is an $\mfs$-type.
		
	\end{corollary}
	
	\subsection{Maximal simple types of $\wt{M}$}
	
	In this part, let $M=G_{r_{1}}\times\dots\times G_{r_{t}}$ be a Levi subgroup of $G$. 
	
	\begin{definition}\label{defmaxsimtypeM}
		
		A maximal simple type of $\wt{M}$ consists of a pair $(\wt{J_{M}},\wt{\lambda}_{M})$, such that $J_{M}=J_{1}\times\dots\times J_{t}$ and $\wt{\lambda}_{M}=\wt{\lambda}_{1}\boxtimes\dots\boxtimes\wt{\lambda}_{t}$, where $(\wt{J_{i}},\wt{\lambda_{i}})$ is a maximal simple type of $\wt{G_{r_{i}}}$ for $i=1,\dots,t$.
		
	\end{definition}
	
	We remark that since $\mrdet(J_{i})\subset\mfo_{F}^{\times}$ for each $i$, the group $J=J_{1}\times\dots\times J_{t}$ is block compatible, so the tensor product $\wt{\lambda}_{1}\boxtimes\dots\boxtimes\wt{\lambda}_{t}$ makes sense. 
	
	Also, as already been explained in the last paragraph of \S \ref{subsectionhomotype}, those $(\wt{J_{M}},\wt{\lambda}_{M})$ in \S \ref{subsectionhomotype} consist of all the maximal simple types of $\wt{M}$ in the homogeneous case (saying that the corresponding simple characters $\theta_{i}$ of $\wt{\lambda}_{i}$ are in the same endo-class). 
	
	Let $\bs{J}_{i}=E_{i}^{\times}J_{i}$ in $G_{r_{i}}$ for each $i$ and let $\bs{J}_{M}=\bs{J}_{1}\times\dots\times\bs{J}_{t}$ which contains $J_{M}$. Here, each $E_{i}=F[\beta_{i}]$ is the field related to the stratum $[\mfa_{i},u_{i},0,\beta_{i}]$ in constructing the simple type $\wt{\lambda}_{i}$.
	
	\begin{definition}
		
		An extended maximal simple type (EMST for short) extending $(\wt{J_{M}},\wt{\lambda}_{M})$ is a pair $(\wt{\bs{J}_{M}},\wt{\bs{\lambda}}_{M})$, where $\wt{\bs{\lambda}}_{M}$ is an irreducible representation of $\wt{\bs{J}_{M}}$ whose restriction to $\wt{J_{M}}$ contains $\wt{\lambda}_{M}$.
		
	\end{definition}
	
	\begin{proposition}\label{propintertlambdaM}
		
		For a maximal simple type $(\wt{J_{M}},\wt{\lambda}_{M})$ of $\wt{M}$, we have $I_{M}(\wt{\lambda}_{M})=N_{M}(\wt{\lambda}_{M})$, which is contained in $\bs{J}_{M}$. More generally, given two maximal simple types $(\wt{J_{M}},\wt{\lambda}_{M})$ and $(\wt{J_{M}},\wt{\lambda}_{M}')$ of $\wt{M}$, if $g\in M$ intertwines $\wt{\lambda}_M$ and $\wt{\lambda}_{M}'$, then $g$ is in $\bs{J}_M$ and $\wt{\lambda}_{M}'\cong\wt{\lambda}_{M}^{g}$.
		
	\end{proposition}
	
	\begin{proof}
		
		We need the following lemma, whose proof follows directly from the formula \eqref{eqBDblock}. 
		
		\begin{lemma}\label{lemmagintertwinelambdaM}
			
			For $\wt{\lambda}_{M}=\wt{\lambda}_{1}\boxtimes\dots\boxtimes\wt{\lambda}_{t}$ and $g=\mrdiag(g_{1},\dots,g_{t})\in M$, there exists a character $\chi=\chi_{1}\boxtimes\dots\boxtimes\chi_{t}$ of $J_{M}=J_{1}\times\dots\times J_{t}$ of order dividing $n$, which depends on $g$ but is independent of $\wt{\lambda}_{M}$, such that $\wt{\lambda}_{M}^{g}\cong(\wt{\lambda}_{1}\chi_{1})^{g_{1}}\boxtimes\dots\boxtimes(\wt{\lambda}_{t}\chi_{t})^{g_{t}}$ as representations of $\wt{J_{M}^{g}}$.
			
		\end{lemma}
		
		
		We only need to prove the second stronger statement. Write $\wt{\lambda}_M'=\wt{\lambda}_{1}'\boxtimes\dots\boxtimes\wt{\lambda}_{t}'$, where $(\wt{J_i},\wt{\lambda}_{i}')$ is a maximal simple type of $\wt{G_{r_{i}}}$. Assume that $g=\mrdiag(g_{1},\dots,g_{t})\in M$ intertwines $\wt{\lambda}_M$ and $\wt{\lambda}_{M}'$. Using the above lemma, for each $i$ we have
		$$\mrhom_{\wt{J_{i}^{g_{i}}}\cap\wt{J_{i}}}((\wt{\lambda}_{i}\chi_{i})^{g_{i}},\wt{\lambda}_{i}')\neq 0,$$
		and thus by Corollary \ref{corintertwinesimpletype} we have $g_{i}\in\bs{J}_{i}$ and $g\in\bs{J}_M$. Finally, since $g$ normalizes $J_{M}$, we have $\wt{\lambda}_{M}'\cong\wt{\lambda}_{M}^{g}$.
		
	\end{proof}
	
	The following proposition is parallel to Proposition \ref{propEMSTG}, with a similar proof. We only need to use Proposition \ref{propintertlambdaM}
	to replace Proposition \ref{propintertwinelambdaM} in the argument of Theorem \ref{thmlambdaMtype}.
	
	\begin{proposition}
		
		\begin{enumerate}
			
			\item An EMST $(\wt{\bs{J}_{M}},\wt{\bs{\lambda}}_{M})$ exists and can be constructed as in Theorem \ref{thmlambdaMtype}. 
			
			\item  The normalizer of $\wt{\bs{\lambda}}_{M}$ in $M$ is $\bs{J}_{M}$.
			
			\item Any two simple types $\wt{\lambda}_{M}$ and $\wt{\lambda}_{M}'$ contained in  $\wt{\bs{\lambda}}_{M}\rest_{\wt{J}_{M}}$ are conjugate by $\bs{J}_{M}$.
			
			\item  The compact induction $\mrind_{\wt{\bs{J}_{M}}}^{\wt{M}}\wt{\bs{\lambda}}_{M}$ is a genuine cuspidal representation of $\wt{M}$. 
			
		\end{enumerate}
		
	\end{proposition} 
	
	We state the corresponding theorem for cuspidal representations of $\wt{M}$. The proof  relies on Theorem \ref{thmcuspidalconstruction}.
	
	\begin{theorem}\label{thmcuspidalconstructionLevi}
		
		\begin{enumerate}
			
			\item Given an EMST of $\wt{M}$, the compact induction $\wt{\pi}=\mrind_{\wt{\bs{J}_{M}}}^{\wt{M}}\wt{\bs{\lambda}}_{M}$ is a genuine cuspidal representation of $\wt{M}$.
			
			\item Every genuine cuspidal representation of $\wt{M}$ contains a simple type $\wt{\lambda}_{M}$, and thus is constructed as in (1).
			
			\item Two EMSTs of $\wt{M}$ are intertwined if and only if they are conjugate by $M$. So the EMST constructing $\pi$ is unique up to $M$-conjugacy.
			
		\end{enumerate}
		
	\end{theorem}
	
	\begin{proof}
		
		The statement (1) has been verified.
		
		For statement (2), let $\wt{\pi}$ be a genuine cuspidal representation of $\wt{M}$. Let $M_{r_{1},\dots,r_{t}}^{(n)}:=G_{r_{1}}^{(n)}\times\dots\times G_{r_{t}}^{(n)}$, which is a block compatible subgroup of $M$ of finite index. Let $\wt{\tau}$ be an irreducible subrepresentation of the semi-simple representation $\wt{\pi}\rest_{\bs{p}^{-1}(M_{r_{1},\dots,r_{t}}^{(n)})}$. Since $M_{r_{1},\dots,r_{t}}^{(n)}$ is block compatible, $\wt{\tau}$ can be written as a tensor product
		$$\wt{\tau}_{1}\boxtimes\dots\boxtimes\wt{\tau}_{t}$$
		with $\wt{\tau}_{i}$ a genuine irreducible representation of $\wt{G_{r_{i}}^{(n)}}$. For each $i$, we take an irreducible subrepresentation $\wt{\pi}_{i}$ of the induction $\mrInd_{\wt{G_{r_{i}}^{(n)}}}^{\wt{G_{r_{i}}}}\wt{\tau}_{i}$. Then by Definition \ref{defcuspidal}.(3), $\wt{\pi}_{i}$ is a genuine cuspidal representation of $\wt{G_{r_{i}}}$. Using Theorem \ref{thmcuspidalconstruction}.(2), for each $i$ we may choose a maximal simple type $\wt{\lambda}_{i}$ of $\wt{G_{r_{i}}}$ contained in $\wt{\pi}_{i}$. 
		
		Using Frobenius reciprocity and the Mackey formula, we may show that $\wt{\lambda}_{M}':=\wt{\lambda}_{1}\boxtimes\dots\boxtimes\wt{\lambda}_{t}$ is contained in the semi-simple representation $\mrInd_{\bs{p}^{-1}(M_{r_{1},\dots,r_{t}}^{(n)})}^{\wt{M}}\wt{\tau}$. More precisely, we have
		\begin{equation*}
			\begin{aligned}
				0\neq  &\mrhom_{\wt{J_{1}}}(\mrInd_{\wt{G_{r_{1}}^{(n)}}}^{\wt{G_{r_{1}}}}\wt{\tau}_{1},\wt{\lambda}_{1})\times\dots\times\mrhom_{\wt{J_{t}}}(\mrInd_{\wt{G_{r_{t}}^{(n)}}}^{\wt{G_{r_{t}}}}\wt{\tau}_{t},\wt{\lambda}_{t})\\
				\cong& \mrhom_{\wt{J_{1}}\cap\wt{G_{r_{1}}^{(n)}}}(\wt{\tau}_{1},\wt{\lambda}_{1})\times\dots\times\mrhom_{\wt{J_{t}}\cap\wt{G_{r_{t}}^{(n)}}}(\wt{\tau}_{t},\wt{\lambda}_{t})\\
				\cong&\mrhom_{\wt{J_{M}}\cap \bs{p}^{-1}(M_{r_{1},\dots,r_{t}}^{(n)})}(\wt{\tau}_{1}\boxtimes\dots\boxtimes\wt{\tau}_{t},\wt{\lambda}_{M}')\quad(\text{Since}\ J_M\ \text{and}\ M_{r_{1},\dots,r_{t}}^{(n)}\ \text{are block compatible})\\
				\cong&\mrhom_{\wt{J_{M}}}(\mrInd_{\wt{J_{M}}\cap \bs{p}^{-1}(M_{r_{1},\dots,r_{t}}^{(n)})}^{\wt{J_M}}(\wt{\tau}\rest_{\wt{J_{M}}\cap \bs{p}^{-1}(M_{r_{1},\dots,r_{t}}^{(n)})}),\wt{\lambda}_{M}')\\
				\subset&\mrhom_{\wt{J_{M}}}(\mrInd_{\bs{p}^{-1}(M_{r_{1},\dots,r_{t}}^{(n)})}^{\wt{M}}\wt{\tau},\wt{\lambda}_{M}')\quad (\text{By the Mackey formula})\\
			\end{aligned}
		\end{equation*}
		
		Using \cite{gelbart1982indistinguishability}*{Lemma 2.1}, there exists an irreducible subrepresentation $\wt{\pi}'$ of $\mrInd_{\bs{p}^{-1}(M_{r_{1},\dots,r_{t}}^{(n)})}^{\wt{M}}\wt{\tau}$, such that $\wt{\pi}'$ contains $\wt{\lambda}_{M}'$, and moreover, there exists a character $\chi$ of $M/M_{r_{1},\dots,r_{t}}^{(n)}$ such that $\wt{\pi}\cong\wt{\pi}'\chi$. Write $\chi=\chi_{1}\boxtimes\dots\boxtimes\chi_{t}$ with $\chi_{i}$ being a character of $G_{r_{i}}/G_{r_{i}}^{(n)}$ for each $i$. Then $\wt{\pi}$ contains $\wt{\lambda}_M:=\wt{\lambda}_{M}'\cdot\chi\rest_{J_{M}}=\wt{\lambda}_{1}\cdot\chi_{1}\rest_{J_{1}}\boxtimes\dots\boxtimes\wt{\lambda}_{t}\cdot\chi_{t}\rest_{J_{t}}$, which is a maximal simple type of $\wt{M}$.
		
		Now we may choose an irreducible representation $\wt{\bs{\lambda}}_{M}$
		of $\wt{\bs{J}_{M}}$ containing $\wt{\lambda}_{M}$ and contained in $\wt{\pi}$. Then $(\wt{\bs{J}_{M}},\wt{\bs{\lambda}}_{M})$ is an EMST that compactly induces $\wt{\pi}$.
		
		The proof of statement (3) is similar to that of Theorem \ref{thmcuspidalconstruction}.(3), which we briefly explain here. Let $(\wt{\bs{J}_{M}},\wt{\bs{\lambda}}_{M})$ and $(\wt{\bs{J}_{M}'},\wt{\bs{\lambda}}_{M}')$ be two 
		intertwined EMSTs related to maximal simple types $(\wt{J_{M}},\wt{\lambda}_{M})$ and $(\wt{J_{M}'},\wt{\lambda}_{M}')$ respectively, then $\wt{\lambda}_{M}$ and $\wt{\lambda}_{M}'$ are also intertwined.
		
		We claim that $\wt{\lambda}_{M}$ and $\wt{\lambda}_{M}'$ are conjugate by $M$. Using Lemma \ref{lemmagintertwinelambdaM}, there exists a character $\chi=\chi_{1}\boxtimes\dots\boxtimes\chi_{t}$ of $J_{M}$, such that $\wt{\lambda}_{i}\chi_{i}$ and $\wt{\lambda}_{i}'$ are intertwined for each $i$. Using Theorem \ref{thmsimpletypeconjugacy}, $[\wt{\lambda}_{i}\chi_{i}]$ and $[\wt{\lambda}_{i}']$ are conjugate by  $G_{r_{i}}$ for each $i$. So up to replacing $\wt{\lambda}_{M}$ by its $M$-conjugation, we may assume that $J_{M}=J_{M}'$, and there exists a character $\chi'=\chi_{1}'\boxtimes\dots\boxtimes\chi_{t}'$ of $J_{M}=J_{1}\times\dots\times J_{t}$ of order dividing $n$, such that $\wt{\lambda}_{M}'\cong\wt{\lambda}_{M}\chi'$. Assume $g=\mrdiag(g_{1},\dots,g_{t})\in M$ intertwines $\wt{\lambda}_{M}$ and $\wt{\lambda}_{M}'$. Using Proposition \ref{propintertlambdaM}, we have $g\in\bs{J}_M$ and  $\wt{\lambda}_{M}'\cong\wt{\lambda}_{M}^g$.
		
		Finally we show that $(\wt{\bs{J}_{M}},\wt{\bs{\lambda}}_{M})$ and $(\wt{\bs{J}_{M}'},\wt{\bs{\lambda}}_{M}')$ are also $M$-conjugate. Up to an $M$-conjugation we may assume that $\bs{J}_{M}=\bs{J}_{M}'$, $J_{M}=J_{M}'$ and $\wt{\lambda}_{M}=\wt{\lambda}_{M}'$. So if an element $g\in M$ intertwines $\wt{\bs{\lambda}}_{M}$ and $\wt{\bs{\lambda}}_{M}'$, then it intertwines $\wt{\lambda}_{M}$ and $\wt{\lambda}_{M}^{h}$ for a certain $h\in\bs{J}_{M}$. Thus by Proposition \ref{propintertlambdaM}, $g$ lies in $\bs{J}_{M}$. This shows that $\wt{\bs{\lambda}}_{M}^{g}$ and $\wt{\bs{\lambda}}_{M}'$ are isomorphic.
		
		
	\end{proof}    
	
	
	
	
	\section{Results for a KP-cover or the S-cover}
	
	In this section, we assume $\wt{G}$ to be either a KP-cover or the S-cover. 
	
	Our main goal here is to prove Proposition \ref{proplocusindirred}, Proposition \ref{propZLimage} and Proposition \ref{propdiscinerclass}. These results are subject to our assumption on $\wt{G}$.

	\subsection{Metaplectic tensor product}
	
	One of the features that makes a KP-cover or the S-cover special is that, any genuine irreducible representation of a Levi subgroup of $\wt{G}$ is ``almost" the tensor product of certain genuine irreducible representations of each block.
	
	More precisely, let $M=G_{r_{1}}\times\dots\times G_{r_{t}}$ be a Levi subgroup of $G=\mrgl_{r}(F)$, where $r=r_{1}+\dots+r_{t}$ and each $G_{r_{i}}$ is a subgroup of $G$ isomorphic to $\mrgl_{r_{i}}(F)$. We write $\wt{M}$ and $\wt{G_{r_{i}}}$ for the preimage of $M$ and $G_{r_{i}}$ in $\wt{G}$.
	
	If $\wt{G}$ is the S-cover, then the above decomposition of $M$ is block compatible. Thus, we may define the multi-multiplicative multi-exact tensor product functor
	$$\boxtimes_{i=1}^{t}:\mrrep_{\epsilon}(\wt{G_{r_{1}}})\times\dots\times\mrrep_{\epsilon}(\wt{G_{r_{t}}})\rightarrow\mrrep_{\epsilon}(\wt{M}),$$
	which induces a bijection for irreducible representations
	$$\boxtimes_{i=1}^{t}:\mrirr_{\epsilon}(\wt{G_{r_{1}}})\times\dots\times\mrirr_{\epsilon}(\wt{G_{r_{t}}})\rightarrow\mrirr_{\epsilon}(\wt{M}).$$
	To unify our terminology, we call the above functor $\boxtimes_{i=1}^{t}$ the metaplectic tensor product functor.
	
	If $\wt{G}$ is a KP-cover, things become a bit trickier, since $M$ is not necessarily block compatible. Still, we are able to define the metaplectic tensor product as in \cite{kaplan2022classification}. For each $i$ we choose a genuine character $\wt{\omega}_{i}$ of $Z(\wt{G_{r_{i}}})$ and a genuine character $\wt{\omega}$ of $Z(\wt{G})$ compatible with $(\wt{\omega}_{1},\dots,\wt{\omega}_{t})$, saying that 
	$$(\wt{\omega}_{1}\rest_{Z(\wt{G_{r_{1}}})\cap\wt{G_{r_{1}}^{(n)}}}\boxtimes\dots\boxtimes\wt{\omega}_{t}\rest_{Z(\wt{G_{r_{t}}})\cap\wt{G_{r_{t}}^{(n)}}})\rest_{Z(\wt{G})\cap\wt{G^{(n)}}}=\wt{\omega}\rest_{Z(\wt{G})\cap\wt{G^{(n)}}}.$$
	Let $\mrrep_{\wt{\omega}_{i}}(\wt{G_{r_{i}}})$ (resp. $\mrrep_{\wt{\omega}}(\wt{M})$) denote the category of smooth locally-$\wt{\omega}_{i}$ (resp. locally-$\wt{\omega}$) representations of $\wt{G_{r_{i}}}$ (resp. $\wt{M}$). Then we may define a multi-multiplicative multi-exact metaplectic tensor product functor (depending on the choice of $\wt{\omega}$)
	$$(\boxtimes_{i=1}^{t})_{\wt{\omega}}:\mrrep_{\wt{\omega}_{1}}(\wt{G_{r_{1}}})\times\dots\times\mrrep_{\wt{\omega}_{t}}(\wt{G_{r_{t}}})\rightarrow\mrrep_{\wt{\omega}}(\wt{M}),$$
	which induces a bijection for irreducible representations
	$$(\boxtimes_{i=1}^{t})_{\wt{\omega}}:\mrirr_{\wt{\omega}_{1}}(\wt{G_{r_{1}}})\times\dots\times\mrirr_{\wt{\omega}_{t}}(\wt{G_{r_{t}}})\rightarrow\mrirr_{\wt{\omega}}(\wt{M}).$$
	The metaplectic tensor product for irreducible representations could be realized in a more concrete way. Let $\wt{\pi}_{i}\in\mrirr_{\wt{\omega}_{i}}(\wt{G_{r_{i}}})$ and  $\wt{\pi}_{i}^{(n)}:=\wt{\pi}_{i}\rest_{\wt{G_{r_{i}}^{(n)}}}$ for each $i=1,\dots,t$. Let $M_{r_{1},\dots,r_{t}}^{(n)}:=G_{r_{1}}^{(n)}\times\dots\times G_{r_{t}}^{(n)}$ which is block compatible. We may define the tensor product
	$$\wt{\pi}_{1}^{(n)}\boxtimes\dots\boxtimes\wt{\pi}_{t}^{(n)}$$ 
	as a genuine representation of $\bs{p}^{-1}(M_{r_{1},\dots,r_{t}}^{(n)})$. Then, there exists a positive integer $m_{n,\bs{c};r_{1},\dots,r_{t}}$ depending only on $n,\bs{c},r_{1},\dots,r_{t}$, such that
	\begin{equation}\label{eqmtpirred}
		\mrInd_{\bs{p}^{-1}(M_{r_{1},\dots,r_{t}}^{(n)})}^{\wt{M}}(\wt{\pi}_{1}^{(n)}\boxtimes\dots\boxtimes\wt{\pi}_{t}^{(n)})\cong m_{n,\bs{c};r_{1},\dots,r_{t}}\cdot\bigoplus_{\wt{\omega}}(\wt{\pi}_{1}\boxtimes\dots\boxtimes\wt{\pi}_{t})_{\wt{\omega}},
	\end{equation}
	where $\wt{\omega}$ in the direct sum ranges over all the genuine characters of $Z(\wt{G})$ that are compatible with $(\wt{\omega}_{1},\dots,\wt{\omega}_{t})$. 
	
	We remark that any two different representations in the direct sum of the right-hand side of \eqref{eqmtpirred} are weakly equivalent. It means that they are isomorphic up to a twist by a character $\chi\circ\mrdet$ of $M$, where $\chi$ is a character of $F^{\times}/F^{\times n}$.

	\subsection{Simple types related to the metaplectic tensor product}\label{subsectionreforsimpletypes}
	
	We come back to the setting of \S \ref{subsectionsimpletypes}. Let $r=r_{0}t$. Let $V=V^{1}\oplus\dots\oplus V^{t}$ be a decomposition of $F$-vector spaces, where $V^{i}$ is a $F$-vector space of dimension $r_{0}$. Let $A=\mrend_{F}(V)$ and $G=\mraut_{F}(V)$.
	
	Up to a choice of an $F$-basis for each $V^{i}$, we identify $A^{i}=\mrend_{F}(V^{i})$ with $\mrm_{r_{0}}(F)$, $G^{i}=\mraut_{F}(V^{i})$ with $G_{r_{0}}:=\mrgl_{r_{0}}(F)$. 
	Let $P$ be a parabolic subgroup of $G$ having $M=\mraut_{F}(V^{1})\times\dots\times\mraut_{F}(V^{t})$ as a Levi factor. 
	
	As before, we consider a KP-cover or the S-cover $\wt{G}$ of $G$. Then each $\wt{G^{i}}$ is identified with $\wt{G_{r_{0}}}$, which is a KP-cover or the S-cover of the same type. 
	
	Let $\wt{\pi}_{0}$ be a genuine cuspidal representation of $\wt{G_{r_{0}}}$. In the S-cover case, we consider the metaplectic tensor product
	\begin{equation}\label{eqpidefScover}
		\wt{\pi}=\wt{\pi}_{0}\boxtimes\dots\boxtimes\wt{\pi}_{0}
	\end{equation} as an irreducible cuspidal representation of $\wt{M}$. In the KP-cover case, we choose a compatible genuine character $\wt{\omega}$ of $Z(\wt{G})$, and we take
	\begin{equation}\label{eqpidefKPcover}
		\wt{\pi}=(\wt{\pi}_{0}\boxtimes\dots\boxtimes\wt{\pi}_{0})_{\wt{\omega}}
	\end{equation}
	as an irreducible cuspidal representation of $\wt{M}$. In this case, we are allowed to replace $\wt{\pi}_{0}$ by its twist of a character of order $n$. When we do so, the related $\wt{\pi}$ remains unchanged. 
	
	Let $\mfs$ (resp. $\mfs_{M}$) be the inertial equivalence class of $\wt{G}$ (resp. $\wt{M}$) that contains $(\wt{M},\wt{\pi})$. 
	
	Our first goal is to construct a simple type $(\wt{J},\wt{\lambda})$ of $\wt{G}$ as an $\mfs$-type.
	
	\begin{lemma}
		
		There exists a maximal simple type $(\wt{J_{0}},\wt{\lambda}_{0})$ of $\wt{G_{r_{0}}}$ contained in $\wt{\pi}_{0}$, such that $(\wt{J_{M}},\wt{\lambda}_{M})$ is a maximal simple type of $\wt{M}$ contained in $\wt{\pi}$, where $J_{M}=J_{0}\times\dots\times J_{0}$ and $\wt{\lambda}_{M}:=\wt{\lambda}_{0}\boxtimes\dots\boxtimes\wt{\lambda}_{0}$ is a genuine irreducible representation of $\wt{J_{M}}$. Here,  in the KP-cover case, we are allowed to replace $\wt{\pi}_{0}$ by its twist of an order  $n$ character. 
		
	\end{lemma}
	
	\begin{proof}
		
		In the S-cover case we construct	$(\wt{J_{0}},\wt{\lambda}_{0})$ via Theorem \ref{thmcuspidalconstruction}.(2), then the related $\wt{\lambda}_{M}:=\wt{\lambda}_{0}\boxtimes\dots\boxtimes\wt{\lambda}_{0}$ is contained in $\wt{\pi}$, since $M$ is block compatible.
		
		In the KP-cover case, 
		using the argument of Theorem \ref{thmcuspidalconstructionLevi}.(2) and \eqref{eqmtpirred}, there exists a maximal simple type $(\wt{J_{0}},\wt{\lambda}_{0}')$ of $\wt{G_{r_{0}}}$ contained in $\wt{\pi}_{0}$, such that the related maximal simple type $(\wt{J_{M}},\wt{\lambda}_{M}')$ of $\wt{M}$ is contained in $(\wt{\pi}_{0}\boxtimes\dots\boxtimes\wt{\pi}_{0})_{\wt{\omega}'}$ for some compatible character $\wt{\omega}'$ of $Z(\wt{G})$, where  $\wt{\lambda}_{M}':=\wt{\lambda}_{0}'\boxtimes\dots\boxtimes\wt{\lambda}_{0}'$. Moreover, there exists a character $\chi$ of $F^{\times}/F^{\times n}$, such that $$\wt{\pi}=(\wt{\pi}_{0}\boxtimes\dots\boxtimes\wt{\pi}_{0})_{\wt{\omega}}\cong (\wt{\pi}_{0}\boxtimes\dots\boxtimes\wt{\pi}_{0})_{\wt{\omega}'}\cdot(\chi\circ\mrdet).$$
		We take $\wt{\lambda}_{0}=\wt{\lambda}_0'\cdot(\chi\circ\mrdet)$. Then, $\wt{\pi}_{0}\cdot(\chi\circ\mrdet)$ contains $\wt{\lambda}_{0}$ and $\wt{\pi}$ contains $\wt{\lambda}_{M}$. So in this case, if we replace $\wt{\pi}_{0}$ by $\wt{\pi}_{0}\cdot(\chi\circ\mrdet)$, the proof is finished.
		
	\end{proof}
	
	From now on we pick $(\wt{J_{0}},\wt{\lambda}_{0})$ and $(\wt{J_{M}},\wt{\lambda}_{M})$ as in the above lemma.
	
	We give more information on the construction of $\wt{\lambda}_{0}$. Let $[\mfa_{0},u,0,\beta_{0}]$ be a strict maximal simple stratum in $\mrm_{r_{0}}(F)$, such that $J_{0}=J(\beta,\mfa_{0})$. We fix a maximal compact subgroup $K_{0}$ of $G_{r_{0}}$ that contains $U(\mfa_{0})$, and a splitting $\bs{s}_{0}$ of $K_{0}$. Write $E=F[\beta_{0}]$, and $\bs{l}$ for the residue field of $E$ and $r_{0}=[E:F]m_{0}$. 
	
	Then, there exist a simple character $\theta_{0}\in\mcc(\mfa_{0},\beta)$, its Heisenberg representation $\eta_{0}$, a $\beta$-extension $\kappa_{0}$ of $\eta_{0}$, a cuspidal representation $\varrho_{0}$ of $\mcg_{0}=\mrgl_{m_{0}}(\bs{l})$ and $\wt{\rho}_{0}=\mrinf_{\wt{\mcg_{0}}}^{\wt{J_{0}}}(\epsilon\cdot\,_{s_{0}}\varrho_{0})$, such that $\wt{\lambda}_{0}=\wt{\kappa}_{0}\otimes\wt{\rho}_{0}$. 
	
	Now we focus on the construction of the simple type $(\wt{J},\wt{\lambda})$. We identify $\beta_{0}$ with an element in each $G^{i}\cong\mrgl_{r_{0}}(F)$. We let $\beta=\mrdiag(\beta_{0},\dots,\beta_{0})$ be an element in $M$. By abuse of notation we also write $E=F[\beta]$ as a subfield of $\mrend_{F}(V)$. 
	Write $B=\mrend_{E}(V_{E})$. Let $\mfa$ be a hereditary order in $A$ that is normalized by $E^{\times}$, such that the decomposition $V=\bigoplus_{i=1}^{t}V^{i}$  is properly subordinate to $\Lambda$, the lattice chain  corresponding to $\mfa$. 
	
	Choose a maximal open compact subgroup $K$ containing $U(\mfa)$, and a splitting $\bs{s}$ of $K$. Indeed, our choice could be made, such that via the isomorphism $G^{i}\cong G_{r_{0}}$, the group $K\cap G^{i}$ is identified with $K_{0}$ and the restriction of $\bs{s}$ to each $K\cap G^{i}$ is identified with $\bs{s}_{0}$.
	
	Then $[\mfa,u,0,\beta]$ is a strict simple stratum in $A$. Write $H^{1}=H^{1}(\beta,\mfa)$, $J^{1}=J^{1}(\beta,\mfa)$ and $J=J(\beta,\mfa)$ as before. 
	
	Let $\theta\in\mcc(\mfa,\beta)$ be the transfer of $\theta_{0}$ and let $\eta$ be the Heisenberg representation of $\theta$. Let $\kappa$ be a $\beta$-extension of $\eta$, such that the corresponding $\kappa_{M}$, as a representation of $J_{M}=J\cap M$, equals $\kappa_{0}\boxtimes\dots\boxtimes\kappa_{0}$. Write $\mcm=U(\mfb)/U^{1}(\mfb)=\mcg_{0}\times\dots\times\mcg_{0}$ and let $\varrho=\varrho_{0}\boxtimes\dots\boxtimes\varrho_{0}$ be a cuspidal representation of $\mcm$. Let $\wt{\rho}=\mrinf_{\wt{\mcm}}^{\wt{J}}(\epsilon\cdot\,_{s}\varrho)$ and $\wt{\lambda}=\wt{\kappa}\otimes\wt{\rho}$. 
	
	Then $(\wt{J},\wt{\lambda})$ is a simple type of $\wt{G}$. Moreover, two related pairs  $(\wt{J_{P}},\wt{\lambda}_{P})$ and $(\wt{J_{M}},\wt{\lambda}_{M})$ are constructed as in \S \ref{subsectionhomotype}. Indeed, from our construction the pair $(\wt{J_{M}},\wt{\lambda}_{M})$ is exactly the maximal simple type of $\wt{M}$ we considered above. By Theorem \ref{thmlambdaMtype} and Theorem \ref{thmJPcoveringpair}, $(\wt{J_{M}},\wt{\lambda}_{M})$ is a $\mfs_{M}$-type, $(\wt{J_{P}},\wt{\lambda}_{P})$ is a covering pair of $(\wt{J_{M}},\wt{\lambda}_{M})$, and both $(\wt{J_{P}},\wt{\lambda}_{P})$ and $(\wt{J},\wt{\lambda})$ are $\mfs$-types.
	
	\subsection{A deeper study of Hecke algebras}
	
	We also study the corresponding Hecke algebras. Let $l_{0}$, $d_{0}=d_{r,l_{0}}$, $n_{0}=n_{r_{0},l_{0}}$ and $s_{0}=s_{r_{0},l_{0}}$ be defined as in \S \ref{subsectionsimpletypes}.  Using Theorem \ref{thmcalhecke} and Remark \ref{remPsilambdaP}, we have an isomorphism $$\Psi:\wt{\mch}(t,s_{0},\bs{q}_{0})\rightarrow\mch(\wt{G},\wt{\lambda}_{P}).$$ Moreover, we choose $\Psi$ to satisfy the result of Corollary \ref{corpsiisometry}. Recall that we have an embedding of algebras
	$$t_{P}:\mch(\wt{M},\wt{\lambda}_{M})\rightarrow\mch(\wt{G},\wt{\lambda}_{P}).$$
	So the restriction of $\Psi$ gives an isomorphism
	$$\Psi:\wt{\mca}(t,s_{0})\rightarrow t_{P}(\mch(\wt{M},\wt{\lambda}_{M})).$$
	Let $\Psi_{M}:=t_{P}^{-1}\circ\Psi\rest_{\wt{\mca}(t,s_{0})}$ be the isomorphism between $\wt{\mca}(t,s_{0})$ and $\mch(\wt{M},\wt{\lambda}_{M})$. 
	
	By definition we have equivalence of categories
	$$\bs{\mathrm{M}}_{\wt{\lambda}_{P}}:\mrrep_{\mfs}(\wt{G})\rightarrow\mrmod(\mch(\wt{G},\wt{\lambda}_{P}))\quad\text{and}\quad\bs{\mathrm{M}}_{\wt{\lambda}_{M}}:\mrrep_{\mfs_{M}}(\wt{M})\rightarrow\mrmod(\mch(\wt{M},\wt{\lambda}_{M})).$$
	Composing with equivalence of categories 
	$$\Psi^{*}:\mrmod(\mch(\wt{G},\wt{\lambda}_{P}))\rightarrow\mrmod(\wt{\mch}(t,s_{0},\bs{q}_{0}))\quad\text{and}\quad\Psi_{M}^{*}:\mrmod(\mch(\wt{M},\wt{\lambda}_{M}))\rightarrow\mrmod(\wt{\mca}(t,s_{0}))$$
	induced by $\Psi$ and $\Psi_{M}$, we get equivalence of categories
	$$\wt{\mct}_{G}:\mrrep_{\mfs}(\wt{G})\rightarrow\mrmod(\wt{\mch}(t,s_{0},\bs{q}_{0}))\quad\text{and}\quad\wt{\mct}_{M}:\mrrep_{\mfs_{M}}(\wt{M})\rightarrow\mrmod(\wt{\mca}(t,s_{0})).$$
	We study the compatibility of these functors with parabolic induction. First we need the following lemma:
	
	\begin{lemma}
		The following diagram is commutative:
		$$\xymatrix{
			\mrmod(\mch(\wt{G},\wt{\lambda}_{P})) \ar[r]^-{\Psi^{*}}    & \mrmod(\wt{\mch}(t,s_{0},\bs{q}_{0})) \\
			\mrmod(\mch(\wt{M},\wt{\lambda}_{M})) \ar[u]^-{(t_{P})_{*}}\ar[r]^-{\Psi_{M}^{*}}& \mrmod(\wt{\mca}(t,s_{0})) \ar[u]_-{\mrInd_{\wt{\mca}}^{\wt{\mch}}}
		},$$
		where we write $\mrInd_{\wt{\mca}}^{\wt{\mch}}=\mrInd_{\wt{\mca}(t,s_{0})}^{\wt{\mch}(t,s_{0},\bs{q}_{0})}$ for short.
		
	\end{lemma}
	
	\begin{proof}
		
		It follows from the commutative diagram
		$$\xymatrix{
			\wt{\mch}(t,s_{0},\bs{q}_{0})  \ar[r]^-{\Psi}    & \mch(\wt{G},\wt{\lambda}_{P}) \\
			\wt{\mca}(t,s_{0}) \ar@{^(->}[u]^-{}	\ar[r]^-{\Psi_{M}}  &  \mch(\wt{M},\wt{\lambda}_{M}) \ar[u]_-{t_{P}}
		}$$

	\end{proof}
	
	Combining this lemma with Theorem \ref{thmJacquetcomm}, we have a commutative diagram
	\begin{equation}\label{eqparindcompcover}
		\xymatrix{
			\mrrep_{\mfs}(\wt{G})  \ar[r]^-{\wt{\mct}_{G}}    & \mrmod(\wt{\mch}(t,s_{0},\bs{q}_{0})) \\
			\mrrep_{\mfs_{M}}(\wt{M}) \ar[r]^-{\wt{\mct}_{M}} \ar[u]^-{i_{\wt{P}}^{\wt{G}}} & \mrmod(\wt{\mca}(t,s_{0})) \ar[u]_-{\mrInd_{\wt{\mca}}^{\wt{\mch}}}
		}.
	\end{equation}
	Composing with the diagram \eqref{eqcommrestparindhecke}, we get another commutative diagram
	\begin{equation}\label{eqparindcomp}
		\xymatrix{
			\mrrep_{\mfs}(\wt{G})  \ar[r]^-{\mct_{G}}    & \mrmod(\mch(t,\bs{q}_{0})) \\
			\mrrep_{\mfs_{M}}(\wt{M}) \ar[r]^-{\mct_{M}} \ar[u]^-{i_{\wt{P}}^{\wt{G}}} & \mrmod(\mca(t)) \ar[u]_-{\mrInd_{\mca}^{\mch}}
		},
	\end{equation}
	where we write $\mrInd_{\mca}^{\mch}=\mrInd_{\mca(t)}^{\mch(t,\bs{q}_{0})}$ for short. 
	
	We consider the $\mfS_{t}$-action on those categories, and the $\mfS_{t}$-equivariance of the above functors. First $\mfS_{t}$ acts on $\mca(t)$ and $\wt{\mca}(t,s_{0})$ by interchanging $X_{i}$ and fixing $Z$. Also, using \eqref{eqStW0bident} we identify $\mfS_{t}$ with $W_{0}(\mfb)$, which is a group of representatives of the Weyl group $W(G,M)$. Since $\mfS_{t}$ stabilizes $\wt{M}$, $\wt{\lambda}_{M}=\wt{\lambda}_{0}\boxtimes\dots\boxtimes\wt{\lambda}_{0}$ and $\mfs_{M}$, we have $\mfS_{t}$-actions on $\wt{M}$ and $\mch(\wt{M},\wt{\lambda}_{M})$. This induces $\mfS_{t}$-action on $\mrrep_{\mfs_{M}}(\wt{M})$, $\mrmod(\mch(\wt{M},\wt{\lambda}_{M}))$, $\mrmod(\wt{\mca}(t,s_{0}))$ and $\mrmod(\mca(t))$. 
	
	\begin{lemma}\label{lemmaTMStequiv}
		
		The morphisms $\bs{\mathrm{M}}_{\wt{\lambda}_{M}}$, $\Psi_{M}^{*}$, $\wt{\mct}_{M}$ and $\mct_{M}$ are $\mfS_{t}$-equivariant.
		
	\end{lemma}
	
	We will prove this lemma in the next subsection.
	
	Finally for $\bs{x}=(x_{1},\dots,x_{t})\in\mbc^{t}$, we define a character $\mbc_{\bs{x}}$ of $\mca(t)=\mbc[X_{1},X_{1}^{-1},\dots,X_{t},X_{t}^{-1}]$, such that each $X_{i}$ acts on $\mbc_{\bs{x}}$ by multiplying with $x_{i}$.

	\begin{lemma}\label{lemmaTMcal}
		
		Let $\chi=\nu^{a_{1}}\boxtimes\dots\boxtimes\nu^{a_{t}}$ be a character of $M$ for $a_{1},\dots,a_{t}\in\mbr$, let $n_{0}=n_{r_{0},l_{0}}$ be defined as in \eqref{eqnr0l0} and let $\bs{x}_{\chi}=(\bs{q}_{0}^{-a_{1}n_{0}},\dots,\bs{q}_{0}^{-a_{t}n_{0}})$. 
		Let $$\mct_{M}:\mrrep_{\mfs_{M}}(\wt{M})\rightarrow\mrmod(\mca(t))$$ 
		be the map in \eqref{eqparindcomp}. Then 
		\begin{itemize}
			\item the image of $\wt{\pi}$ under this map is a character of $\mca(t)$ of the form $\mbc_{\bs{z}}$ for a certain $0\neq z\in\mbc$ and $\bs{z}=(z,z,\dots,z)\in\mbc^{t}$. 
			
			\item the image of $\wt{\pi}\cdot\chi$ under this map is the character $\mbc_{\bs{z}}\cdot\mbc_{\bs{x}_{\chi}}$ of $\mca(t)$.
			
		\end{itemize}
		
	\end{lemma}
	
	\begin{proof}
		
		
		For the first statement, using Lemma \ref{lemmaTMStequiv}, $\mct_{M}:\mrrep_{\mfs_{M}}(\wt{M})\rightarrow\mrmod(\mca(t))$ is $\mfS_{t}$-equivariant. Since $\wt{\pi}$ is $\mfS_{t}$-stable, so is its image in $\mrmod(\mca(t))$, which turns out to be a character of the form $\mbc_{\bs{z}}$.
		
		Now we prove the second statement. Let $g$ be any element in $T_{0}$ (\emph{cf.} \S \ref{subsectionsimpletypes}) and let $\phi_{g}\in\mch(\wt{M},\wt{\lambda}_{M})$ be a non-zero function that is supported on $\wt{J_{M}}g\wt{J_{M}}$. We identify the underlying vector space of the modules $\bs{M}_{\wt{\lambda}_{M}}(\wt{\pi}\cdot\chi)$ and $\bs{M}_{\wt{\lambda}_{M}}(\wt{\pi})$. Then the action of $\phi_{g}$ on both modules differ exactly by multiplying by $\chi(g)$. In particular, let $$g_{i}=\mrdiag(I_{m_{0}},\dots,I_{m_{0}},\varpi_{E}^{n_{0}}I_{m_{0}},I_{m_{0}},\dots, I_{m_{0}})\in T_{0},$$
		with $\varpi_{E}^{n_{0}}I_{m_{0}}$ occurring in the $i$-th block. Then from our construction, there exist non-zero elements $\phi_{g_{i}}$ supported on $\wt{J_{M}}g_{i}\wt{J_{M}}$, $i=1,\dots,t$, such that their images under the composition map
		$$\xymatrix{ \mch(\wt{M},\wt{\lambda}) \ar[r]^-{\Psi_{M}^{-1}} & \wt{\mca}(t,s_{0}) \ar[r]^-{\mrres} & \mca(t) }$$ are exactly $X_{i}$, $i=1,\dots,t$. As a result, if we identify the underlying vector space of the image of $\mct_{M}(\wt{\pi}\cdot\chi)$ and $\mct_{M}(\wt{\pi})$ as $\mca(t)$-modules, then the $X_{i}$-action on both modules differ by $$\chi(g_{i})=\nu(\varpi_{F})^{a_{i}n_{0}m_{0}f}=\bs{q}_{0}^{-a_{i}n_{0}}.$$ 
		So we finish the proof of the second statement.

	\end{proof}
	
	
	
	\subsection{Proof of Lemma \ref{lemmaTMStequiv}}
	
	In this part, we give a proof of Lemma \ref{lemmaTMStequiv}, which already contains some ideas of the previous and the next subsections.
	
	By definition, $\bs{\mathrm{M}}_{\wt{\lambda}_{M}}$ is 
	$\mfS_{t}$-equivariant. Also, $\wt{\mct}_{M}$ and $\mct_{M}$ are $\mfS_{t}$-equivariant once we know that $\Psi_{M}^{*}$ is $\mfS_{t}$-equivariant.
	
	Now we prove the statement for $\Psi_{M}^{*}$. Let $X_{1},\dots,X_{t}$ be elements in both $\wt{\mca}(t,s_{0})$ and $\wt{\mch}(t,s_{0},\bs{q}_{0})$. Let $\varphi_{1}=\Psi_{M}(X_{1})$. 
	Let $\varsigma_{1i}$ be the transposition of $1$ and $i$ in $\mfS_{t}$, let $\varphi_{i}=\varphi_{1}^{\varsigma_{1i}}$  
	for each $i=1,\dots,t$. 
	By definition, both $\varphi_{i}$ and $\Psi_M(X_{i})$ are supported on $\wt{J_{M}}g_{i}\wt{J_{M}}$ for each $i=1,\dots,t$, where $$g_{i}:=\mrdiag(I_{m_{0}},\dots,I_{m_{0}},\varpi_{E}^{n_{0}}I_{m_{0}},I_{m_{0}},\dots, I_{m_{0}})\in T_{0}$$ with $\varpi_{E}^{n_{0}}I_{m_{0}}$ occurring in the $i$-th block.
	Thus for each $i=1,\dots,t$, there exists a non-zero complex number $c_{i}$, such that $\varphi_{i}=c_{i}\cdot\Psi_{M}(X_{i})$.
	In particular $c_{1}=1$. 
	
	Let $\bs{c}=(c_{1},\dots,c_{t})$. Our goal is to prove that $c_{1}=\dots=c_{t}=1$, then by definition $\mfS_{t}$ permutes $\Psi_{M}(X_{i})$, $i=1,\dots,t$, which shows that $\Psi_{M}^{*}$ is $\mfS_{t}$-equivariant. 
	
	We first show that $\bs{c}$ is unitary, saying that each $c_{i}$ has absolute value $1$. Note that $\Psi_{M}$ is an isometry, since both $\Psi$ and $t_{P}$ are isometries. Since $\pairang{X_{1}}{X_{1}}=\dots=\pairang{X_{t}}{X_{t}}$, we have $h_{M}(\Psi_{M}(X_{1}),\Psi_{M}(X_{1}))=\dots=h_{M}(\Psi_{M}(X_{t}),\Psi_{M}(X_{t}))$. By definition, we also have $h_{M}(\varphi_{1},\varphi_{1})=\dots=h_{M}(\varphi_{t},\varphi_{t})$. So $1=\abs{c_{1}}^{2}=\dots=\abs{c_{t}}^{2}$.
	
	We define an isomorphism of algebras $T_{\bs{c}}:\mca(t)\rightarrow\mca(t)$ that maps $X_{i}$ to $c_{i}^{-1}X_{i}$ for each $i$, and the induced map $T_{\bs{c}}^{*}:\mrmod(\mca(t))\rightarrow\mrmod(\mca(t))$ via pull-back.
	We define $\mct_{M,\bs{c}}=T_{\bs{c}}^{*}\circ\mct_{M}$ as a functor from $\mrrep_{\mfs_{M}}(\wt{M})$ to $\mrmod(\mca(t))$, thus by definition it is $\mfS_{t}$-equivariant. As in Lemma \ref{lemmaTMcal}.(1),  $\mct_{M,\bs{c}}(\wt{\pi})$ equals $\mbc_{\bs{z}}$ for a certain $\bs{z}=(z,\dots,z)\in\mbc^{t}$. Thus $\mct_{M}(\wt{\pi})$ equals $\mbc_{\bs{z}}\cdot\mbc_{\bs{c}}$. By Lemma \ref{lemmaTMcal}.(2) (and more precisely its argument), $\mct_{M}(\wt{\pi}\cdot\chi)=\mbc_{\bs{z}}\cdot\mbc_{\bs{c}}\cdot\mbc_{\bs{x}_{\chi}}$.
	
	Using the diagram \eqref{eqparindcomp} and
	Remark \ref{remresirr}, the parabolic induction $i_{\wt{P}}^{\wt{G}}(\wt{\pi}\cdot\chi)$ is irreducible if and only if $\mrind_{\mca}^{\mch}(\mbc_{\bs{z}}\cdot\mbc_{\bs{c}}\cdot\mbc_{\bs{x}_{\chi}})$ is irreducible. If $\bs{c}\neq(1,\dots,1)$, up to $\mfS_{t}$-conjugacy, without loss of generality we may assume that $c_{1}=\dots=c_{t'}=1$ and $c_{t'+1},\dots,c_{t}\neq 1$ for a certain $t'=1,\dots,t-1$. Consider $$\chi=\underbrace{1\boxtimes\dots\boxtimes1}_{t'\text{-copies}}\boxtimes\underbrace{\nu^{s}\boxtimes\dots\boxtimes\nu^{s}}_{(t-t')\text{-copies}}.$$ 
	Then using \cite{kaplan2022classification}*{Proposition 6.10}, which in parallel can be stated and proved for the S-cover, there exists $s>0$ such that $i_{\wt{P}}^{\wt{G}}(\wt{\pi}\cdot\chi)$ is not irreducible (in the next subsection $s$ is shown to be $1/n_{0}$). On the other hand,
	$$\mrind_{\mca}^{\mch}(\mbc_{\bs{z}}\cdot\mbc_{\bs{c}}\cdot\mbc_{\bs{x}_{\chi}})=\mrind_{\mca}^{\mch}(\mbc_{\bs{x}})$$
	where $\bs{x}=(z,\dots,z,z\bs{q}_{0}^{-sn_{0}}c_{t'+1},\dots,z\bs{q}_{0}^{-sn_{0}}c_{t})$. The following lemma can be easily deduced from the classification of irreducible representations of $\mch(t,\bs{q}_{0})$, see for instance \cite{solleveld2021affine}*{\S 2.3}.
	
	\begin{lemma}
		
		For $\bs{x}=(x_{1},\dots,x_{t})$ such that $x_{j}\neq x_{i}\bs{q}_{0}^{\pm 1}$ for any $1\leq i\neq j\leq t$, the induced representation $\mrind_{\mca}^{\mch}(\mbc_{\bs{x}})$ is irreducible.
		
	\end{lemma} 
	
	As a result, $\mrind_{\mca}^{\mch}(\mbc_{\bs{z}}\cdot\mbc_{\bs{c}}\cdot\mbc_{\bs{x}_{\chi}})$ is an irreducible representation of $\mch(t,\bs{q}_{0})$, contradictory! So $\bs{c}=(1,\dots,1)$ and Lemma \ref{lemmaTMStequiv} is proved.
	
	\begin{remark}
		
		Our proof here is not so direct, which uses representation theory of $\wt{G}$ and $\mch(t,\bs{q}_{0})$. We wonder if a proof on the level of generators and relations of Hecke algebras could be given, like what has been done in \cite{bushnell129admissible}*{\S 7.6}.
		
	\end{remark}
	
	
	
	
	

	\subsection{Reducible parabolic induction of the metaplectic tensor product of two cuspidal representations}
	
	Keep the notation as above. We further assume that $t=2$, $r=2r_{0}$ and $M=G_{r_{0}}\times G_{r_{0}}$. The goal is to prove the following proposition.
	
	\begin{proposition}\label{proplocusindirred}
		
		Let $\wt{\pi}_{s}$ be the cuspidal representation $\wt{\pi}\cdot(\nu^{-s}\boxtimes\nu^{s})$ of $\wt{M}$ for $s\in\mbr$. Then the parabolic induction $i_{\wt{P}}^{\wt{G}}(\wt{\pi}_{s})$ is reducible if and only if $s=\pm1/2n_{0}$.
		
	\end{proposition}
	
	\begin{proof}
		
		Using the commutative diagram \eqref{eqparindcomp} and Remark \ref{remresirr}, $i_{\wt{P}}^{\wt{G}}(\wt{\pi}_{s})$ is irreducible if and only if $\mrind_{\mca}^{\mch}(\mct_{M}(\wt{\pi}_{s}))$ is irreducible. Using Lemma \ref{lemmaTMcal}, we get $\mct_{M}(\wt{\pi}_{s})=\mbc_{\bs{z}}\cdot\mbc_{(\bs{q}_{0}^{sn_{0}},\bs{q}_{0}^{-sn_{0}})}$. From the study of irreducible representations of $\mch(2,\bs{q}_{0})$ (\emph{cf.} \cite{solleveld2021affine}*{Theorem 2.5}), the induction $\mrind_{\mca}^{\mch}(\mbc_{\bs{z}}\cdot\mbc_{(\bs{q}_{0}^{sn_{0}},\bs{q}_{0}^{-sn_{0}})})$ is reducible if and only if $sn_{0}=\pm1/2$, or equivalently $s=\pm1/2n_{0}$.
		
	\end{proof}
	
	\begin{remark}
		
		This proposition complements \cite{kaplan2022classification}*{Proposition 6.10} by giving an explicit description of $s_{\rho}$ in \emph{loc. cit.} for KP-covers. A similar calculation has also been done in \cite{zou2022metaplectic}*{Proposition 3.13}. A natural and interesting question is to verify that the two calculations give the same value, which in turn will lead us to a study of ``explicit" metaplectic correspondence, i.e. describing the image of the metaplectic lift of a genuine cuspidal representation using the simple type theory.
		
	\end{remark}
	
	\subsection{The image of Zelevinsky standard modules under $\mct_{G}$}\label{subsectionZelevinskystandard}
	
	Keep the notation as above. We consider the parabolic induction $$i_{\wt{P}}^{\wt{G}}(\wt{\pi}\cdot(\nu^{-(t-1)/2n_{0}}\boxtimes\nu^{-(t-3)/2n_{0}}\boxtimes\dots\boxtimes\nu^{(t-1)/2n_{0}})),$$
	which, by Proposition \ref{proplocusindirred} and \cite{kaplan2022classification}*{\S 7}, admits a unique irreducible subrepresentation denoted by $Z(\wt{\pi},t)$ and a unique irreducible quotient denoted by $L(\wt{\pi},t)$. 
	
	\begin{remark}
		
		Although only KP-covers are considered in \cite{kaplan2022classification}*{\S 7}, the corresponding results for the S-cover are also true, whose proofs are similar and simpler.
		
	\end{remark}
	
	Let $\bs{1}_{t,z,\bs{q}_{0}}$ (resp. $\mrst_{t,z,\bs{q}_{0}}$) be the unique irreducible subrepresentation (resp. quotient) of $$\mrind_{\mca}^{\mch}(\mbc_{(z\bs{q}_{0}^{(t-1)/2},z\bs{q}_{0}^{(t-3)/2},\dots,z\bs{q}_{0}^{(1-t)/2})}).$$
	They are one-dimensional representations of $\mch(t,\bs{q}_{0})$. See for instance \cite{solleveld2021affine}*{\S 2.3}. The following proposition follows from Remark \ref{remresirr}, \eqref{eqparindcomp} and Lemma \ref{lemmaTMcal}.
	
	\begin{proposition}\label{propZLimage}
		
		We have $\mct_{G}(Z(\wt{\pi},t))=\bs{1}_{t,z,\bs{q}_{0}}$ and $\mct_{G}(L(\wt{\pi},t))=\mrst_{t,z,\bs{q}_{0}}$.
		
	\end{proposition}
	
	\subsection{Inertial equivalence classes of discrete series}
	
	An inertial equivalence class $\mfs$ of $\wt{G}$ is called \emph{discrete} if $\mrrep_{\mfs}(\wt{G})$ contains a (genuine) discrete series representation (i.e. an essentially square integrable representation) of $\wt{G}$.
	
	First we recall that discrete series representations of $\wt{G}$ could be classified, which was originally a statement of Bernstein in the linear case \cite{zelevinsky1980induced}*{Theorem 9.3}.
	
	\begin{lemma}\label{lemmaclassifydiscrete}
		
		For $r=r_{0}t$ and a genuine cuspidal representation $\wt{\pi}_{0}$ of $\wt{G_{r_{0}}}$, the corresponding $L(\wt{\pi},t)$ defined in \S \ref{subsectionZelevinskystandard}
		is a discrete series representation. Conversely, every genuine discrete series representation of $\wt{G}$ is of such form.
		
	\end{lemma} 
	
	It follows from a similar argument as Jantzen \cite{jantzen2000square}*{\S 2.3}. Note that all the required ingredients in the proof could be found in \cite{kaplan2022classification} for a KP-cover, and could also be modified for the S-cover with minor changes.
	
	Now we may state the main result of this subsection, generalizing Corollary \ref{cormaximalsimpletypescorrep} for a KP-cover or the S-cover. 
	
	\begin{proposition}\label{propdiscinerclass}
		
		We have a bijection between the set of $G$-conjugacy classes of weak equivalence classes of simple types $(\wt{J},\wt{\lambda})$ and the set of discrete inertial equivalence classes $\mfs$ of $\wt{G}$, such that $(\wt{J},\wt{\lambda})$ is an $\mfs$-type.
		
	\end{proposition}
	
	\begin{proof}
		
		First given a simple type $(\wt{J},\wt{\lambda})$ of $\wt{G}$, we show that it is a type of a discrete inertial equivalence class $\mfs$. Indeed, let $(\wt{J_{M}},\wt{\lambda}_{M})$ be the corresponding maximal simple type of $\wt{M}$. In particular, we have $\wt{\lambda}_{M}=\wt{\lambda}_{0}\boxtimes\dots\boxtimes\wt{\lambda}_{0}$, where $\wt{\lambda}_{0}$ is a maximal simple type of $\wt{G_{r_{0}}}$. Then we may choose a genuine cuspidal representation $\wt{\pi}_{0}$ of $\wt{G_{r_{0}}}$ containing $\wt{\lambda}_{0}$ and a compatible genuine character $\wt{\omega}$ of $Z(\wt{G})$ in the KP-case, such that the corresponding genuine cuspidal representation $\wt{\pi}$ of $\wt{M}$, defined as in \eqref{eqpidefScover} or \eqref{eqpidefKPcover}, contains $\wt{\lambda}_{M}$. Let $\mfs_{M}$ be the inertial equivalence class of $\wt{M}$ containing $\wt{\pi}$, and let $\mfs$ be the corresponding inertial equivalence class of $\wt{G}$. Then $(\wt{J_{M}},\wt{\lambda}_{M})$ is an $\mfs_{M}$-type and $(\wt{J},\wt{\lambda})$ is an $\mfs$-type. Moreover by definition $L(\wt{\pi},t)\in\mrrep_{\mfs}(\wt{G})$.
		
		A by-product of the above argument is that: two simple types having the same $G$-conjugacy class of the weak equivalence class correspond to the same $\mfs$.
		
		Conversely, given $\wt{\pi}_{0}$, $\wt{\pi}$ and a discrete inertial equivalence class $\mfs$ of $\wt{G}$ such that $\mrrep_{\mfs}(\wt{G})$ contains $L(\wt{\pi},t)$, the simple type $(\wt{J},\wt{\lambda})$ of $\wt{G}$ constructed as in \S \ref{subsectionreforsimpletypes} is an $\mfs$-type.
		
		If two simple types correspond to the same discrete inertial equivalence class, then they are intertwined. Thus by Theorem \ref{thmsimpletypeconjugacy} their weakly equivalence classes are $G$-conjugate. This shows the injectivity. The surjectivity follows from Lemma \ref{lemmaclassifydiscrete}.

	\end{proof}
	
	\section{Exhaustion of constructing cuspidal representations}\label{sectionexhaustion}
	
	In this section, we prove Theorem \ref{thmcuspidalconstruction}.(2) to finish this article. Let $\wt{\pi}$ be a genuine irreducible representation of $\wt{G}$. We would like to show that if $\wt{\pi}$ is cuspidal, then it contains an EMST, or equivalently it contains a maximal simple type of $\wt{G}$. Our argument largely follows from \cite{bushnell129admissible}*{\S 8} and \cite{secherre2008representations}.
	
	\subsection{The depth 0 case}
	
	First of all, we recall the known result in the depth 0 case as a warm-up (\emph{cf.} \cite{howard2009depth}). 
	
	Assume that $\wt{\pi}$ is cuspidal and of depth 0, where the latter condition means that for the maximal pro-$p$-subgroup $K^{1}$ of a maximal compact subgroup $K$ of $G$,  $$\wt{\pi}^{\,_{s}K^{1}}=\mrhom_{\,_{s}K^{1}}(1,\wt{\pi})\neq 0.$$ 
	Note that all such maximal compact subgroups $K$ are $G$-conjugate to each other, so the above definition does not depend on the choice of $K$. In particular, we choose $K=\mrgl_{r}(\mfo_{F})$ and $K^{1}=I_{r}+\mrm_{r}(\mfp_{F})$ in this subsection. We fix a splitting $\bs{s}:K\rightarrow\wt{G}$.
	
	We consider the null stratum $[\mfa,0,0,\beta]$, where $\mfa=\mrm_{r}(\mfo_{F})$. So by convention we have $H^{1}(\beta,\mfa)=J^{1}(\beta,\mfa)=K^{1}$ and $J(\beta,\mfa)=K$. Also, the only possible simple character and the Heisenberg representation is the identity character of $K^{1}$. Also, we consider the identity character of $K$ as a $\beta$-extension, extending trivially to $\wt{K}$.
	
	By \cite{howard2009depth}*{Theorem 3.10}, there exists a cuspidal representation $\varrho$ of $\mcg=\mrgl_{r}(\bs{k})\cong K/K^{1}$, such that $\wt{\pi}\rest_{\wt{K}}$ contains the inflation $\wt{\rho}=\mrinf_{\wt{\mcg}}^{\wt{K}}(\epsilon\cdot\,_{s}\wt{\varrho})$. 
	
	Since the pair $(\wt{K},\wt{\rho})$ is a maximal simple type with respect to the null stratum $[\mfa,0,0,\beta]$, our proof for depth 0 cuspidal representations is finished.
	
	\subsection{A reduction procedure}
	
	From now on, we assume that our representation $\wt{\pi}$ is not of depth 0. Also, all the strata we will consider are not null.
	
	\begin{definition}\label{defsplitsimplecharacter}
		
		We say that $\wt{\pi}$ contains a
		\begin{enumerate}
			\item \emph{split stratum} if there exists a strict split stratum $[\mfa,u,u-1,b]$ in $A$, such that $\wt{\pi}$ contains the character $\,_{s}\psi_{b}$ of $\,_{s}U^{u}(\mfa)$.
			\item \emph{simple character} if there exist a strict simple stratum $[\mfa,u,0,\beta]$ in $A$ and a simple character $\theta\in\mcc(\mfa,0,\beta)$, such that $\wt{\pi}$ contains $\,_{s}\theta$ of $\,_{s}H^{1}(\beta,\mfa)$.
			\item \emph{split character} if there exist a strict simple stratum $[\mfa,u,l,\beta]$ in $A$ with $l\geq 1$, a simple character $\theta'\in\mcc(\mfa,l-1,\beta)$ and $c\in\mfp_{\mfa}^{-l}$ such that
			\begin{itemize}
				\item $\wt{\pi}$ contains the character $\,_{s}\vartheta=\,_{s}\theta'\,_{s}\psi_{c}$ of $\,_{s}H^{l}(\beta,\mfa)$.
				\item The derived stratum $[\mfb,l,l-1,s(c)]$ in $B$ is split, where $E=F[\beta]$, $B=\mrend_{E}(V_{E})$, $\mfb=\mfa\cap B$ and $s:A\rightarrow B$ is any tame corestriction.
				
			\end{itemize}
		\end{enumerate}
		
	\end{definition}
	
	The main result of this part is the following theorem, whose proof will be sketched below.
	
	\begin{theorem}\label{thmirredthreecases}
		
		Let $\wt{\pi}$ be a genuine irreducible positive depth representation of $\wt{G}$. Then one of the three cases in Definition \ref{defsplitsimplecharacter} happens. 
		
	\end{theorem}	
	
	The following lemma follows from \cite{bushnell1987hereditary7}*{Theorem $2'$}, whose proof can be modified here without change.
	
	\begin{lemma}
		
		There exists a strict fundamental stratum $[\mfa',u',u'-1,b]$ in $A$, such that the restriction of $\wt{\pi}$ to $\,_{s}U^{u'}(\mfa')$ contains $\,_{s}\psi_{b}$.
		
	\end{lemma}
	
	Starting from this lemma, if the fundamental stratum $[\mfa',u',u'-1,b]$ is split, then by definition $\wt{\pi}$ contains a split stratum, meaning that we are in case (1) of Definition \ref{defsplitsimplecharacter}. 
	
	Otherwise, using \cite{broussous1999minimal}*{Theorem 1.2.5} there exists a simple stratum $[\mfa,u,u-1,\beta]$ in $A$, such that the restriction of $\wt{\pi}$ to $\,_{s}U^{u}(\mfa)$ contains $\,_{s}\psi_{\beta}$. Since $\mcc(\mfa,u-1,\beta)=\{\psi_{\beta}\}$ (\emph{cf.} \cite{bushnell129admissible}*{Proposition 3.2.2}), we have indeed proved that (see also \cite{secherre2008representations}*{Proposition 3.19})
	
	\begin{proposition}\label{proppisplitsimpleckcharacter}
		
		Let $\wt{\pi}$ be defined as before, then either $\wt{\pi}$ contains a split stratum, or there exist a simple stratum $[\mfa,u,l,\beta]$ in $A$ 
		and a simple character $\theta\in\mcc(\mfa,l,\beta)$, such that $\wt{\pi}$ contains $\,_{s}\theta$.
		
	\end{proposition} 
	
	Assume that we are in the second case of Proposition \ref{proppisplitsimpleckcharacter}. We also choose $[\mfa,u,l,\beta]$ and $\theta$ such that the rational number $l/e(\mfa|\mfo_{F})$ is minimal. 
	
	If $l=0$, then we are in case (2) of Definition \ref{defsplitsimplecharacter}.
	
	Now assume that $l\geq 1$. We fix a character $\vartheta$ of $H^{l}(\beta,\mfa)$ extending $\theta$, such that $\wt{\pi}$ contains $\,_{s}\vartheta$. By the construction of simple characters, there exist $\theta'\in\mcc(\mfa,l-1,\beta)$ and $c\in\mfp_{\mfa}^{-l}$ such that $\vartheta=\theta'\psi_{c}$. We fix a tame corestriction $s:A\rightarrow B$, and we get a derived stratum $[\mfb,l,l-1,s(c)]$ in $B$. 
	
	\begin{proposition}
		
		$[\mfb,l,l-1,s(c)]$ is a split stratum in $B$.
		
	\end{proposition}
	
	The proposition follows from \cite{bushnell129admissible}*{\S 8.1} or \cite{secherre2008representations}*{\S 3}. Indeed, the corresponding argument in \emph{loc. cit.} essentially concerns only open compact pro-$p$-groups and their representations. In this case we have unique splittings for corresponding groups and representations, so the same argument could be modified verbatim to our case without any difficulty.
	
	Thus in this case $\wt{\pi}$ contains the split character $\,_{s}\vartheta=\,_{s}\theta'\,_{s}\psi_{c}$, saying that we are in case (3) of Definition \ref{defsplitsimplecharacter}. 
	
	\subsection{Containment of a simple character I, elimination of split strata}\label{subsectioneliminationsplitstrata}
	
	The main goal of the following two subsections is to prove the following theorem.
	
	\begin{theorem}\label{thmpicontainsimplechar}
		
		Let $\wt{\pi}$ be a genuine cuspidal positive depth representation of $\wt{G}$. Then it contains a simple character.
		
	\end{theorem}
	
	Using Theorem \ref{thmirredthreecases}, the strategy is to show that case (1) and (3) in Definition \ref{defsplitsimplecharacter} cannot happen. 
	
	In this subsection, we sketch the proof in \cite{broussous1999minimal}*{\S 2} to eliminate case (1). Assume that $\wt{\pi}$ contains a split stratum, saying that there exists a split stratum $[\mfa,u,u-1,b]$ in $A$, such that $\wt{\pi}$ contains the character $\,_{s}\psi_{b}\rest_{\,_{s}U^{u}(\mfa)}$. Let $\Lambda$ be the lattice chain related to $\mfa$. 
	
	By definition, the characteristic polynomial $\varphi_{b}$ can be written as the product of two relatively prime polynomials $\varphi_{1}$ and $\varphi_{2}$ in $\bs{k}[X]$. Write $y=\varpi_{F}^{u/\mrgcd(u,e)}b^{e/\mrgcd(u,e)}\in\mfa$ with $e=e(\mfa|\mfo_{F})$. Then using Hensel's lemma, there exist $\Phi_{i}\in\mfo_{F}[X]$ whose modulo $\mfp_{F}$ reduction is $\phi_{i}$ for $i=1,2$, such that the characteristic polynomial of $y$ in $A=\mrend_{F}(V)$ is the product of $\Phi_{1}$ and $\Phi_{2}$. 
	
	Write $V^{i}=\mrker(\Phi_{i}(y))$ for $i=1,2$. Then by construction $V=V^{1}\oplus V^{2}$ is an $E$-decomposition with $E=F[b]$. Let $\Lambda^{i}=\Lambda\cap V^{i}$ be a lattice sequence of $V^{i}$ for $i=1,2$. Then both $\Lambda$, $\Lambda^{1}$, $\Lambda^{2}$ are $E$-pure, and we have $\Lambda=\Lambda^{1}\oplus\Lambda^{2}$ (\emph{cf.} \cite{broussous1999minimal}*{Proposition 2.2.1}) Let $A^{ij}=\mrhom_{F}(V^{i},V^{j})$ be defined as $F$-subalgebra of $A$, let  $M=\mraut_{F}(V^{1})\times \mraut_{F}(V^{2})$ be a Levi subgroup of $G$. 
	
	In \cite{broussous1999minimal}*{\S 2.3}, an $\mfo_{F}$-orders $\mfh$ in $A$ is constructed, satisfying $\mfh\cap A^{ii}=\mfp_{\mfa}^{u}\cap A^{ii}$, $\mfh\cap A^{12}=\mfa\cap A^{12}$ and $\mfh\cap A^{21}=\mfp_{\mfa}^{u+1}\cap A^{21}$. Let $H=1+\mfh$, which is an open compact pro-$p$-subgroup of $G$. By definition $\psi_{b}$ is a character of $H$. Extend the character $\,_{s}\psi_{b}$ of $\,_{s}H$ to a genuine representation of $\wt{H}$, which we denote by $\wt{\psi}_{b}$. 
	
	We have the following proposition, which follows directly from \cite{broussous1999minimal}*{Proposition 2.3.1}.
	
	\begin{proposition}
		
		The intertwining set $I_{G}(\wt{\psi}_{b})$ is contained in $HMH$.
		
	\end{proposition}
	
	Then by verifying condition (1)(2)(4), $(\wt{H},\wt{\psi}_{b})$ is a covering pair of $(\wt{H}\cap\wt{M},\wt{\psi}_{b}\rest_{\wt{H}\cap\wt{M}})$.
	
	Using \cite{broussous1999minimal}*{Proposition 2.4.4}, whose statement and proof can be easily adapted to our case, $\wt{\pi}$ contains $\,_{s}\psi_{b}\rest_{\,_{s}H}$ and $\wt{\psi}_{b}$. Let $P=MN$ be a parabolic of $G$ having a Levi factor $M$ and the unipotent radical $N$. Then using Theorem \ref{thmJacquetcomm} we have
	$$0\neq\mrhom_{\wt{H}}(\wt{\pi},\wt{\psi}_{b})=\mrhom_{\wt{H}\cap\wt{M}}(r_{N}(\wt{\pi}),\wt{\psi}_{b}),$$
	implying that $r_{N}(\wt{\pi})\neq 0$ and contradicting the fact that $\wt{\pi}$ is cuspidal. So $\wt{\pi}$ cannot contain a split stratum.
	
	\subsection{Containment of a simple character II, elimination of split characters}
	
	In this subsection, we sketch the proof in \cite{secherre2008representations}*{\S 4} to eliminate case (3). Thus $\wt{\pi}$ must contain a simple character.
	
	Assume that $\wt{\pi}$ contains a split character, saying that there exist a strict simple stratum $[\mfa,u,l,\beta]$ in $A$ with $l\geq 1$, a simple character $\theta'\in\mcc(\mfa,l-1,\beta)$ and $c\in\mfp_{\mfa}^{-l}$ such that
	\begin{itemize}
		\item $\wt{\pi}$ contains the character $\,_{s}\vartheta=\,_{s}\theta'\,_{s}\psi_{c}$ of $\,_{s}H^{l}(\beta,\mfa)$.
		\item The derived stratum $[\mfb,l,l-1,s(c)]$ in $B$ is split, where $E=F[\beta]$, $B=\mrend_{E}(V_{E})$, $\mfb=\mfa\cap B$ and $s:A\rightarrow B$ is any tame corestriction.
	\end{itemize}
	Let $\Lambda$ be the lattice chain related to $\mfa$. Being regarded as an $\mfo_E$-lattice chain and denoted by $\Lambda_{E}$, it is the lattice chain related to $\mfb$. 
	
	Following the procedure in \S \ref{subsectioneliminationsplitstrata} with $[\mfa,u,u-1,b]$ replaced by $[\mfb,l,l-1,s(c)]$, or more precisely using \cite{secherre2008representations}*{Proposition 4.9}, we may construct a decomposition of $E$-vector spaces $V_{E}=V_{E}^{1}\oplus V_{E}^{2}$ that conforms with $\Lambda_{E}$. Let $V=V^{1}\oplus V^{2}$ be the corresponding decomposition of $F$-vector spaces that conforms with $\Lambda$ with the same underlying space. Define $A^{ij}$ and $M$, and $P=MN$ as in \S \ref{subsectioneliminationsplitstrata}. In particular, we choose $N=I_{r}+A^{12}$. Let $\Omega=U^{1}(\mfb)\Omega_{v-l+1}(\beta,\mfa)$ be a pro-$p$-subgroup of $G$, where $v$ and $\Omega_{i}(\beta,\mfa)$ are defined as in \S \ref{subsectionsimplecharacters}. Up to changing $c$ if necessary, we may also assume that $c=c_{1}+c_{2}$ with $c_{i}\in A^{ii}$ for $i=1,2$.
	
	We extend $\,_{s}\vartheta$ to a genuine character $\wt{\vartheta}$ of $\wt{H^{l}(\beta,\mfa)}$. So $\wt{\pi}$ contains $\wt{\vartheta}$. Also $\Omega$ normalizes $\wt{\vartheta}$. The following proposition follows from \cite{secherre2008representations}*{Th\'eor\`eme 4.3}.
	
	\begin{proposition}\label{propintwintwingIGvartheta}
		
		The intertwining set $I_{G}(\wt{\vartheta})$ is contained in $\Omega M\Omega$.
		
	\end{proposition}
	
	Let $K'=H^{l}(\beta,\mfa)(\Omega\cap N)$, which is a pro-$p$-subgroup of $G$. By construction, the character $\vartheta$ is trivial on $H^{l}(\beta,\mfa)\cap N$. So we extend $\vartheta$ to a character $\xi$ of $K'$ that is trivial on $\Omega\cap N$. Let $\wt{\xi}$ be the extension of $\,_{s}\xi$ to $\wt{K'}$ as a genuine character. Using the argument in \cite{secherre2008representations}*{\S 4.3}, which can be modified here directly, $\wt{\pi}$ contains the character $\wt{\xi}$ as well.
	
	We claim that $(\wt{K'},\wt{\xi})$ is a covering pair of $(\wt{K'}\cap\wt{M},\wt{\xi}\rest_{\wt{K'}\cap\wt{M}})$. Condition (1) follows from the Iwahori decomposition of $H^{i}(\beta,\mfa)$ and $\Omega_{i}(\beta,\mfa)$ and condition (2) is direct. We verify condition (3) for the parabolic subgroup $P=MN$. Choose an element $\zeta=\mrdiag(\varpi_{F}^{n}I_{r_{1}},I_{r_{2}})$ with respect to the decomposition $V=V^{1}\oplus V^{2}$, where $r_{i}=\mrdim_{F}(V^{i})$ for $i=1,2$. Then $\bs{s}(\zeta)$ is in the center $Z(\wt{M})$ and is strongly $(\wt{P},\wt{K'})$-positive. We need to find an element $\phi_{\zeta}$ supported on $\wt{K'}\zeta\wt{K'}$ that is invertible. This follows from a similar argument of \cite{bushnell1999semisimple}*{Corollary 6.6} as well as \cite{secherre2008representations}*{Corollaire 4.6} using Proposition \ref{propintwintwingIGvartheta}. For the opposite parabolic subgroup $P^{-}=MN^{-}$, we just need to use $\zeta^{-1}$ in place of $\zeta$.
	
	Using Theorem \ref{thmJacquetcomm} we have
	$$0\neq\mrhom_{\wt{K'}}(\wt{\pi},\wt{\xi})=\mrhom_{\wt{K'}\cap\wt{M}}(r_{N}(\wt{\pi}),\wt{\xi}),$$
	implying that $r_{N}(\wt{\pi})\neq 0$ and contradicting the fact that $\wt{\pi}$ is cuspidal. So $\wt{\pi}$ cannot contain a split character. Thus Theorem \ref{thmpicontainsimplechar} is proved.
	
	\subsection{From simple characters to simple types}
	
	Still, let $\wt{\pi}$ be a genuine cuspidal positive depth representation of $\wt{G}$. Using Theorem \ref{thmpicontainsimplechar}, we choose a strict simple stratum $[\mfa,u,0,\beta]$ in $A$ and $\theta\in\mcc(\mfa,0,\beta)$, such that $\,_{s}\theta$ is contained in $\wt{\pi}$. In particular, we assume $\mfa$ to be minimal among all the hereditary orders in $A$ that we may choose. We use the abbreviations $H^{1}$, $J^{1}$, $J$, $\bs{J}$ etc. as in Section \ref{sectionsimpletypes}. We fix a maximal open compact subgroup $K$ that contains $U(\mfa)$, and a splitting $\bs{s}$ of $K$.
	
	Let $\eta$ be the Heisenberg representation of $\theta$ and let $\kappa$ be a $\beta$-extension of $\eta$.  By construction $\wt{\pi}$ contains $\,_{s}\eta$. Moreover, using Frobenius reciprocity $\wt{\pi}$ contains an irreducible subrepresentation of the induction $$\mrind_{\wt{J^{1}}}^{\wt{J}}(\epsilon\cdot\,_{s}\eta)=\wt{\kappa}\otimes\mrind_{\wt{J^{1}}}^{\wt{J}}(\epsilon\cdot\,_{s}1),$$ 
	where $\,_{s}1$ denotes the identity character of $\,_{s}J^{1}$. Then there exist	 an irreducible representation $\varrho$ of $\mcm=J/J^{1}=U(\mfb)/U^{1}(\mfb)$ and the corresponding inflation $\wt{\rho}=\mrinf_{\wt{\mcm}}^{\wt{J}}(\epsilon\cdot\,_{s}\varrho)$ as a genuine irreducible representation of $\wt{J}$, such that $\wt{\lambda}:=\wt{\kappa}\otimes\wt{\rho}$ is contained in $\wt{\pi}$. Here, we recall that $\mcm$ is a Levi subgroup of the finite general linear group $\mrgl_{m}(\bs{l})$, where $\bs{l}$ is the residue field of $E=F[\beta]$, $d=[E:F]$ and $m=r/d$.
	
	\begin{proposition}
		
		The pair $(\wt{J},\wt{\lambda})$ is a homogeneous type in $\wt{G}$, or in other words, $\varrho$ is a cuspidal representation of $\mcm$.
		
	\end{proposition}
	
	\begin{proof}
		
		The proof follows from \cite{secherre2008representations}*{Proposition 5.15}. Assume that $\varrho$ is not cuspidal. Then, there exist a proper parabolic subgroup $\mcp'$ of $\mcm$ with $\mcm'$ a Levi factor and $\mcu'$ its unipotent radical, and a representation $\varrho'$ of $\mcm'$, such that $\varrho$ is an irreducible subrepresentation of the parabolic induction $i_{\mcp'}^{\mcm}(\varrho')$.  We may take an $E$-pure hereditary order $\mfa'$ in $A$ contained in $\mfa$ and the corresponding hereditary order $\mfb'=\mfa'\cap B$ in $B$, such that the image of $U(\mfb')$ (resp. $U^{1}(\mfb')$) in the quotient $\mcm=U(\mfb)/U^{1}(\mfb)$ is $\mcp'$ (resp. $\mcu'$). So $[\mfa',u,0,\beta]$ is also a strict simple stratum in $A$, and we write $J'^{1}=J^{1}(\beta,\mfa')$ and $J'=J(\beta,\mfa')$ for short.
		
		Let $\theta'\in\mcc(\mfa',0,\beta)$ be the transfer of $\theta$ and let $\eta'$ be the Heisenberg representation of $\theta'$. Let $\kappa'$ be the $\beta$-extension of $\eta'$ related to $\kappa$ and let $\wt{\kappa}'$ be its non-genuine pull-back to $\wt{J'}$. Let $\wt{\rho}'=\mrinf_{\wt{\mcm'}}^{\wt{U(\mfb')}}(\epsilon\cdot\,_{s}\varrho')$ be an irreducible representation of $\wt{U(\mfb')}$, which can be extended to representations of $\wt{J}'$,  $\wt{U(\mfb')}\wt{J^{1}}$ and $\wt{U(\mfb')}\wt{U^{1}(\mfa')}$ (still denoted by $\wt{\rho}'$), whose restrictions to $\,_{s}J'^{1}$, $\,_{s}J^{1}$ and $\,_{s}U^{1}(\mfa')$ respectively are trivial.
		
		The restriction $\wt{\kappa}\rest_{\wt{U(\mfb')}\wt{J^{1}}}$ is an irreducible representation of $\wt{U(\mfb')}\wt{J^{1}}$. Also, $\wt{\rho}'$ is an irreducible subrepresentation of the restriction $\wt{\rho}\rest_{\wt{U(\mfb')}\wt{J^{1}}}$, which is because by  Frobenius reciprocity $\varrho'$ is an irreducible subrepresentation of $\varrho\rest_{\mcp}$. Since $\wt{\pi}$ contains $\wt{\lambda}$, it contains $\wt{\kappa}\rest_{\wt{U(\mfb')}\wt{J^{1}}}\otimes\wt{\rho}'$ as well. Let $\wt{\lambda}'=\wt{\kappa}'\otimes\wt{\rho}'$ be a genuine irreducible representation of $\wt{J'}$. Then using \eqref{eqrelatedbetaext} we have
		\begin{align*}
			\mrind_{\wt{U(\mfb')}\wt{J^{1}}}^{\wt{U(\mfb')}\wt{U^{1}(\mfa')}}(\wt{\kappa}\rest_{\wt{U(\mfb')}\wt{J^{1}}}\otimes\wt{\rho}')&\cong\mrind_{\wt{U(\mfb')}\wt{J^{1}}}^{\wt{U(\mfb')}\wt{U^{1}(\mfa')}}(\wt{\kappa}\rest_{\wt{U(\mfb')}\wt{J^{1}}})\otimes\wt{\rho}'\\&\cong\mrind_{\wt{J'}}^{\wt{U(\mfb')}\wt{U^{1}(\mfa')}}(\wt{\kappa}')\otimes\wt{\rho}'\cong\mrind_{\wt{J'}}^{\wt{U(\mfb')}\wt{U^{1}(\mfa')}}\wt{\lambda}'.
		\end{align*}
		By Frobenius reciprocity $\wt{\pi}$ contains $\wt{\lambda}'$. In particular it contains $\,_{s}\theta'$, which contradicts the minimality of $\mfa$. So $\varrho$ is cuspidal. 
		
	\end{proof}
	
	From now on, we adopt the notation in \S \ref{subsectionhomotype}. More precisely, 
	\begin{itemize}
		
		\item Fix a containment of $E$-pure hereditary orders $\amin\subset\mfa\subset\amax$, such that $\bmin=B\cap\amin$ is a minimal hereditary order, and $\bmax=B\cap\amax$ is a maximal hereditary order in $B$. Let $\Lambdamin$ and $\Lambdamax$ be the corresponding lattice chains related to $\amin$ and $\amax$ respectively. Assume that $K$ contains $U(\amax)$.
		
		\item Fix an $E$-decomposition $V=\bigoplus_{i=1}^{t}V^{i}$ that conforms with $\Lambda$, $\Lambdamax$, $\Lambdamin$ and is properly subordinate to $\Lambda$. Let $P=MN$ be a corresponding parabolic subgroup with $M=\prod_{i=1}^{t}\mraut_{F}(V^{i})$. 
		
		\item Write $r_{i}=\mrdim_{F}(V^{i})$ and $m_{i}=\mrdim_{E}(V^{i}_{E})=\mrdim_{F}(V^{i})/d$ for each $i$. Then $m_{1}+\dots+m_{t}=m$ and $r_{1}+\dots+r_{t}=r$. 
		
		\item Fix a certain $E$-basis of $V_{E}$ to identify $B$ with $\mrm_{m}(E)$, such that $\bmax$ is identified with $\mrm_{m}(\mfo_{E})$,  and $\mfb$ is identified with the standard hereditary order in $B$ with respect to the composition $m_{1}+\dots+m_{t}=m$, and $\bmin$ is identified with the standard minimal hereditary order in $B$. 
		
		\item For $i=1,\dots,t$, let $A^{i}=\mrend_{F}(V^{i})$ (resp. $B^{i}=\mrend_{E}(V_{E}^{i})$) which is identified with a subalgebra of $A$ (resp. $B$) via the $i$-th block diagonal embedding. Let $\mfa^{i}=A^{i}\cap\mfa$ and $\mfb^{i}=B^{i}\cap \mfb$.
		
		\item Let $\mcg^{i}=\mrgl_{m_{i}}(\bs{l})\cong U(\mfb^{i})/U^{1}(\mfb^{i})$. Then
		$$\mcp\cong U(\mfb)/U^{1}(\bmax)\quad\text{and}\quad\mcm=\mcg^{1}\times\dots\times\mcg^{t}\cong U(\mfb)/U^{1}(\mfb)\cong J/J^{1}.$$
		
		\item Write $\varrho=\varrho_{1}\boxtimes\dots\boxtimes\varrho_{t}$, where each $\varrho_{i}$ is a cuspidal representation of $\mcg^{i}$ for $i=1,\dots,t$.
		
		\item Define $(\wt{J_{M}},\wt{\lambda}_{M})$ and $(\wt{J_{P}},\wt{\lambda}_{P})$ as in \S \ref{subsectionhomotype}.
		
	\end{itemize}
	
	Our next goal is to show that $(\wt{J},\wt{\lambda})$ is a twisted simple type. Notice that conjugating by an element in $W_{0}(B)\subset U(\bmin)$, we may interchange the position of $V^{i}$ and $V^{j}$ in the decomposition $V=V^{1}\oplus\dots\oplus V^{t}$ for $1\leq i< j\leq t$. The related  $(A^{i},B^{i},\mfa^{i},\mfb^{i},\mcg^{i},\varrho_{i})$ and $(A^{j},B^{j},\mfa^{j},\mfb^{j},\mcg^{j},\varrho_{j})$ are also interchanged. Then, without loss of generality, we may assume that there exist $1\leq t'\leq t$ and a cuspidal representation $\varrho_{0}$ of $\mrgl_{m_{1}}(\bs{l})$, such that
	\begin{itemize}
		\item  $m_{1}=\dots=m_{t'}$, and for each $i=1,\cdots,t'$ there exists $s_{i}\in\mbz$ such that $\varrho_{i}\cong\varrho_{0}\cdot (\chi_{\varpi_{E}}\circ\mrdet_{\bs{l}})^{s_{i}\bs{d}}$, where $\chi_{\varpi_{E}}=\epsilon((\varpi_{E},\cdot)_{n,E})$ is a character of $\bs{l}^{\times}\cong\mfo_{E}^{\times}/(1+\mfp_{E})$;
		\item For every $i=t'+1,\dots,t$, the representation $\varrho_{i}$ is not isomorphic to $\varrho_{0}\cdot (\chi_{\varpi_{E}}\circ\mrdet_{\bs{l}})^{s_{i}\bs{d}}$ for any $s_{i}\in\mbz$.
	\end{itemize}
	We want to show that $t=t'$. Otherwise, we consider $W=V^{1}\oplus\dots\oplus V^{t'}$ and $W'=V^{t'+1}\oplus\dots\oplus V^{t}$ as non-zero $F$-subspaces of $V$. Let $M'=\mraut_{F}(W)\times\mraut_{F}(W')$ be a Levi subgroup of $G$. Let $J_{M'}=J_{P}\cap M'$ and let $\wt{\lambda}_{M'}=\wt{\lambda}_{P}\rest_{\wt{J_{P}}\cap\wt{M'}}$ be a genuine irreducible representation of $\wt{J_{M'}}$.
	
	\begin{lemma}
		
		$(\wt{J_{P}},\wt{\lambda}_{P})$ is a covering pair of $(\wt{J_{M'}},\wt{\lambda}_{M'})$.
		
	\end{lemma}
	
	\begin{proof}
		
		We follow the proof of \cite{secherre2008representations}*{Proposition 5.17}. The condition (1)(2) of being a covering pair is easily verified as in \emph{loc. cit.}. We verify condition (4) by estimating $I_{G}(\wt{\lambda}_{P})$. Using $\mrind_{\wt{J_{P}}}^{\wt{J}}\wt{\lambda}_{P}\cong\wt{\lambda}$, Proposition \ref{propIglambda} and the Bruhat decomposition $B^{\times}=U(\bmin)W(B)U(\bmin)$, we have
		$$I_{G}(\wt{\lambda}_{P})=J_{P}I_{W(B)}(\wt{\rho})J_{P}.$$
		Using Proposition \ref{proprhoWBintertwine}, $w\in I_{W(B)}(\wt{\rho})$ implies that $w\in W(\mfb)$ and $w$ normalizes $\wt{\rho}\rest_{\wt{M^{0}(\mfb)}}$. Write $w=w_{0}h$ for $w_{0}\in W_{0}(\mfb)$ and $h\in T(\mfb)$. Then $w$ normalizes $\wt{\rho}\rest_{\wt{M^{0}(\mfb)}}$ if and only if 
		\begin{equation}\label{eqvarrhocondition}
			(\varrho_{1}\boxtimes\dots\boxtimes\varrho_{t})^{w_{0}}\cdot\chi_{h}\cong\varrho_{1}\boxtimes\dots\boxtimes\varrho_{t},
		\end{equation}
		where $\chi_{h}:=\epsilon([h,\cdot]_{\sim})$ is a character of $\mcm=M^{0}(\mfb)/M^{1}(\mfb)$. From our assumption on $V^{1},\dots,V^{t}$ and $\varrho_{1},\dots,\varrho_{t}$ and formula \eqref{eqchih}, the condition \eqref{eqvarrhocondition} happens only if $w_{0}\in M'$. As a result, we have $I_{G}(\wt{\lambda}_{P})\subset J_{P}M'J_{P}$, verifying the condition (4) of being a covering pair.
		
	\end{proof}
	
	Let $P'=M'N'$ be a parabolic subgroup of $G$ having a Levi factor $M'$ and the unipotent radical $N'$, such that $P'$ contains $P$. Using Theorem \ref{thmJacquetcomm}, we have 
	$$0\neq\mrhom_{\wt{J}}(\wt{\pi},\wt{\lambda})=\mrhom_{\wt{J_{P}}}(\wt{\pi},\wt{\lambda}_{P})=\mrhom_{\wt{J_{M'}}}(r_{N'}(\wt{\pi}),\wt{\lambda}_{M'}),$$
	contradicting the fact that $\wt{\pi}$ is cuspidal. 
	
	So we must have $t'=t$. We let $m_{0}=m_{1}=\dots=m_{t}$ and $r_{0}=r_{1}=\dots=r_{t}$. For each $i=1,\dots,t$ we write $\varrho_{i}\cong \varrho_{0}\cdot (\chi_{\varpi_{E}}\circ\mrdet_{\bs{l}})^{s_{i}\bs{d}}$ for a certain $s_{i}\in\mbz$. Let $$\varrho_{0}'=\varrho_{0}\cdot(\chi_{\varpi_{E}}\circ\mrdet_{\bs{l}})^{(s_{1}r_{1}+\dots+s_{t}r_{t})(2\bs{c}+\bs{d})}$$
	be a cuspidal representation of $\mrgl_{m_{0}}(\bs{l})$.
	Using formula \eqref{eqchih} for $g_{0}=\mrdiag(\varpi_{E}^{-s_{1}}I_{m_{1}},\dots\varpi_{E}^{-s_{t}}I_{m_{t}})$ we have $$\varrho=\varrho_{1}\boxtimes\dots\boxtimes\varrho_{t}\cong(\varrho_{0}'\boxtimes\dots\boxtimes\varrho_{0}')\cdot\chi_{g_{0}}.$$
	This implies that $(\wt{J},\wt{\lambda})$ is indeed a twisted simple type that is contained in $\wt{\pi}$.
	
	We have the following proposition, which is interesting in its own right.
	
	\begin{proposition}\label{proptwistnontwist}
		
		For a genuine irreducible representation $\wt{\pi}$ of $\wt{G}$, if it contains a twisted simple type, then it contains all the related weakly equivalent simple types as well.
		
	\end{proposition}
	
	\begin{proof}
		
		We follow the proof of \cite{bushnell129admissible}*{Proposition 8.3.4, 8.3.5} or \cite{secherre2008representations}*{Proposition 5.19}. Let $(\wt{J},\wt{\lambda})$ be a twisted simple type contained in $\wt{\pi}$ with the notation as above. We claim that
		\begin{itemize}
			\item Let $\Pi(\mfb)=\begin{pmatrix}  & I_{(t-1)m_{0}}  \\ \varpi_{E}I_{m_{0}}  & \end{pmatrix}\in B^{\times}\cong \mrgl_{m}(E)$. Then $\wt{\pi}$ contains $\wt{\lambda}^{\Pi(\mfb)}\cong\wt{\kappa}\otimes\wt{\rho}^{\Pi(\mfb)}$.
			\item If $\varrho_{i}\ncong\varrho_{i+1}$ for some $i$, then $\wt{\pi}$ contains $\wt{\lambda}^{\circ}:=\wt{\kappa}\otimes\wt{\rho}^{\circ}$, where $\wt{\rho}^{\circ}$ is the inflation of $$\varrho^{\circ}:=\varrho_{1}\boxtimes\dots\boxtimes\varrho_{i-1}\boxtimes\varrho_{i+1}\boxtimes\varrho_{i}\boxtimes\varrho_{i+2}\boxtimes\dots\boxtimes\varrho_{t}.$$
		\end{itemize}
		Since $\Pi(\mfb)$ and $W_{0}(\mfb)$ generate $W(\mfb)$, it is clear that we may use these two claims to finish the proof. 
		
		The first claim is direct. We only remark that the statement itself makes sense since $\Pi(\mfb)$ normalizes $\wt{\kappa}$. In particular, we have $\mrind_{\wt{J}}^{\wt{G}}(\wt{\lambda})\cong\mrind_{\wt{J}}^{\wt{G}}(\wt{\lambda}^{\Pi(\mfb)})$.
		
		For the second claim,  we consider the lattice chain $\Lambda'_{E}$ of $V_{E}$ of period $t-1$, such that $(\Lambda'_{E})_{j}=(\Lambda_{E})_{j}$ for $j=1,2,\dots, i-1$ and $(\Lambda'_{E})_{j}=(\Lambda_{E})_{j+1}$ for $j=i,\dots,t-1$. Then, it relates to a hereditary order $\mfa'$ in $A$ and the standard hereditary order $\mfb'$ in $B$ with respect to the composition $(m_{0},\dots,m_{0},2m_{0},m_{0},\dots,m_{0})$ of $m$, where $2m_{0}$ occurs in the $i$-th coordinate. In particular, $\mfb'$ contains $\mfb$. Let $J'=J(\beta,\mfa')$ and $J'^{1}=J^{1}(\beta,\mfa')$. Let $\mcm'= U(\mfb')/U^{1}(\mfb')\cong J'/J'^{1}$, which contains $\mcp= U(\mfb)/U^{1}(\mfb')$ as a parabolic subgroup with a Levi factor $\mcm=U(\mfb)/U^{1}(\mfb)$. 
		
		Let $\wt{\kappa}'$ be the $\beta$-extension of $\wt{J'}$ related to $\wt{\kappa}$, let $\varrho'=\mrInd_{\mcp}^{\mcm'}(\varrho)\cong\mrInd_{\mcp}^{\mcm'}(\varrho^{\circ})$ be the related parabolic induction, which is irreducible since $\varrho_{i}\ncong\varrho_{i+1}$, let $\wt{\rho}'=\mrinf_{\wt{\mcm'}}^{\wt{J'}}(\epsilon\cdot\,_{s}\varrho')$, and let $\wt{\lambda}'=\wt{\kappa}'\otimes\wt{\rho}'$ be an irreducible representation of $\wt{J'}$. 
		
		We may regard $\wt{\rho}'$ as a representation of $\wt{U(\mfb')}\wt{U^{1}(\mfa)}$ that is trivial on $\,_{s}U^{1}(\mfa)$, and $\wt{\rho}$ (resp. $\wt{\rho}^{\circ}$) as a representation of $\wt{U(\mfb)}\wt{U^{1}(\mfa)}$ that is trivial on $\,_{s}U^{1}(\mfa)$, then by definition we have $$\mrInd_{\wt{U(\mfb)}\wt{U^{1}(\mfa)}}^{\wt{U(\mfb')}\wt{U^{1}(\mfa)}}(\wt{\rho})\cong\wt{\rho}'\cong\mrInd_{\wt{U(\mfb)}\wt{U^{1}(\mfa)}}^{\wt{U(\mfb')}\wt{U^{1}(\mfa)}}(\wt{\rho}^{\circ}).$$ Combining with \eqref{eqrelatedbetaext}, we have $$\mrInd_{\wt{J}}^{\wt{U(\mfb')}\wt{U^{1}(\mfa)}}\wt{\lambda}\cong\mrInd_{\wt{U(\mfb)}\wt{J'^{1}}}^{\wt{U(\mfb')}\wt{U^{1}(\mfa)}}(\wt{\lambda}'\rest_{\wt{U(\mfb)}\wt{J'^{1}}})\cong\mrInd_{\wt{J}}^{\wt{U(\mfb')}\wt{U^{1}(\mfa)}}\wt{\lambda}^{\circ}.$$
		Thus by Frobenius reciprocity $\wt{\pi}$ contains $\wt{\lambda}$ if and only if it contains $\wt{\lambda}^{\circ}$, which finishes the proof of the second claim.
		
	\end{proof}
	
	Extracting from the above argument, we also have the following interesting corollary.
	
	\begin{corollary}
		
		For two weakly equivalent twisted simple types $(\wt{J},\wt{\lambda})$ and $(\wt{J},\wt{\lambda}')$ of $\wt{G}$, we have $\mrind_{\wt{J}}^{\wt{G}}\wt{\lambda}\cong\mrind_{\wt{J}}^{\wt{G}}\wt{\lambda}'$. As a result, we have an isomorphism of Hecke algebras $\mch(\wt{G},\wt{\lambda})\cong\mch(\wt{G},\wt{\lambda'})$.
		
	\end{corollary}
	
	Using Proposition \ref{proptwistnontwist}, we may assume that $(\wt{J},\wt{\lambda})$ is indeed a simple type that is contained in $\wt{\pi}$. We claim that $(\wt{J},\wt{\lambda})$ must be a maximal simple type. Otherwise, $P=MN$ is a proper parabolic subgroup of $G$. Using Theorem \ref{thmJacquetcomm} and Theorem \ref{thmJPcoveringpair}, we have
	$$0\neq\mrhom_{\wt{J}}(\wt{\pi},\wt{\lambda})=\mrhom_{\wt{J_{P}}}(\wt{\pi},\wt{\lambda}_{P})=\mrhom_{\wt{J_{M}}}(r_{N}(\wt{\pi}),\wt{\lambda}_{M}),$$
	contradicting the fact that $\wt{\pi}$ is cuspidal. So $(\wt{J},\wt{\lambda})$ is indeed a maximal simple type. 
	
	At last, the proof of Theorem \ref{thmcuspidalconstruction}.(2) is accomplished. 
	

\end{document}